\documentclass[reqno]{amsart} 

\def\printappendix{on}    

\usepackage{ifthen}
\newcommand{\ifthen}[2]{ \ifthenelse{#1}{#2}{} }

\newcommand{\inappendix}[2]{\ifthenelse{\equal{\printappendix}{on}}{#1}{#2}}

\usepackage[english]{babel}
\usepackage[latin1]{inputenc}

\usepackage[notref,notcite,color]{showkeys}

\usepackage{a4wide}

\usepackage{multirow}

\usepackage{amsmath,amssymb,bbm}
\usepackage{enumitem}

\usepackage{epsfig,psfrag,color}

\usepackage[overload]{textcase}

\usepackage[colorlinks,bookmarks,linkcolor=black,citecolor=black]{hyperref}

\let\eps\varepsilon
\let\epsilon\varepsilon
\let\rho\varrho
\let\phi\varphi

\let\subset\subseteq

\def\dcup{\dot\cup}

\def\subsc#1{\textsc{\MakeTextLowercase{#1}}} 
\def\isubsc#1{\text{\it\tiny #1}}

\def\itm#1{\rm ({#1})}
\def\itmit#1{\itm{\it #1\,}}

\def\irom{\itmit{\roman{*}}}

\def\cF{\hyper{F}}

\def\cS{\mathcal{S}}  
\def\cH{\mathcal{H}}
\def\cC{\mathcal{C}}
\def\cP{\mathcal{P}}  

\def\tand{\ \text{and}\ }
\def\qand{\quad\text{and}\quad}
\def\qqand{\qquad\text{and}\qquad}

\newtheorem{theorem}                   {Theorem}
\newtheorem{lemma}           [theorem] {Lemma}   
   
\newtheorem{proposition}     [theorem] {Proposition}  
\newtheorem{claim}           [theorem] {Claim}

\newtheorem{definition}      [theorem] {Definition} 
 
 \theoremstyle{remark} 
\newtheorem*{remark}{Remark} 

\newcommand{\comment}[1]{}

\newcommand{\sml}[1]{\scalebox{.7}{#1}}

\newcommand{\By}[2]{\overset{\mbox{\tiny{#1}}}{#2}} 
\newcommand{\ByRef}[2]{   \By{\eqref{#1}}{#2} }     

\newcommand{\leBy}[1]{    \By{#1}{\le} }
\newcommand{\geBy}[1]{    \By{#1}{\ge} }
\newcommand{\eqByRef}[1]{ \ByRef{#1}{=} }

\newcommand{\leByRef}[1]{ \ByRef{#1}{\le} }
\newcommand{\geByRef}[1]{ \ByRef{#1}{\ge} }

\newcommand{\subref}[2][L]{\text{\tiny\ref{#2}}}
\newcommand{\hyper}[1]{\mathcal{#1}}
\newcommand{\coN}{N^{\cap}}

\newcommand{\REALS}{\mathbb{R}}
\newcommand{\NATS}{\mathbb{N}}

\newcommand{\bigO}{\mathcal{O}}
\newcommand{\smallo}{o}

\newcommand{\ti}[1]{\widetilde{#1}} 

\newcommand{\Gnp}[1][n]{\ensuremath{\mathcal{G}_{#1,p}}}

\newcommand{\aas}{a.a.s.}

\newcommand{\bad}[4][G]{\operatorname{bad}^{#1,#2}_{#3,#4,p}}
\newcommand{\Bad}[4][G]{\operatorname{Bad}^{#1,#2}_{#3,#4,p}}
\newcommand{\BAD}[4][G]{\operatorname{bad}^{#1,#2}_{#3,#4,p}}

\newcommand{\ndist}[1][\ti{U}]{\operatorname{d}_{N(#1)}} 
\newcommand{\stars}[1][G]{\#\operatorname{stars}^{#1}}

\newcommand{\MON}{R} 
\newcommand{\LDR}{R^*}  

\newcommand{\STi}[1][i]{\rm ($\cS_{#1}$)}

\DeclareMathOperator{\Prob}{\mathbb{P}}
\DeclareMathOperator{\Exp}{\mathbb{E}}
\DeclareMathOperator{\Bin}{Bi}

\DeclareMathOperator{\bw}{bw} 

\DeclareMathOperator{\ldeg}{ldeg}
\DeclareMathOperator{\edeg}{edeg}
\DeclareMathOperator{\swdist}{d_{sw}} 

\date{\today}
\title[Almost spanning subgraphs of random graphs after edge removal] 
{Almost spanning subgraphs of random graphs after adversarial edge removal}

\author[Julia B\"ottcher]{Julia B\"ottcher} 
\address{Instituto de
  Matem\'atica e Estat\'{\i}stica, Universidade de S\~ao Paulo, Rua do
  Mat\~ao 1010, 05508--090~S\~ao Paulo, Brazil} 
\curraddr{Department of
  Mathematics, Columbia House, London School of Economics, Houghton Street,
  London WC2A 2AE, UK } \email{boettche@lse.ac.uk}

\author[Yoshiharu Kohayakawa]{Yoshiharu Kohayakawa}
\address{Instituto de Matem\'atica e Estat\'{\i}stica, Universidade de
  S\~ao Paulo, Rua do Mat\~ao 1010, 05508--090~S\~ao Paulo, Brazil}
\email{yoshi@ime.usp.br}

\author[Anusch Taraz]{Anusch Taraz} 
\address{Zentrum Mathematik, Technische
 Universit\"at M\"unchen, Boltzmannstra\ss{}e~3, D--85747 Garching bei
 M\"unchen, Germany} 
\email{taraz@ma.tum.de} 

\thanks{%
  The first and third author were partially supported by DFG grant TA
  309/2-1.
  The first author was partially supported by an EUBRANEX grant of the
  EU programme EM ECW.
  The second author was partially supported by CNPq (Proc.~308509/2007-2,
  485671/2007-7 and 486124/2007-0). The cooperation of the three authors was
  supported by a joint CAPES-DAAD project (415/ppp-probral/po/D08/11629,
  Proj.~no.~333/09). \\
  An extended abstract appeared in: V Latin-American Algorithms,
  Graphs and Optimization Symposium (LAGOS), Gramado, Brazil,
  Electronic Notes in Discrete Mathematics, 2009.
  The authors are
  grateful to NUMEC/USP, N\'ucleo de Modelagem Estoc\'astica e Complexidade
  of the University of S\~{a}o Paulo, and Project MaCLinC/USP, for supporting
  this research.
}

\dedicatory{Dedicated to Vojt\v{e}ch R\"odl on the occasion of his sixtieth birthday}

\begin{document}

\begin{abstract}
  Let~$\Delta\geq2$ be a fixed integer.  We show that the random graph
  $\Gnp$ with $p\gg (\log n/n)^{1/\Delta}$ is robust with respect to
  the containment of almost spanning bipartite graphs~$H$ with maximum
  degree $\Delta$ and sublinear bandwidth in the following sense:
  asymptotically almost surely, if an adversary deletes arbitrary
  edges from~$\Gnp$ in such a way that each vertex loses less than
  half of its neighbours, then the resulting graph still contains a
  copy of all such~$H$.
\end{abstract}

\maketitle


\section{Introduction and results}
\label{sec:intro}

In this paper we study graphs that are robust in the following sense: even
after adversarial removal of a specified proportion of their edges, they still
contain copies of every graph from a certain class of graphs.

In order to make this precise, we use the notion of \emph{resilience}
(see~\cite{SuVu}). Let~$\cP$ be a monotone increasing graph property
and~$G=(V,E)$ be a graph. The \emph{global resilience}~$R_g(G,\cP)$
of~$G$ with respect to~$\cP$ is the minimum~$r\in\REALS$ such that
deleting a suitable set of $r\cdot|E|$ edges from~$E$ creates a graph
which is not in $\cP$. The \emph{local resilience}~$R_\ell(G,\cP)$
of~$G$ with respect to~$\cP$ is the minimum~$r\in\REALS$ such that
deleting a suitable set of edges, respecting the restriction that at
most $r\cdot \deg_G(v)$ edges incident to $v$ should be removed for
every vertex $v\in V$, creates a graph which is not in $\cP$.

For example, using this terminology, the classical theorems of Tur\'an~\cite{Tur}
and Dirac~\cite{Dir} can be stated as follows: the global resilience of the
complete graph $K_n$ with respect to containing a clique on $r$ vertices is
$\frac{1}{r-1}-o(1)$, and the local resilience of $K_n$ with respect to
containing a Hamilton cycle is $\frac12-o(1)$. In this paper we stay quite close to the
scenario of these two examples insofar as we will also consider properties that
deal with subgraph containment. However, we are interested in the resilience of
graphs which are much sparser than the complete graph.

It turns out that the random graph $\Gnp$ is well suited for this purpose ($\Gnp$
is defined on vertex set $[n]=\{1,\dots,n\}$ and edges exist independently of
each other with probability $p$).  
Clearly, asymptotically almost surely (\aas)\ the local resilience of $\Gnp$
with respect to containing a Hamilton cycle (or in fact any connected graph
on more than, say, $\frac12n$ vertices) is at most $\frac12+o(1)$, since for
bigger values it is easy to disconnect the graph into
components of size at most $\frac12n$ by deleting edges respecting the
corresponding resilience definition.
Sudakov and Vu~\cite{SuVu} showed that indeed
\aas\ the local resilience of $\Gnp$ with respect
to containing a Hamilton cycle is $\frac12-o(1)$ if $p>\log^4n/n$.
A result of Dellamonica, Kohayakawa, Marciniszyn and Steger~\cite{DKMS}
implies that \aas\ the local resilience of $\Gnp$ with respect to
containing cycles of length at least $(1-\alpha)n$ is $\frac12-o(1)$ for
any $0< \alpha < \frac12$ and $p\gg 1/n$. We shall discuss the various
lower bounds for the edge probability $p$ occuring in these and later
results at the end of Section~\ref{sec:background}.

Recently Balogh, Csaba, and Samotij~\cite{BalCsaSam} studied the
local resilience of $\Gnp$ with respect to containing all
trees on $(1-\eta)n$ vertices with constant maximum degree~$\Delta$. They showed
that there is a constant $c=c(\Delta,\eta)$ such that for $p\ge c/n$ this
local resilience is also $\frac12-o(1)$ \aas

Now we extend the scope of investigations to the containment of a much
larger class of subgraphs.  A graph has \emph{bandwidth} at most~$b$
if there exists a labelling of the vertices by numbers $1,\dots,n$,
such that for every edge $ij$ of the graph we have $|i-j| \le b$. Let
$\cH(m,\Delta)$ denote the class of all graphs on $m$ vertices with
maximum degree at most $\Delta$, and $\cH_2^b(m,\Delta)$ denote the
class of all \emph{bipartite} graphs in $\cH(m,\Delta)$ which have
bandwidth at most $b$. Our result asserts that the local resilience of
$\Gnp$ with respect to containing all graphs~$H$ from $\cH_2^{\beta
  n}((1-\eta)n,\Delta)$ is $\frac12-o(1)$ for small $\beta$ and $\eta$
and for $p=p(n)=o(1)$ sufficiently large.

\begin{theorem}\label{thm:main}
  For each $\eta,\gamma>0$ and $\Delta\ge 2$ there exist positive
  constants~$\beta$ and~$c$ such that 
  the following holds for
  $p\ge c(\log n/n)^{1/\Delta}$. 
  Asymptotically almost surely every spanning subgraph
  $G=(V,E)$ of $\Gnp$ with $\deg_G(v)\ge(\frac{1}{2}+\gamma)\deg_{\Gnp}(v)$
  for all $v\in V$ contains a copy of every graph~$H$ in $\cH_2^{\beta
  n}((1-\eta)n,\Delta)$. 
\end{theorem}

We note that several important classes of graphs have sublinear bandwidth, and
hence Theorem~\ref{thm:main} does apply to them: this is the case for, e.g.,
the class of all bounded degree planar graphs (see~\cite{BPTW}).

As an application of this theorem we derive a result on rainbow
$H$-copies with $H\in\cH_2^{\beta n}((1-\eta)n,\Delta)$ for certain
edge-colourings of $K_n$ in Section~\ref{sec:bunt}. The proof of
Theorem~\ref{thm:main} is prepared in
Sections~\ref{sec:reg}--\ref{sec:joint} and presented in
Section~\ref{sec:proof}.  First, however, we will compare our result
to related results in the next section.


\section{Background}
\label{sec:background}

As we saw at the end of the last section, we are looking for graphs that not
only contain one specific subgraph but a large class of graphs. A graph~$G$ is
called \emph{universal} for a class of graphs $\cH$ if~$G$ contains a copy of every graph
from $\cH$ as a subgraph. In this section, we first briefly sketch some results
concerning universality in general and then come back to resilience with respect
to universality.

Dellamonica, Kohayakawa, R\"odl, and Ruci\'nski~\cite{DeKoRoRu} show
that $\Gnp$ is \aas\ universal for
$\cH(n,\Delta)$ for some~$p$ in~$\ti\bigO(n^{-1/2\Delta})$ (where $\ti\bigO$
hides polylogarithmic factors). It is also shown in~\cite{DeKoRoRu} that
the lower bound for the edge probability $p$ can be improved if we restrict
our attention to balanced bipartite graphs: Let $\cH_2(m,m,\Delta)$ denote
the class of bipartite graphs in $\cH(2m,\Delta)$ with two colour classes
of equal size. Then $\Gnp[2n]$ \aas\ is universal for $\cH_2(n,n,\Delta)$
for some~$p$ in $\ti\bigO(n^{-1/\Delta})$. The same lower bound for $p$ also
guarantees universality for \emph{almost spanning} graphs of arbitrary
chromatic number: Alon, Capalbo, Kohayakawa, R{\"o}dl, Ruci{\'n}ski and
Szemer\'edi~\cite{millennium} prove that for every $\eta>0$ and for some~$p$ in
$\ti\bigO(n^{-1/\Delta})$, the random graph $\Gnp$ \aas\ is universal for
$\cH((1-\eta)n,\Delta)$. 
Recently, Dellamonica, Kohayakawa, R\"odl, and
Ruci\'nski~\cite{DeKoRoRu12} generalised these results and obtained a
corresponding lower bound for spanning graphs: They have
shown that $\Gnp$ is \aas\ universal for $\cH(n,\Delta)$ for some~$p$
in~$\ti\bigO(n^{-1/2\Delta})$. 

Alon and Capalbo~\cite{AloCap,AloCapSODA} gave
explicit constructions of graphs with average degree
$\ti\bigO(n^{-2/\Delta})n$ that are universal for $\cH(n,\Delta)$.
For results concerning universal graphs for trees see, e.g.,~\cite{AKStrees}.

Moving on to resilience, it is clear that an adversary can destroy any spanning
subgraph by deleting the edges incident to a single vertex. Hence any graph must
have trivial global resilience with respect to universality for spanning
subgraphs.

However, if we focus on subgraphs of smaller order, then sparse random
graphs have a global resilience arbitrarily close to~$1$: 
Alon, Capalbo, Kohayakawa, R{\"o}dl, Ruci{\'n}ski and
Szemer\'edi~\cite{millennium}
show that for every $\gamma >0$ there is a
constant $\eta>0$ such that for some $p$ in $\ti\bigO(n^{-1/2\Delta})$
the random graph $\Gnp$ \aas\ has global resilience $1-\gamma$ with
respect to universality for $\cH_2(\eta n,\eta n,\Delta)$.  In other
words, $\Gnp$ contains \emph{many} copies of all graphs from
$\cH_2(\eta n,\eta n,\Delta)$ \emph{everywhere}.

Finally, the concept of local resilience allows for non-trivial
results concerning universality for almost spanning subgraphs.
For example, a conjecture of Bollob\'as and Koml\'os proven
in~\cite{BST09} asserts that the local resilience of the
complete graph $K_n$ with respect to universality for $\cH_r^{\beta
  n}(n,\Delta)$ is $\frac1r -o(1)$.  Here $\cH_r^{\beta n}(n,\Delta)$
is the class of all $r$-colourable $n$-vertex graphs with maximum
degree at most $\Delta$ and bandwidth at most $\beta n$, and one can
show that the bandwidth constraint cannot be omitted.
 
\begin{theorem}[\cite{BST09}]
\label{thm:bandwidth}
  For all $r,\Delta\in\mathbb{N}$ and $\gamma>0$, there exist constants $\beta>0$
  and $n_0\in\mathbb{N}$ such that for every $n\geq n_0$ the following holds.
  If~$H$ is an $r$-chromatic graph on~$n$ vertices with $\Delta(H) \leq \Delta$,
  and bandwidth at most $\beta n$ and if~$G$ is a graph on~$n$ vertices with
  minimum degree $\delta(G) \geq (\frac{r-1}{r}+\gamma)n$, then~$G$ contains a
  copy of~$H$.
\end{theorem}

Our Theorem~\ref{thm:main} replaces $K_n$ by the much sparser graph
$\Gnp$, but it only treats the case $r=2$ and 
almost spanning subgraphs.

Let us mention two more recent papers which continued this line of
research.  Huang, Lee and Sudakov considered almost spanning factors
for constant probability $p$ and showed that one cannot hope to obtain
spanning subgraphs, as~$\Omega(p^{-2})$ vertices may be forced to be
left out (see~\cite{HuaLeeSud}).  Also, Balogh, Lee and Samotij
considered the case of almost spanning triangle factors for $p\gg
(\log n/n)^{1/2}$ (see~\cite{balogh:_sp_corrad-_hajnal}).

Before we conclude this section, let us briefly discuss the lower
bounds for the edge probability $p$ mentioned in the results above,
summarized in Table~\ref{tb:results}.  First, a straightforward
counting argument shows that any graph that is universal for
$\cH(n,\Delta)$ must have at least $\Omega(n^{2-2/\Delta})$
edges. Moreover, it is easy to see that an edge probability
$p=n^{\eps-2/\Delta}$ with $\eps<\frac{1}{\Delta^2}$ is not sufficient
to guarantee that $\Gnp$ is universal even for the more restrictive
class $\cH_2(\eta n,\eta n,\Delta)$.  Indeed, consider the graph $H\in
\cH_2(\eta n,\eta n,\Delta)$ consisting of $\eta n /\Delta$ copies of
$K_{\Delta,\Delta}$. The expected number of copies of
$K_{\Delta,\Delta}$ in $\Gnp$ is at most
$$
n^{2\Delta} p^{\Delta^2} 
= n^{2\Delta} (n^{-\frac{2}{\Delta}+\eps})^{\Delta^2}  
= n^{2\Delta-2\Delta+\eps\Delta^2}
\ll n,
$$
and hence \aas\ $\Gnp$ does not contain a copy of $H$. 

{ 
\begin{table*}[ht]
   \begin{center}
    \begin{tabular}{|l|l|l|c|}
      \hline
      & \bf Result & \it\textbf{p} &  \bf Reference \\ \hline

      \multirow{1}{*}{Universality} &




      $\cH(n,\Delta) \subset\Gnp$ &
      $p = n^{-1/\Delta}$
      \vphantom{$\Big($} &
      \cite{DeKoRoRu12} \\ \hline

      \multirow{2}{*}{Resilience} &
      $R_g\big(\Gnp,\cH_2(\eta n,\eta n,\Delta)\big)\ge1-\gamma$ &
      $p = n^{-1/2\Delta}$
      \vphantom{$\Big($} &
      \cite{millennium} \\ \cline{2-4}

      & $R_\ell\big(\Gnp,\cH_2^{\beta n}((1-\eta)n,\Delta)\big)\ge\frac12-\gamma$ &
      $p = n^{-1/\Delta}$
      \vphantom{$\Big($} &
      Theorem~\ref{thm:main} \\ \hline

    \end{tabular}
  \end{center}
\caption{Summary of (best) known universality and resilience results 
(logarithmic factors for $p$ are omitted).}\label{tb:results}
 \end{table*}
}


\section{An application: rainbow copies of bipartite graphs}
\label{sec:bunt}

Let~$\phi$ be an arbitrary colouring of the edges of the complete graph~$K_n$. If~$\phi$
uses no colour more than~$k$ times then we say that~$\phi$ is \emph{$k$-bounded}.
Moreover, a copy of a graph~$H$ in~$K_n$ is a \emph{rainbow} copy if $\phi$ uses
no colour more than once on~$H$. If there is a rainbow copy of~$H$ in~$K_n$
then~$\phi$ is called \emph{$H$-rainbow}.

Erd\H{o}s, Ne\v{s}et\v{r}il, and R\"odl~\cite{ErdNesRoe} asked for
which $k=k(n)$ every $k$-bounded edge colouring of~$K_n$ 
has a rainbow Hamilton cycle.
Frieze and Reed~\cite{FriRee} showed that $k(n)$ can grow as fast as
$\kappa n/\log n$ for some constant~$\kappa$
(for early progress on this problem see the references in~\cite{FriRee}).
Albert, Frieze, and Reed~\cite{AlbFriRee} improved this bound to
$n/65$, which shows that $k$ can grow linearly, as was previously
conjectured by Hahn and Thomassen~\cite{HahTho}.

Here we consider the analogous question for $H$-rainbow
colourings with $H\in\cH_2^{\beta n}((1-\eta)n,\Delta)$.
As a consequence of our main theorem, Theorem~\ref{thm:main}, we
prove the following result.

\begin{theorem}\label{thm:bunt}
  For every $\eta>0$ and $\Delta\ge 2$ there exist positive
  constants~$\beta$ and~$\kappa$ such that for~$n$ sufficiently large,
  for every graph~$H\in\cH_2^{\beta n}((1-\eta)n,\Delta)$ and
  $k\le \kappa(n/\log n)^{1/\Delta}$, 
  every $k$-bounded edge-colouring of $K_n$ is $H$-rainbow.
\end{theorem}

For the proof of this theorem we apply the strategy of~\cite{FriRee}
and do the following for a given $k$-bounded edge colouring~$\phi$
of~$K_n$. We first take a random subgraph~$\Gamma=\Gnp$ of~$K_n$ and
then delete all edges in $\Gamma$ whose colour appears more than once
in~$\Gamma$. Denote the resulting graph by~$\Gamma(\phi)$. Any
subgraph of $\Gamma(\phi)$ is trivially rainbow and hence it
remains to show that there is a copy of~$H$ in $\Gamma(\phi)$ in order
to establish Theorem~\ref{thm:bunt}. In view of Theorem~\ref{thm:main}
it clearly suffices to prove the following lemma.

\begin{lemma}
\label{lem:bunt}
  Let~$p=p(n)$ and~$k=k(n)$ be such that
  $p\ge10^6\log n/n$ and $pk\le 10^{-3}$.
  For any $k$-bounded edge colouring~$\phi$ of $K_n$, 
  with probability $1-o(1)$ all vertices $v$ in $\Gamma=\Gnp$ satisfy
  $\deg_{\Gamma(\phi)}(v)\ge\frac23\deg_\Gamma(v)$.
\end{lemma}
\begin{proof}[Proof (sketch)]
  Let~$v$ be an arbitrary vertex of~$\Gamma$.  We classify the
  `deleted' edges incident to~$v$, that is, those edges in
  $E\big(v,N_\Gamma(v)\setminus N_{\Gamma(\phi)}(v)\big)$, into two
  sets: the set~$N_1$ of those edges whose colour appears only once
  in~$E(v,N_\Gamma(v))$ (but also somewhere else in~$\Gamma$) and the
  set~$N_2$ of those edges whose colour appears at least twice
  in~$E(v,N_\Gamma(v))$.
  With probability $1-\smallo(1/n)$ we have that
  $\deg_\Gamma(v)$ lies in the interval
  $[(1-\frac1{20})np,(1+\frac1{20})np]$ by a Chernoff bound. Therefore, showing
  \begin{enumerate}[label=\irom]
    \item\label{lem:bunt:N1} $\Prob(|N_1|\ge\frac1{10} np)=\smallo(1/n)$ and
    \item\label{lem:bunt:N2} $\Prob(|N_2|\ge\frac1{10} np)=\smallo(1/n)$
  \end{enumerate}
  and applying the union bound proves the lemma.
  
  For establishing~\ref{lem:bunt:N1} we expose the edges incident
  to~$v$ first, which enables us to determine $\deg_\Gamma(v)$. We
  have
  $\Prob\big(\deg_\Gamma(v)\ge\tfrac{21}{20}np\big)=\smallo(1/n)$. Subsequently
  we expose the remaining edges. Recall that for any edge $vw\in N_1$
  the colour $\phi(vw)$ appears somewhere else in~$\Gamma$, which
  happens with probability at most $p':=pk$.  Since these events are
  independent for different colours, we have
  \begin{equation*}
  \begin{split}
   \Prob(|N_1|\ge t)
     &\le\Prob\big(\deg_\Gamma(v)\ge\tfrac{21}{20}np\big)
       +\Prob\big(|N_1|\geq t\,\big|\deg_\Gamma(v)\le\tfrac{21}{20}np\big) \\
     &=\smallo(1/n)+\Prob(X\ge t)\,,
  \end{split}
  \end{equation*}
  where $X$ is a random variable with distribution $\Bin(n',p')$ where
  $n'=\deg_\Gamma(v)\le\frac{21}{20}np$. 
  Clearly $\Exp X\le\frac{21}{20}np\cdot pk\le\frac1{100}np$ and
  therefore~\ref{lem:bunt:N1} follows from an application of
  a Chernoff bound, since $np\ge10^6\log n$.

  For establishing~\ref{lem:bunt:N2} consider the random variable~$Y$ that counts
  edges in $E(v,N_\Gamma(v))$ whose colour appears only once in
  $E(v,N_\Gamma(v))$. Then $|N_2|=\deg_\Gamma(v)-Y$ and so it suffices to
  show that $\Prob(Y\le\frac{19}{20}np)=\smallo(1/n)$, using again
  that $\deg_{\Gamma}(v)>\frac{21}{20}np$ happens with probability
  $\smallo(1/n)$.  To see this, assume that $1,\dots,\ell$ are the
  colours that appear on the edges of $K_n$ containing~$v$, and let
  $k_i$ be the number of such edges with colour $i\in[\ell]$. Then
  $Y=\sum_{i\in[\ell]} Y_i$ where $Y_i$ is the indicator variable for
  the event that $E(v,N_\Gamma(v))$ contains exactly one edge of
  colour~$i$. Observe that the $Y_i$ are independent random variables
  and that $\Prob(Y_i=1)=k_ip(1-p)^{k_i-1}$.  In addition
  $1\ge(1-p)^{k_i-1}\ge(1-p)^k\ge\exp(-\frac{p}{1-p}k)\ge\frac{100}{101}$
  and hence
  \begin{equation*}
     \Exp Y = \sum_{i\in[\ell]} k_ip(1-p)^{k_i-1}\le np
     \qand
     \Exp Y \ge \tfrac{100}{101} (n-1)p \ge \tfrac{99}{100} np\,.
  \end{equation*}
  We conclude $\Prob(Y\le\frac{19}{20}np)=\smallo(1/n)$ from
  $\Prob(Y\le\Exp Y-t)\le\exp(-\frac1{2}t^2/\Exp Y)$
  (see~\cite[Theorem~2.10]{purpleBook}) by setting $t:=\frac1{100}np$
  and using $np\ge 10^6\log n$.
\end{proof}

As mentioned earlier, the bound on~$k(n)$ established in~\cite{FriRee}
for rainbow Hamilton cycles is not best possible.  As it turns
out, the bound on~$k$ in Theorem~\ref{thm:bunt} above can be improved
as well.  Indeed, such an improvement has recently been established
in~\cite{boettcher10+:_proper_colour_rainb}, where Lov\'asz's local
lemma is used.  However, we observe that the method of proof above is
more robust in the sense that one can, for instance, prove that
suitably bounded colourings of \textit{sparse random graphs} are
$H$-rainbow---something that does not seem to be within reach of
the method of proof in~\cite{boettcher10+:_proper_colour_rainb} (we
omit the details).


\section{Sparse regularity}
\label{sec:reg}

In this section we will introduce one of the main tools for our proof,
a sparse version of the regularity lemma developed by R\"odl and one
of the current authors (see~\cite{Ko97,KohRod_sparseRL}).  Before
stating this lemma we introduce the necessary definitions.


Let $G=(V,E)$ be a graph, and suppose~$p\in(0,1]$ and $\eps>0$ are
reals.  For disjoint nonempty sets $U,W\subset V$ the
\emph{$p$-density} of the pair $(U,W)$ is defined as
$d_{G,p}(U,W):=e_G(U,W)/(p|U||W|)$.  The pair $(U,W)$ is
\emph{$(\eps,p)$-regular} if $|d_{G,p}(U',W')-d_{G,p}(U,W)|\leq\eps$
for all $U'\subset U$ and $W'\subset W$ with $|U'|\ge\eps|U|$ and
$|W'|\ge\eps|W|$.

An \emph{$(\eps,p)$-regular partition} of~$G=(V,E)$ is an
\emph{$\epsilon$-equipartition} $V_0\dcup V_1 \dcup\dots\dcup V_r$ of
$V$, that is, with $|V_0|\le\eps|V|$ and $|V_1|=\dots=|V_r|$, such
that $(V_i,V_j)$ is an $(\eps,p)$-regular pair in $G$ for all but at
most $\eps \binom{r}2$ pairs $ij\in\binom{[r]}2$.  The partition
classes $V_i$ with $i\in[r]$ are called the \emph{clusters} of the
partition and $V_0$ is the \emph{exceptional set}.

The sparse regularity lemma asserts the existence of
$(\epsilon,p)$-regular partitions for sparse graphs~$G$ without `dense
spots'.  To quantify this latter property we need the following
notion.  Let~$\eta>0$ and $K>1$ be real numbers.  We say that
$G=(V,E)$ is \emph{$(\eta,K)$-bounded with respect to $p$} if for all
disjoint sets $X,Y\subset V$ with $|X|,|Y|\ge\eta|V|$ we have
$e_G(X,Y)\le Kp|X||Y|$.

\begin{lemma}[sparse regularity lemma] \label{lem:sparse-RL} For each
  $\eps>0$, $K>1$, and $r_0\ge 1$ there are constants $r_1$,
  $\nu$, and $n_0$ such that for any $p\in(0,1]$ the following
  holds. Any graph $G=(V,E)$ which has at least $n_0$ vertices and is
  $(\nu,K)$-bounded with respect to $p$ admits an $(\eps,p)$-regular
  $\epsilon$-equipartition with $r$ clusters, for some $r_0\le r\le
  r_1$.  \qed
\end{lemma}

As it turns out, we shall only make use of what one could call
`one-sided regularity'.  We call a pair $(U,W)$
\emph{$(\eps,d,p)$-dense} if $d_{G,p}(U',W')\ge d-\eps$ for all
$U'\subset U$ and $W'\subset W$ with $|U'|\ge\eps|U|$ and
$|W'|\ge\eps|W|$.  Clearly, an $(\epsilon,p)$-regular pair~$(U,W)$ is
$(\epsilon,d,p)$-dense for~$d=d_{G,p}(U,W)$.  Occasionally, in
informal discussions, when the particular value of~$d$ or~$\epsilon$
is not immediately relevant, we say that an $(\epsilon,p)$-regular
pair~$(U,W)$ is \emph{$(\eps,p)$-dense} or \emph{$p$-dense}.

An $\epsilon$-equipartition $V_0\dcup V_1 \dcup\dots\dcup V_r$ of a
graph~$G=(V,E)$ is an \emph{$(\eps,d,p)$-dense partition} with
\emph{reduced graph} $R$ if~$V(R)=[r]$ and the pair $(V_i,V_j)$ is
$(\eps,d,p)$-dense in $G$ whenever $ij\in E(R)$.  Note that, given an
$(\epsilon,p)$-regular partition as in Lemma~\ref{lem:sparse-RL} and a
real number~$d$, one has an $(\eps,d,p)$-dense partition of~$G$ with
the reduced graph~$R$, with $ij\in E(R)$ if and only if
only~$(V_i,V_j)$ is $(\epsilon,p)$-regular and~$d_{G,p}(V_i,V_j)\geq
d$.

It follows directly from the definition that sub-pairs of $p$-dense pairs again
form $p$-dense pairs.

\begin{proposition}
\label{prop:subpairs}
Let $(X,Y)$ be $(\eps,d,p)$-dense and suppose $X'\subset X$ satisfies
$|X'|\geq\mu|X|$. Then $(X',Y)$ is $(\frac{\eps}{\mu},d,p)$-dense.
\qed
\end{proposition}

In addition, neighbourhoods of most vertices in a $p$-dense pair are
not much smaller than expected. Again, this is a direct consequence of
the definition of $p$-dense pairs.

\begin{proposition}
  \label{prop:typical}
  Let $(X,Y)$ be $(\eps,d,p)$-dense. Then less than $\eps|X|$ vertices
  $x\in X$ are such that $|N_Y(x)|<(d-\eps)p|Y|$.  \qed
\end{proposition}

Some properties of the graph~$G$ translate to certain properties of the reduced
graph~$R$ of the partition constructed by the sparse regularity lemma.
For example the following well known consequence of Lemma~\ref{lem:sparse-RL}
is a minimum degree version of the sparse regularity lemma. For a proof see 
\inappendix{Section~\ref{sec:aux:reduced}}{the appendix of~\cite{BoeKohTar_sparsebip_arxiv}}.

\begin{lemma}[sparse regularity lemma, minimum degree version for $\Gnp$]
\label{lem:reduced} 
  For all $\alpha\in[0,1]$, $\eps>0$, and every integer $r_0$, there is an
  integer $r_1\ge 1$ such that for all $d\in[0,1]$ the following holds \aas\ 
  for $\Gamma=\Gnp$ if $\log^4n/(pn)=\smallo(1)$.
  Let $G=(V,E)$ be a spanning subgraph of $\Gamma$ with
  $\deg_G(v)\ge\alpha\deg_\Gamma(v)$ for all $v\in V$.
  Then there is an $(\eps,d,p)$-dense partition of $G$ with reduced graph $R$ of
  minimum degree $\delta(R)\ge(\alpha-d-\eps)|V(R)|$ with
  $r_0\le|V(R)|\le r_1$.
\end{lemma}

We remark that 
we do observe ``more'' than a mere inheritance of properties here: the graph~$G$
we started with is \emph{sparse}, but the reduced graph~$R$ we obtain 
in Lemma~\ref{lem:reduced} is
\emph{dense}. This will enables us to apply results obtained for dense graphs
to the reduced graph~$R$, and hence use such dense results to draw conclusions about
sparse graphs. 


\section{Main Lemmas}
\label{sec:idea}

In this section we will formulate the main lemmas 
and outline how they will be combined in Section~\ref{sec:proof} to give the 
proof of Theorem~\ref{thm:main}. 
For this we first
need to define two (families of) special graphs.

\label{def:spin}
For $r,t\in\NATS$, let $U=\{u_1,\dots,u_r\}$, $V=\{v_1,\dots,v_r\}$,
$C=\{c_{i,j},c'_{i,j}\colon i\in[r],j\in[2t]\}$, and
$B=\{b_{i,j},b'_{i,j}\colon i\in[r],j\in[2t]\}$. Let the \emph{ladder}~$\LDR_{r}$ be
the graph with vertex set~$U\dcup V$
and edge set $E(\LDR_r):=\{u_iv_j \colon  i,j\in[r],|i-j|\le 1\}$.
Let the \emph{spin graph}~$\MON_{r,t}$ be the graph with vertex set $U\dcup V\dcup C\dcup
B$ and the following edge set (see
Figure~\ref{fig:backbone}):
\begin{multline*}
  E(\MON_{r,t}):=
  \bigcup_{\substack{ i,i'\in[r], i'\neq 1\\ j,j'\in[2t]\\ k,k'\in[t]\\
    \ell,\ell'\in[t+1,2t]}}
  \Bigg( 
    \Big\{ u_iv_i,\, b_{i,k}b'_{i,k'},\, b_{i,\ell}b'_{i,\ell'},\,
      c_{i,k}c'_{i,k'},\,c_{i,\ell}c'_{i,\ell'} \Big\} 
    \cup \Big\{ b_{i,j}v_i,\, c_{i,j}v_i \Big\}
    \\[-12mm]
    \cup \Big\{b'_{i,k}b'_{i,\ell},\, c_{i'-1,\ell}c'_{i',k},\,
      c'_{i'-1,\ell}c_{i',k} \Big\} 
  \Bigg).
\end{multline*}

\begin{figure}
  \begin{center}
    \psfrag{R}{\scalebox{1.2}{$\LDR_{r}$}}
    \psfrag{RR}{\scalebox{1.2}{$\MON_{r,4}$}}
    \psfrag{u1}{$u_{i-1}$}
    \psfrag{v1}{$v_{i-1}$}
    \psfrag{u2}{$u_{i}$}
    \psfrag{v2}{$v_{i}$}
    \psfrag{u3}{$u_{i+1}$}
    \psfrag{v3}{$v_{i+1}$}
    \psfrag{c11}{\scalebox{0.7}{$c_{i-1,1}$}}
    \psfrag{c12}{\scalebox{0.7}{$c_{i-1,2}$}}
    \psfrag{c13}{\scalebox{0.7}{$c_{i-1,3}$}}
    \psfrag{c14}{\scalebox{0.7}{$c_{i-1,4}$}}
    \psfrag{c'11}{\scalebox{0.7}{$c'_{i-1,1}$}}
    \psfrag{c'12}{\scalebox{0.7}{$c'_{i-1,2}$}}
    \psfrag{c'13}{\scalebox{0.7}{$c'_{i-1,3}$}}
    \psfrag{c'14}{\scalebox{0.7}{$c'_{i-1,4}$}}
    \psfrag{c21}{\scalebox{0.7}{$c_{i,1}$}}
    \psfrag{c22}{\scalebox{0.7}{$c_{i,2}$}}
    \psfrag{c23}{\scalebox{0.7}{$c_{i,3}$}}
    \psfrag{c24}{\scalebox{0.7}{$c_{i,4}$}}
    \psfrag{c'21}{\scalebox{0.7}{$c'_{i,1}$}}
    \psfrag{c'22}{\scalebox{0.7}{$c'_{i,2}$}}
    \psfrag{c'23}{\scalebox{0.7}{$c'_{i,3}$}}
    \psfrag{c'24}{\scalebox{0.7}{$c'_{i,4}$}}
    \psfrag{b11}{\scalebox{0.7}{$b_{i-1,1}$}}
    \psfrag{b12}{\scalebox{0.7}{$b_{i-1,2}$}}
    \psfrag{b13}{\scalebox{0.7}{$b_{i-1,3}$}}
    \psfrag{b14}{\scalebox{0.7}{$b_{i-1,4}$}}
    \psfrag{b'11}{\scalebox{0.7}{$b'_{i-1,1}$}}
    \psfrag{b'12}{\scalebox{0.7}{$b'_{i-1,2}$}}
    \psfrag{b'13}{\scalebox{0.7}{$b'_{i-1,3}$}}
    \psfrag{b'14}{\scalebox{0.7}{$b'_{i-1,4}$}}
    \psfrag{b21}{\scalebox{0.7}{$b_{i,1}$}}
    \psfrag{b22}{\scalebox{0.7}{$b_{i,2}$}}
    \psfrag{b23}{\scalebox{0.7}{$b_{i,3}$}}
    \psfrag{b24}{\scalebox{0.7}{$b_{i,4}$}}
    \psfrag{b'21}{\scalebox{0.7}{$b'_{i,1}$}}
    \psfrag{b'22}{\scalebox{0.7}{$b'_{i,2}$}}
    \psfrag{b'23}{\scalebox{0.7}{$b'_{i,3}$}}
    \psfrag{b'24}{\scalebox{0.7}{$b'_{i,4}$}}
    \psfrag{c31}{\scalebox{0.7}{$c_{i+1,1}$}}
    \psfrag{c32}{\scalebox{0.7}{$c_{i+1,2}$}}
    \psfrag{c'31}{\scalebox{0.7}{$c'_{i+1,1}$}}
    \psfrag{c'32}{\scalebox{0.7}{$c'_{i+1,2}$}}
    \psfrag{b31}{\scalebox{0.7}{$b_{i+1,1}$}}
    \psfrag{b32}{\scalebox{0.7}{$b_{i+1,2}$}}
    \psfrag{b33}{\scalebox{0.7}{$b_{i+1,3}$}}
    \psfrag{b34}{\scalebox{0.7}{$b_{i+1,4}$}}
    \psfrag{b'31}{\scalebox{0.7}{$b'_{i+1,1}$}}
    \psfrag{b'32}{\scalebox{0.7}{$b'_{i+1,2}$}}
    \psfrag{b'33}{\scalebox{0.7}{$b'_{i+1,3}$}}
    \psfrag{b'34}{\scalebox{0.7}{$b'_{i+1,4}$}}
    \includegraphics[scale=1.0]{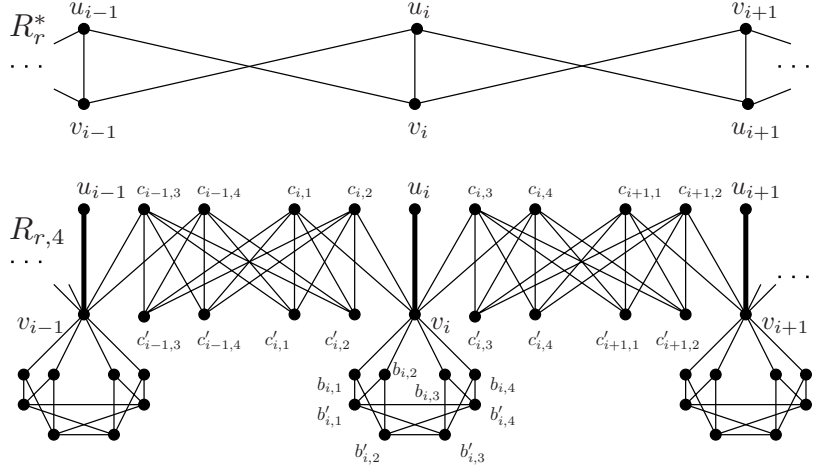}
  \end{center}
  \caption{The ladder $\LDR_r$ and the spin graph $\MON_{r,t}$ for the
  special case $t=2$.}
\label{fig:backbone}
\end{figure}

Now we can state our four main lemmas, two partition lemmas and two
embedding lemmas.  We start with the lemma for~$G$, which constructs a
partition of the host graph~$G$. This lemma is a consequence of the sparse
regularity lemma (Lemma~\ref{lem:reduced}) and asserts the existence of a
$p$-dense partition of $G$ such that its reduced graph contains a spin
graph.  We will indicate below why this is useful for the embedding of~$H$.
The lemma for~$G$ produces clusters of very different sizes: A set of
larger clusters $U_{i}$ and $V_i$ which we call \emph{big clusters} and
which will accommodate most of the vertices of~$H$ later, and a set of
smaller clusters $B_{i,j}$,$B'_{i,j}$, $C_{i,j}$, and $C'_{i,j}$.  The
$B_{i,j}$ and $B'_{i,j}$ are called \emph{balancing clusters} and the
$C_{i,j}$ and $C'_{i,j}$ \emph{connecting clusters}. They will be used to
host a small number of vertices of~$H$. These vertices balance and connect
the pieces of~$H$ that are embedded into the big clusters. The proof of
Lemma~\ref{lem:G} is given in Section~\ref{sec:G}.


\begin{lemma}[Lemma for $G$] \label{lem:G}
  For all integers $t,r_0>0$ and reals $\eta_\isubsc{G},\gamma>0$ there are positive reals
  $\eta'_\isubsc{G}$ and $d$ such that for all $\eps>0$ there is $r_1$
  such that the following holds \aas\ for $\Gamma=\Gnp$ with
  $\log^4n/(pn)=\smallo(1)$. Let $G=(V,E)$ be a spanning subgraph of $\Gamma$
  with $\deg_G(v)\ge(\frac{1}{2}+\gamma)\deg_\Gamma(v)$ for all $v\in V$. Then
  there is $r_0\le r\le r_1$, a subset $V_0$ of $V$ with $|V_0|\le\eps n$, and a
  mapping $g$ from $V\setminus V_0$ to the 
  spin graph $\MON_{r,t}$ such that for every $i\in[r],j\in[2t]$ we have
  \begin{enumerate}[label={\rm (G\arabic{*})}]
    \item \label{lem:G:Vi}
      $|U_{i}|,|V_{i}|\ge (1-\eta_\isubsc{G})\frac{n}{2r}$  
      \quad for $U_{i}:=g^{-1}(u_i)$ and 
      $V_{i}:=g^{-1}(v_i)$,
    \item \label{lem:G:Ci}
      $|C_{i,j}|,|C'_{i,j}|,|B_{i,j}|,|B'_{i,j}|\ge \eta'_\isubsc{G}\frac{n}{2r}$ \\ 
      for
      $C_{i,j}:=g^{-1}(c_{i,j})$,
      $C'_{i,j}:=g^{-1}(c'_{i,j})$,
      $B_{i,j}:=g^{-1}(b_{i,j})$, and
      $B'_{i,j}:=g^{-1}(b'_{i,j})$,
    \item \label{lem:G:reg} the pair $\big(g^{-1}(x),g^{-1}(y)\big)$ is
      $(\eps,d,p)$-dense for all $xy\in E(\MON_{r,t})$.
  \end{enumerate}
\end{lemma}
Since the dependencies of the constants appearing in this lemma are quite
involved, we remark that their quantification is as follows:
\begin{equation*}
  \forall\, t,\,r_0,\,\eta_\isubsc{G},\,\gamma \quad
  \exists\, \eta'_\isubsc{G},\,d \quad
  \forall\, \eps \quad
  \exists\, r_1\,.
\end{equation*}

Our second lemma provides a partition of~$H$ 
that fits the structure of the partition of $G$ generated by
Lemma~\ref{lem:G}. We will first state this lemma and then explain the
different properties which it guarantees. A set $S$ of vertices in a
graph~$H$ is called \emph{$\ell$-independent} for an integer $\ell$ if each pair
of distinct vertices in~$S$ has distance at least $\ell+1$ in~$H$.

\begin{lemma}[Lemma for $H$]
\label{lem:H}
  For all integers $\Delta$ there is an integer $t>0$ such
  that for any $\eta_\isubsc{H}>0$ 
  and any integer $r\ge 1$ 
  there is $\beta>0$ such that 
  the following holds for all integers $m$ and all bipartite graphs $H$ on
  $m$ vertices with $\Delta(H)\le\Delta$ and $\bw(H)\le\beta m$. There is a
  homomorphism $h$ from $H$ to the spin graph $\MON_{r,t}$
  such that for every $i\in[r],j\in[2t]$
 \begin{enumerate}[label={\rm (H\arabic{*})}]
    \item \label{lem:H:Vi}
      $|\ti{U}_{i}|,|\ti{V}_{i}|\le(1+\eta_\isubsc{H})\frac{m}{2r}$ \quad
     for $\ti{U}_{i}:=h^{-1}(u_i)$ and $\ti{V}_{i}:=h^{-1}(v_i)$,
    \item \label{lem:H:Ci}
      $|\ti{C}_{i,j}|,|\ti{C}'_{i,j}|,|\ti{B}_{i,j}|,|\ti{B}'_{i,j}|\le\eta_\isubsc{H}\frac{m}{2r}$ \\
      for 
      $\ti{C}_{i,j}:=h^{-1}(c_{i,j})$, 
      $\ti{C}'_{i,j}:=h^{-1}(c'_{i,j})$,
      $\ti{B}_{i,j}:=h^{-1}(b_{i,k})$, 
      and
      $\ti{B}'_{i,j}:=h^{-1}(b'_{i,k})$,
    \item \label{lem:H:Cindep} $\ti{C}_{i,j}$, $\ti{C}'_{i,j}$, $\ti{B}_{i,j}$, 
      and $\ti{B}'_{i,j}$ are $3$-independent in $H$,
    \item \label{lem:H:deg}
      $\deg_{\ti{V}_{i}}(y)=\deg_{\ti{V}_{i}}(y')\le\Delta-1$ 
      for all $yy'\in\binom{\ti{C}_{i,j}}{2}\cup\binom{\ti{B}_{i,j}}{2}$, \\
      $\deg_{\ti{C}_i}(y)=\deg_{\ti{C}_i}(y')$ for all $y,y'\in\ti{C}'_{i,j}$, \\ 
      $\deg_{L(i,j)}(y)=\deg_{L(i,j)}(y')$ for all $y,y'\in\ti{B}'_{i,j}$,
  \end{enumerate}
  where $\ti{C}_i:=\bigcup_{k\in[2t]}\ti{C}_{i,k}$ and
  $L(i,j):=\bigcup_{k\in[2t]}\ti{B}_{i,k}\cup\bigcup_{k<j}\ti{B}'_{i,k}$. 
  Further, let $\ti{X}_i$ with $i\in[r]$ be the set of vertices in $\ti{V}_i$ 
  with neighbours outside $\ti{U}_i$. Then
  \begin{enumerate}[label={\rm (H\arabic{*})},resume]
    \item \label{lem:H:X} $|\ti{X}_{i}|\le \eta_\isubsc{H} |\ti{V}_i|$.
  \end{enumerate} 
\end{lemma}
The quantification of the constants appearing in this lemma is as follows:
\begin{equation*}
  \forall\, \Delta\quad
  \exists\, t \quad
  \forall\, \eta_\isubsc{H},\, r\quad
  \exists\, \beta\,.
\end{equation*}

This lemma asserts the existence of a homomorphism~$h$ from~$H$ to a spin
graph~$\MON_{r,t}$.  Recall that~$\MON_{r,t}$ is contained in the reduced
graph of the $p$-dense partition provided by Lemma~\ref{lem:G}.  As we will
see, we can fix the parameters in this lemma such that, when we apply it
together with Lemma~\ref{lem:G}, the homomorphism~$h$ has the following
additional property.  The number $\ti L$ of vertices that it maps to a
vertex $a$ of the spin graph is less than the number $L$ contained in the
corresponding cluster $A$ provided by Lemma~\ref{lem:G}
(compare~\ref{lem:G:Vi} and~\ref{lem:G:Ci} with~\ref{lem:H:Vi}
and~\ref{lem:H:Ci} and recall that~$m$ is slightly smaller than~$n$). If
$A$ is a big cluster, then the numbers $L$ and $\ti L$ differ only slightly
(these vertices will be embedded using the constrained blow-up lemma), but
for balancing and connecting clusters~$A$ the number~$\ti L$ is much
smaller than~$L$ (this is necessary for the embedding of these vertices
using the connection lemma).  With property~\ref{lem:H:X} Lemma~\ref{lem:H}
further guarantees that only few edges of~$H$ are not assigned either to
two connecting or balancing clusters, or to two big clusters. This is
helpful because it implies that we do not have to take care of ``too many
dependencies'' between the applications of the blow-up lemma and the
connection lemma. The remaining
properties~\ref{lem:H:Cindep}--\ref{lem:H:deg} of Lemma~\ref{lem:H} are
technical but required for the application of the connection lemma (see
conditions~\ref{lem:CL:indep} and~\ref{lem:CL:deg} of Lemma~\ref{lem:CL}).

The vertices in $\ti{C}_{i,j}$ and $\ti{C}'_{i,j}$ are also called
\emph{connecting vertices} of $H$, the vertices in
$\ti{B}_{i,j}$ and $\ti{B}'_{i,j}$ \emph{balancing
vertices}.

We next describe the two embedding lemmas, the constrained blow-up lemma
(Lemma~\ref{lem:blowup}) and the connection lemma (Lemma~\ref{lem:CL}),
which we would like to use on the partitions of $G$ and $H$ provided by
Lemmas~\ref{lem:G} and~\ref{lem:H}. The connecting lemma will be
used to embed the connecting and balancing vertices into the connecting and balancing clusters
\emph{after} all the other vertices are embedded into the big clusters with the
help of the constrained blow-up lemma.

The constrained blow-up lemma states that bipartite graphs~$H$ with bounded
maximum degree can be embedded into a $p$-dense pair $G=(U,V)$ whose cluster
sizes are just slightly bigger than the partition classes of~$H$. This lemma further
guarantees the following. If we specify 
a small family of small special sets 
in one of the partition classes of~$H$
and a small family of small forbidden sets
in the corresponding cluster of~$G$, then no special set is mapped to a
forbidden set. 

The existence of
these forbidden sets is in fact a main difference to the classical blow-up lemma which is used in the dense setting,
where a small family of special vertices of~$H$ can be guaranteed to be mapped to a 
\emph{required set} of linear size in~$G$. This is very useful in a dense graph, because its
neighbourhoods (into which we would like to embed neighbours of already embedded vertices) 
are of linear size. 
In contrast, 
the property of having forbidden sets will be crucial for the sparse setting when we will 
apply this lemma
together with the connection lemma in the proof of Theorem~\ref{thm:main} in
order to handle the ``dependencies'' between these applications.
The proof of this lemma is given in Section~\ref{sec:blowup} and relies
on techniques developed in~\cite{millennium}.

\begin{lemma}[Constrained blow-up lemma] \label{lem:blowup}
  For every integer $\Delta>1$ and for all positive reals~$d$, 
  and~$\eta$ there    
  exist positive constants $\eps$ and $\mu$ 
  such that for all positive integers $r_1$ there is $c$ such that for 
  all integers $1\le r\le r_1$ the following holds \aas\ for $\Gamma=\Gnp$ with
  $p\ge c(\log n/n)^{1/\Delta}$. Let $G=(U,V)\subset\Gamma$ be an
  $(\eps,d,p)$-dense pair with $|U|,|V|\ge n/r$ and let $H$ be a bipartite graph
  on vertex classes $\ti{U}\dcup\ti{V}$ of sizes
  $|\ti{U}|,|\ti{V}|\le(1-\eta)n/r$ and with $\Delta(H)\le\Delta$.
  Moreover, suppose that there is 
  a family $\hyper{H}\subset\binom{\ti{V}}{\Delta}$ of
  \emph{special $\Delta$-sets} in $\ti{V}$ such that
  each $\ti{v}\in\ti{V}$ is contained in at most $\Delta$ special sets and a family
  $\hyper{B}\subset\binom{V}{\Delta}$ of
  \emph{forbidden $\Delta$-sets} in $V$ with
  $|\hyper{B}|\le\mu |V|^\Delta$.
  Then there is an embedding of $H$ into $G$ such that
  no special set is mapped to a forbidden set.
\end{lemma}
The quantification of the constants appearing in this lemma is as follows:
\begin{equation*}
  \forall\, \Delta,\,d,\,\eta\quad
  \exists\, \eps\,,\mu \quad
  \forall\, r_1\quad
  \exists\, c\,.
\end{equation*}

At first sight, the r\^ole of the integer $r$ in Lemma~\ref{lem:blowup} (and also in Lemma~\ref{lem:CL} below)
seems a little obscure.
The only reason for stating the lemma as above is that it is more readily applicable in this form,
since we will need it for pairs of partition classes $(U,V)$ whose size in relation to $n$ 
will be determined by the regularity lemma.

Our last main lemma, the connection lemma (Lemma~\ref{lem:CL}), embeds
graphs~$H$ into graphs~$G$ forming a system of $p$-dense pairs.
In contrast to the blow-up lemma, however, the graph~$H$ has to be much smaller
than the graph~$G$ now (see condition~\ref{lem:CL:Wi}). 
In addition, each vertex $\ti y$ of~$H$ is equipped with a
candidate set $C(\ti y)$ in $G$ from which the connection lemma will choose the
image of $\ti y$ in the embedding.
Lemma~\ref{lem:CL} requires that these candidate
sets are big (condition~\ref{lem:CL:Cbig}) and that pairs of candidate sets
that correspond to an edge of~$H$ form $p$-dense pairs
(condition~\ref{lem:CL:Cdense}). The remaining conditions
(\ref{lem:CL:indep} and \ref{lem:CL:deg}) are conditions on the
neighbourhoods and degrees of the vertices in~$H$ (with respect to the given
partition of~$H$). For their statement we need the following
additional definition. 

For a graph $H$ on vertex set
$\ti{V}=\ti{V}_1\dcup\dots\dcup\ti{V}_t$ and $y\in\ti{V}_i$ with $i\in[t]$
define the \emph{left degree} of~$y$ with respect to the
partition $\ti{V}_1\dcup\dots\dcup\ti{V}_t$ to be
$\ldeg(y;\ti{V}_1,\dots,\ti{V}_t):=\sum_{j=1}^{i-1}\deg_{\ti{V}_j}(y)$. When
clear from the context we may also omit the partition and simply write $\ldeg(y)$.
For two sets of vertices $S$, $T$ we denote the \emph{common neighbourhood} of 
(the vertices of) $S$ in $T$ by
$\coN_{T}(S):= \bigcap_{s\in S} N_T(s)$.

\begin{lemma}[Connection lemma] \label{lem:CL}
  For all integers $\Delta>1$, $t>0$ and reals $d>0$ there are $\eps$,
  $\xi>0$ such that for all positive integers $r_1$ there is $c>1$ such that
  for
  all integers $1\le r\le r_1$ the
  following holds \aas\ for $\Gamma=\Gnp$ with $p\ge c(\log n/n)^{1/\Delta}$.  
  Let $G\subset\Gamma$ be any graph on vertex set $W=W_1\dcup\dots\dcup
  W_t$ and let $H$ be any graph on vertex set
  $\ti{W}=\ti{W}_1\dcup\dots\dcup\ti{W}_t$.
  Suppose further that for each $i\in[t]$ each vertex $\ti{w}\in\ti{W}_i$ is
  equipped with an arbitrary set $X_{\ti w}\subset V(\Gamma)\setminus W$ with
  the property that the indexed set system
  $\big(X_{\ti{w}}\colon\ti{w}\in\ti{W}_i\big)$ consists of pairwise disjoint
  sets such that the following holds. We define the
  \emph{external degree} of $\ti{w}$ to be
  $\edeg(\ti{w}):=|X_{\ti{w}}|$, its \emph{candidate set} $C(\ti{w})\subset
  W_i$ to be $C(\ti{w}):=\coN_{W_i}(X_{\ti{w}})$, and require that 
  \begin{enumerate}[label={\rm (\Alph{*})}]
    \item \label{lem:CL:Wi} $|W_i|\ge n/r$ and $|\ti{W}_i|\le\xi n/r$,
    \item \label{lem:CL:indep} $\ti{W}_i$ is a $3$-independent set in $H$,
    \item \label{lem:CL:deg}
    $\edeg(\ti{w})+\ldeg(\ti{w}) =\edeg(\ti{v})+\ldeg(\ti{v})$ and
    $\deg_H(\ti{w})+\edeg(\ti{w})\le\Delta$ for all $\ti{w},\ti{v}\in\ti{W}_i$,
    \item \label{lem:CL:Cbig} $|C(\ti{w})|\ge((d-\eps)p)^{\edeg(\ti{w})} |W_i|$
    for all $\ti{w}\in\ti{W}_i$, and
    \item \label{lem:CL:Cdense} $(C(\ti{w}),C(\ti{v}))$ forms an
    $(\eps,d,p)$-dense pair for all $\ti{w}\ti{v}\in E(H)$.
  \end{enumerate}
  Then there is an embedding of $H$ into $G$ such that every vertex $\ti{w}\in
  \ti{W}$ is mapped to a vertex in its candidate set $C(\ti{w})$.
\end{lemma}
The quantification of the constants appearing in this lemma is as follows:
\begin{equation*}
  \forall\, \Delta,\,t,\,d\quad
  \exists\, \eps\,,\xi \quad
  \forall\, r_1\quad
  \exists\, c\,.
\end{equation*}
	
The proof of this lemma is inherent in~\cite{KohRoeSchSze}. 
For the details in our setting see 
\inappendix{Section~\ref{sec:CL}}{the appendix of~\cite{BoeKohTar_sparsebip_arxiv}}.
	

\section{Stars in random graphs}
\label{sec:random}

In this section we formulate two lemmas concerning properties of random graphs
that will be useful when analysing neighbourhood properties of $p$-dense pairs
in the following section. More precisely, we consider the following question
here. Given a set of vertices~$X$ in a random graph $\Gamma=\Gnp$ together with a
family~$\hyper F$ of pairwise disjoint $\ell$-sets in $V(\Gamma)$, we
would like to determine how many pairs $(x,F)$ with $x\in X$ and
$F\in\hyper{F}$ have the property that $x$ lies in the common neighbourhood of
the vertices in~$F$.

\begin{definition}[stars]
\label{def:stars}
Let $G=(V,E)$ be a graph, $X$ be a subset of $V$ and $\hyper{F}$ be a family
of pairwise disjoint $\ell$-sets in $V\setminus X$ for some $\ell$. Then the number of
\emph{stars} in $G$ between $X$ and $\hyper{F}$ is
\begin{equation}\label{eq:stars}
  \stars(X,\hyper{F}):=\Big|\,\big\{\,(x,F)\colon\,x\in X,\,F\in\hyper{F},\,
      F\subset N_G(x)\,\big\}\,\Big|.
\end{equation}
\end{definition}

Observe that in a random graph $\Gamma=\Gnp$ and for fixed sets $X$ and
$\hyper{F}$ the random variable $\stars[\Gamma](X,\hyper{F})$ has
binomial distribution $\Bin(|X||\hyper{F}|,p^\ell)$. This will be used in the
proofs of the following lemmas. The first of these lemmas states that in
$\Gnp$ the number of stars between $X$ and $\hyper{F}$ does not exceed its
expectation by more than seven times as long as $X$ and $\hyper{F}$ are not too
small. This is a straightforward consequence of Chernoff's inequality.

\begin{lemma}[star lemma for big sets]\label{lem:stars-big}
  For every positive integer $\Delta$ and every positive real $\nu$
  there is $c$ such that if $p\ge c(\log n/n)^{1/\Delta}$ the
  following holds \aas\ for $\Gamma=\Gnp$ on vertex set $V$. Let $X$
  be any subset of $V$ and $\hyper{F}$ be any family of pairwise
  disjoint $\Delta$-sets in $V\setminus X$.  If $\nu
  n\le|X|\le|\hyper{F}|\le n$, then
  \begin{equation*}
    \stars[\Gamma](X,\hyper{F}) \le 7p^{\Delta}|X||\hyper{F}|.
  \end{equation*}
\end{lemma}
\begin{proof}
  Given $\Delta$ and $\nu$ let $c$ be such that $7c^\Delta\nu^2\ge 3\Delta$.
  From Chernoff's inequality (see~\cite[Chapter~2]{purpleBook}) we know that
  $\Prob[Y\ge7\Exp Y]\le\exp(-7\Exp Y)$ 
  for a
  binomially distributed random variable $Y$. We conclude that for fixed~$X$
  and~$\hyper{F}$
  \begin{equation*}\begin{split}
    \Prob\left[\stars[\Gamma](X,\hyper{F})>7p^{\Delta}|X||\hyper{F}|\right]
      &\le\exp(-7p^{\Delta}|X||\hyper{F}|) \\
      &\le\exp(-7c^\Delta(\log n/n)\nu^2n^2) 
      \le\exp(-3\Delta n\log n)
  \end{split}\end{equation*}
  by the choice of~$c$. Thus the probability that there are sets~$X$
  and~$\hyper{F}$ violating the assertion of the lemma is at most
  \begin{equation*}
    2^n n^{\Delta n} \exp(-3\Delta n\log n)
      \le\exp(2\Delta n\log n-3\Delta n\log n),
  \end{equation*}
  which tends to $0$ as~$n$ tends to infinity.
\end{proof}

We will also need a variant of Lemma~\ref{lem:stars-big} for smaller
sets~$X$ and families~$\hyper{F}$. As a trade-off, the bound on the number
of stars provided by the next lemma will be somewhat worse.  Lemma~\ref{lem:stars-small}
appears almost in this form in~\cite{KohRoeSchSze}. The only (slight)
modification that we need here is that~$X$ is allowed to be bigger than
$\hyper{F}$. However, the same proof as presented in~\cite{KohRoeSchSze}
still works for this modified version.  For the details see
\inappendix{Section~\ref{sec:aux:stars-small}}{the appendix
  of~\cite{BoeKohTar_sparsebip_arxiv}}.

\begin{lemma}[star lemma for small sets]\label{lem:stars-small}
  For all positive integers $\Delta$ and positive reals $\xi$ there are positive
  constants $\nu$ and $c$ such that if $p\ge c(\log n/n)^{1/\Delta}$, then the
  following holds \aas\ for $\Gamma=\Gnp$ on vertex set $V$. Let $X$ be any
  subset of $V$ and $\hyper{F}$ be any family of pairwise disjoint
  $\Delta$-sets in $V\setminus X$. If $|X|\le\nu np^\Delta|\hyper{F}|$ and
  $|X|,|\hyper{F}|\le\xi n$, then
  \begin{equation}\label{eq:stars-small}
    \stars[\Gamma](X,\hyper{F}) 
      \le p^{\Delta}|X||\hyper{F}|+6\xi np^\Delta|\hyper{F}|.
  \end{equation}
\end{lemma}


\section{Common neighbourhoods in \texorpdfstring{$p$}{p}-dense pairs}
\label{sec:joint}

As discussed in Section~\ref{sec:reg} it follows
directly from the definition of $p$-denseness that sub-pairs of dense pairs form again
dense pairs.
In order to apply Lemma~\ref{lem:blowup} and Lemma~\ref{lem:CL} together,
we will need corresponding results on common neighbourhoods in systems of dense
pairs (see Lemmas~\ref{lem:joint} and~\ref{lem:Bad}). For this 
it is necessary to first introduce some notation.

Let $G=(V,E)$ be a graph, $\ell,T>0$ be integers, $p$, $\eps$, $d$ be positive
reals, and $X$, $Y$, $Z\subset V$ be disjoint vertex sets. Recall that for a
set $B$ of vertices from $V$ and a vertex set $Y\subset V$ we call
the set $\coN_Y(B)=\bigcap_{b\in B}N_Y(b)$ the common
neighbourhood of (the vertices in) $B$ in $Y$. 

\begin{definition}[Bad and good vertex sets]
\label{def:bad}
Let $G$, $\ell$, $T$, $p$, $\eps$, $d$, $X$, $Y$, and $Z$
be as above. 
We define the following family of $\ell$-sets in $X$ with small
common neighbourhood in~$Y$:
\begin{equation}\label{eq:BAD}
  \BAD{\ell}{\eps}{d}(X,Y) := \Big\{B\in\binom{X}{\ell}\colon\,
    |\coN_Y(B)|<(d-\eps)^{\ell}p^{\ell}|Y| \Big\}\,.
\end{equation}
If $(X,Y)$ has $p$-density $d_{G,p}(X,Y)\ge d-\eps$, then all
$\ell$-sets $T\in\binom{X}{\ell}$ that are not in 
$\BAD{\ell}{\eps}{d}(X,Y)$ are called \emph{$p$-good} in
$(X,Y)$.

Let further
\begin{equation*}
  \Bad{\ell}{\eps}{d}(X,Y,Z)
\end{equation*}
be the family of $\ell$-sets $B \in\binom{X}{\ell}$ that contain an
$\ell'$-set $B'\subset B$ with $\ell'>0$ such that either
$|\coN_Y(B')|<(d-\eps)^{\ell'}p^{\ell'}|Y|$ or $(\coN_Y(B'),Z)$ is not
$(\eps,d,p)$-dense in $G$.
\end{definition}
The following lemma 
states that $p$-dense pairs in
random graphs have the property that most $\ell$-sets have big common
neighbourhoods. Results of this type (with a slightly smaller exponent in the
edge probability~$p$)
were established in~\cite{KR03}.
The proof of Lemma~\ref{lem:joint} can be found in
\inappendix{Section~\ref{sec:aux:joint}}{the appendix of~\cite{BoeKohTar_sparsebip_arxiv}}.

\begin{lemma}[common neighbourhood lemma]
\label{lem:joint}
  For all integers $\Delta,\ell\ge 1$ and positive reals $d$, $\eps'$ and $\mu$,
  there is $\eps>0$ such that for all $\xi>0$ there is $c>1$ such that
  if $p\ge c(\log n/n)^{1/\Delta}$, then the following holds \aas\ for
  $\Gamma=\Gnp$. For $n_1\ge\xi p^{\Delta-1}n$, $n_2\ge\xi
  p^{\Delta-\ell}n$ let $G=(X\dcup Y,E)$ be any bipartite subgraph of $\Gamma$
  with $|X|=n_1$ and $|Y|=n_2$. If $(X,Y)$ is an $(\eps,d,p)$-dense pair, then
  $|\bad{\ell}{\eps'}{d}(X,Y)|\le\mu n_1^\ell$.
\end{lemma}

Thus we know that typical vertex sets in dense pairs inside random graphs are
$p$-good. In the next lemma we observe that families of such $p$-good vertex
sets exhibit strong expansion properties.

Given $\Delta$ and $p$ we say that a bipartite graph~$G=(X\dcup Y,E)$ is
\emph{$(A,f)$-expanding}, if, for any family
$\cF\subset\binom{X}{\Delta}$ of pairwise disjoint $p$-good $\Delta$-sets 
in $(X,Y)$
with~$|\cF|\leq A$, we have
$|\coN_Y(\cF)|\geq f|\cF|$.

\begin{lemma}[expansion lemma]
\label{lem:exp}
  For all positive integers~$\Delta$ and positive reals~$d$ and~$\eps$,
  there are positive~$\nu$ and~$c$ such that if $p\ge c(\log
  n/n)^{1/\Delta}$, then the following holds \aas\ for $\Gamma=\Gnp$. Let
  $G=(X\dcup Y,E)$ be a bipartite subgraph of $\Gamma$. If $(X,Y)$ is an
  $(\eps,d,p)$-dense pair, then $(X,Y)$ is $(1/p^\Delta,\nu
  np^\Delta)$-expanding.
\end{lemma}
\begin{proof}
  Given~$\Delta$, $d$, $\eps$, set $\delta:=d-\eps$,
  $\xi:=\delta^\Delta/7$ and let $\nu'$ and $c$ be the constants from
  Lemma~\ref{lem:stars-small} for this $\Delta$ and $\xi$.
  Further, choose $\nu$ such that
  $\nu\le\xi$ and $\nu\le\nu'$.
  Let~$\cF\subset\binom{X}{\Delta}$ be a family of
  pairwise disjoint $p$-good $\Delta$-sets with~$|\cF|\leq1/p^\Delta$. Let~$U=\coN_Y(\cF)$ be
  the common neighbourhood of~$\cF$ in $Y$. We wish to show
  that~$|U|\geq(\nu np^\Delta)|\cF|$. Suppose the contrary. Then
  $|U|<\nu'np^\Delta|\cF|$, $|U|<\nu np^\Delta|\cF|\le\nu
  n\le\xi n$ and $|\cF|\le1/p^\Delta\le c^\Delta n/\log n\le\xi n$ for $n$
  sufficiently large and so we can apply Lemma~\ref{lem:stars-small} with
  parameters $\Delta$ and $\xi$ to $U$ and $\cF$. Since every member of~$\cF$ is
  $p$-good in $(X,Y)$, we thus have
  \begin{equation*}\begin{split}
    \delta^{\Delta}p^{\Delta}n|\cF| &\leq \stars(U,\cF) 
    \leq \stars[\Gamma](U,\cF)
    \leByRef{eq:stars-small} p^\Delta|U||\cF|+6\xi np^{\Delta}|\cF|\\
    &<p^{\Delta}(\nu np^\Delta)|\cF||\cF|+6\xi np^\Delta|\cF|
    \leq\nu np^\Delta|\cF|+6\xi np^{\Delta}|\cF|
    \leq 7\xi np^{\Delta}|\cF|,
  \end{split}\end{equation*}
  which yields that $\delta^\Delta<7\xi$, a
  contradiction.
\end{proof}

In the remainder of this section we are interested in the inheritance of
$p$-denseness to sub-pairs $(X',Y')$ of $p$-dense pairs $(X,Y)$ in a
graph~$G=(V,E)$. It comes as a surprise that even for sets $X'$ and $Y'$ that
are much smaller than the sets considered in the definition of $p$-denseness,
such sub-pairs are typically dense. Phenomena of this type were observed
in~\cite{KR03,GKRS07}.

Here, we will consider sub-pairs induced by neighbourhoods of vertices~$v\in V$
(which may or may not be in $X\dcup Y$), i.e., sub-pairs $(X',Y')$ where $X'$
(or $Y'$ or both) is the neighbourhood of $v$ in~$Y$ (or in $X$). 
Further, we only consider
the case when $G$ is a subgraph of a random graph $\Gnp$.

In~\cite{KohRoeSchSze} an inheritance result of this form was obtained for
triples of dense pairs. More precisely, the following holds for subgraphs~$G$
of~$\Gnp$. For sufficiently large vertex set $X$, $Y$, and $Z$ in~$G$ such that $(X,Y)$ and
$(Y,Z)$ form $p$-dense pairs we have that most vertices $x\in X$ are such
that $(N_Y(x),Y)$ forms again a $p$-dense pair (with slightly changed
parameters). If, moreover, $(X,Z)$ forms a $p$-dense pair, too, then 
$(N_Y(x),N_Z(x))$ is typically also a $p$-dense pair.

\begin{lemma}[inheritance lemma for vertices~\cite{KohRoeSchSze}] 
\label{lem:reg-neighb}
  For all integers $\Delta>0$ and positive reals $d_0$, $\eps'$ and $\mu$ there
  is $\eps$ such that for all $\xi>0$ there is $c>1$ such that if
  $p>c(\log n/n)^{1/\Delta}$, then the following holds \aas\ for $\Gamma=\Gnp$.
  For $n_1$, $n_3\ge\xi p^{\Delta-1}n$ and $n_2\ge\xi p^{\Delta-2}n$ let $G=(X\dcup
  Y\dcup Z,E)$ be any tripartite subgraph of $\Gamma$ with $|X|=n_1$, $|Y|=n_2$, and
  $|Z|=n_3$. If $(X,Y)$ and $(Y,Z)$ are $(\eps,d,p)$-dense pairs in $G$ with
  $d\ge d_0$, then there are at most $\mu n_1$ vertices $x\in X$ such that
  $(N(x)\cap Y,Z)$ is not an $(\eps',d,p)$-dense pair in $G$. 
  
  If, additionally, $(X,Z)$ is $(\eps,d,p)$-dense and $n_1$, $n_2$, $n_3\ge\xi
  p^{\Delta-2}n$, then there are at most $\mu n_1$ vertices $x\in X$ such that
  $(N(x)\cap Y,N(x)\cap Z)$ is not an $(\eps',d,p)$-dense pair in $G$.
\qed
\end{lemma}

In order to combine the constrained blow-up lemma (Lemma~\ref{lem:blowup}) and
the connection lemma (Lemma~\ref{lem:CL}) in the proof of
Theorem~\ref{thm:main} we will need a version of this result for
$\ell$-sets. Such a lemma, stating that common
neighbourhoods of certain $\ell$-sets form again $p$-dense pairs, can be
obtained by an inductive argument from the first part of
Lemma~\ref{lem:reg-neighb}. 
For a proof see
\inappendix{Section~\ref{sec:aux:Bad}}{the appendix of~\cite{BoeKohTar_sparsebip_arxiv}}.

\begin{lemma}[inheritance lemma for $\ell$-sets]
\label{lem:Bad}
  For all integers $\Delta,\ell>0$ and positive reals $d_0,\eps'$, and $\mu$
  there is $\eps$ such that for all $\xi>0$ there is $c>1$ such that if $p>c(\frac{\log
  n}{n})^{1/\Delta}$, then the following holds \aas\ for $\Gamma=\Gnp$. 
  For $n_1,n_3\ge\xi p^{\Delta-1}n$ and $n_2\ge\xi p^{\Delta-\ell-1}n$ let
  $G=(X\dcup Y\dcup Z,E)$ be any tripartite subgraph of $\Gamma$ with
  $|X|=n_1$, $|Y|=n_2$, and $|Z|=n_3$. Assume further that $(X,Y)$ and $(Y,Z)$
  are $(\eps,d,p)$-dense pairs with $d\ge d_0$. Then
  \begin{equation*}
    \big|\Bad{\ell}{\eps'}{d}(X,Y,Z)\big|\le\mu n_1^\ell.
  \end{equation*}  
\end{lemma}


\section{Proof of Theorem~\ref{thm:main}}
\label{sec:proof}

In this section we present a proof of Theorem~\ref{thm:main} that
combines our four main lemmas, namely the lemma for~$G$ (Lemma~\ref{lem:G}),
the lemma for~$H$ (Lemma~\ref{lem:H}), the constrained blow-up lemma
(Lemma~\ref{lem:blowup}), and the connection lemma
(Lemma~\ref{lem:CL}). This proof follows the outline given in
Section~\ref{sec:idea}. In addition we will
apply the inheritance lemma for $\ell$-sets (Lemma~\ref{lem:Bad}), which
supplies an appropriate interface between the constrained blow-up lemma and
the connection lemma.

\begin{proof}[Proof of Theorem~\ref{thm:main}]
  We first set up the constants.
  Given $\eta$, $\gamma$, and $\Delta$ let $t$ be the
  constant promised by the lemma for $H$ (Lemma~\ref{lem:H}) for
  input~$\Delta$. Set 
  \begin{equation}\label{eq:proof:etar0}
     \eta_\isubsc{G}:=\eta/10, 
     \qquad\text{and}\qquad
     r_0=1\,,
  \end{equation}
  and apply the lemma for $G$ (Lemma~\ref{lem:G}) with input
  $t$, $r_0$, $\eta_\isubsc{G}$, and $\gamma$ in order to obtain $\eta'_\isubsc{G}$ and $d$.
  Next, the connection lemma (Lemma~\ref{lem:CL}) with input $\Delta$, $2t$, and
  $d$ provides us with $\eps_{\subsc{CL}}$, and $\xi_{\subsc{CL}}$. We
  apply the constrained blow-up
  lemma (Lemma~\ref{lem:blowup}) with $\Delta$, $d$, and $\eta/2$ 
  in order to obtain $\eps_{\subsc{BL}}$ and $\mu_{\subsc{BL}}$. 
  With this we set
  \begin{equation}\label{eq:proof:eta}
    \eta_\isubsc{H}:=\min\{\eta/10,\,\xi_{\subsc{CL}}\eta'_\isubsc{G},\,1/(\Delta+1)\}.
  \end{equation}
  Choose $\mu>0$ such that
  \begin{equation}\label{eq:proof:mu}
    100t^2\mu\le\eta_{\subsc{BL}},
  \end{equation}
  and apply Lemma~\ref{lem:Bad} with $\Delta$ and
  $\ell=\Delta-1$, 
  $d_0=d$, $\eps'=\eps_{\subsc{CL}}$, and $\mu$ to obtain
  $\eps_{\subref{lem:Bad}}$. Let 
  \begin{equation}\label{eq:proof:xi}
      \xi_{\subref{lem:Bad}}:=\eta'_\isubsc{G}/2r
  \end{equation}
  and continue the application of Lemma~\ref{lem:Bad} with $\xi_{\subref{lem:Bad}}$ to obtain
  $c_{\subref{lem:Bad}}$. 
  Now we can fix
  \begin{equation}\label{eq:proof:eps}
    \eps:=\min\{\eps_{\subsc{CL}},\eps_{\subsc{BL}},\eps_{\subref{lem:Bad}}\}
  \end{equation}
  and continue the application of Lemma~\ref{lem:G} with input $\eps$ to get
  $r_1$. Let $\hat{r}_{\subsc{BL}}$ and $\hat{r}_{\subsc{CL}}$ be such that
  \begin{equation}\label{eq:proof:rhat}
     \frac{2r_1}{1-\eta_\isubsc{G}} \le \hat{r}_{\subsc{BL}}
     \qquad\text{and}\qquad
     \frac{2r_1}{\eta_\isubsc{G}} \le \hat{r}_{\subsc{CL}}
  \end{equation}
  and let $c_{\subsc{CL}}$ and  $c_{\subsc{BL}}$ be the constants obtained 
  from the continued application of Lemma~\ref{lem:CL} with $r_1$ replaced
 by $\hat{r}_{\subsc{CL}}$ and Lemma~\ref{lem:blowup} with $r_1$ replaced
  by $\hat{r}_{\subsc{BL}}$, respectively. 
  
  We continue the application of Lemma~\ref{lem:H} with input $\eta_\isubsc{H}$.
  For each $r\in[r_1]$ Lemma~\ref{lem:H} provides a value
  $\beta_r\,$, among all of which we choose the smallest one and set $\beta$ to
  this value.
  Finally, we set 
  $c:=\max\{c_{\subsc{BL}},c_{\subsc{CL}},c_{\subref{lem:Bad}}\}$.
  
  Consider a graph $\Gamma=\Gnp$ with $p\ge c(\log
  n/n)^{1/\Delta}$. Then $\Gamma$ \aas\ satisfies the properties stated in
  Lemma~\ref{lem:G}, Lemma~\ref{lem:blowup}, Lemma~\ref{lem:CL},
  and Lemma~\ref{lem:Bad}, with the parameters previously specified. We
  assume in the following that this is the case and show that then also the following holds.  
  For all subgraphs
  $G\subset\Gamma$ and all graphs $H$ such that $G$ and $H$ have the properties
  required by Theorem~\ref{thm:main} we have $H\subset G$. 
  To summarise the definition of the constants above, we can now
  assume that $\Gamma$ satisfies the conclusion of the
  following lemmas:
  \begin{itemize}[leftmargin=1.5cm]
    \item[(L\ref{lem:G})]
      { 
      Lemma~\ref{lem:G} for parameters $t$, $r_0=1$,
      $\eta_\isubsc{G}$, $\gamma$, $\eta'_\isubsc{G}$, $d$, $\eps$, and $r_1$,
      i.e., if $G$ is any spanning subgraph of $\Gamma$ satisfying the
      requirements of Lemma~\ref{lem:G}, then we obtain a partition
      of~$G$ as specified in the lemma with these parameters,}
    \item[(L\ref{lem:blowup})] 
      { 
      Lemma~\ref{lem:blowup} for parameters
      $\Delta$, $d$, $\eta/2$, $\eps_\subsc{BL}$, $\mu_\subsc{BL}$,
      and $\hat{r}_{\subsc{BL}}$,}
    \item[(L\ref{lem:CL})] 
      { 
      Lemma~\ref{lem:CL} for parameters
      $\Delta$, $2t$, $d$, $\eps_\subsc{CL}$, $\xi_\subsc{CL}$,
      and $\hat{r}_{\subsc{CL}}$,}
    \item[(L\ref{lem:Bad})] 
      { 
      Lemma~\ref{lem:Bad} for parameters $\Delta$, $\ell=\Delta-1$, 
      $d_0=d$, $\eps'=\eps_{\subsc{CL}}$, $\mu$, $\eps_{\subref{lem:Bad}}$, and
      $\xi_{\subref{lem:Bad}}$.}
  \end{itemize}  
  
  Now suppose we are given a graph $G=(V,E)\subset\Gamma$ with
  $\deg_G(v)\ge(\frac{1}{2}+\gamma)\deg_\Gamma(v)$ for all $v\in V$ and $|V|=n$,
  and a graph $H=(\ti{V},\ti{E})$ with $|\ti{V}|=(1-\eta)n$.
  Before we show that $H$ can be embedded into~$G$ we will use
  the lemma for~$G$ (Lemma~\ref{lem:G}) and the lemma
  for $H$ (Lemma~\ref{lem:H}) to prepare $G$ and $H$ for this embedding.

  First we use the fact that~$\Gamma$ has property (L\ref{lem:G}).
  Hence, for the graph $G$ we obtain an $r$ with $1\le r\le r_1$
  from Lemma~\ref{lem:G}, together with a set
  $V_0\subset V$ with $|V_0|\le\eps n$, and a mapping $g\colon V\setminus
  V_0\to\MON_{r,t}$ such that \ref{lem:G:Vi}--\ref{lem:G:reg} of
  Lemma~\ref{lem:G} are fulfilled. For all $i\in[r], j\in[2t]$ let $U_{i}$,
  $V_{i}$, $C_{i,j}$, $C'_{i,j}$, $B_{i,j}$, and $B'_{i,j}$ be the sets defined
  in Lemma~\ref{lem:G}. 
  Recall that these sets were called big clusters, connecting clusters, and
  balancing clusters. With this the graph~$G$ is prepared for the embedding.
  We now turn to the graph~$H$.
  
  We assume for simplicity that
  $2r/(1-\eta_\isubsc{G})$ and $r/(t\eta'_\isubsc{G})$ are integers and define
  \begin{equation}\label{eq:proof:r}
    r_{\subsc{BL}}:=2r/(1-\eta_\isubsc{G})
    \qquad\text{and}\qquad
    r_{\subsc{CL}}:=2r/\eta'_\isubsc{G}\,.
  \end{equation}
  We apply Lemma~\ref{lem:H} which we already provided with $\Delta$ and
  $\eta_\isubsc{H}$. For input~$H$ this lemma provides a homomorphism
  $h$ from $H$ to $\MON_{r,t}$ such that \ref{lem:H:Vi}--\ref{lem:H:X} of
  Lemma~\ref{lem:H} are fulfilled. For all $i\in[r], j\in[2t]$ let
  $\ti{U}_{i}$, $\ti{V}_{i}$, $\ti{C}_{i,j}$, $\ti{C}'_{i,j}$, $\ti{B}_{i,j}$,
  $\ti{B}'_{i,j}$, and $\ti{X}_i$
  be the sets whose existence is guaranteed by Lemma~\ref{lem:H}. Further, set
  $C_i:=C_{i,1}\dcup\dots\dcup C_{i,2t}$,
  $\ti{C}_i:=\ti{C}_{i,1}\dcup\dots\dcup \ti{C}_{i,2t}$, that is, 
  $C_i$ consists of connecting clusters and $\ti{C}_i$
  of connecting vertices. Define $C'_i$, $\ti{C}'_i$, $B_i$,
  $\ti{B}_i$, $B'_i$, and $\ti{B}'_i$ analogously ($B_i$ consists of
  balancing clusters and $\ti B_i$ of balancing vertices).
  
  Our next goal will be to appeal to property (L\ref{lem:blowup}) which asserts
  that we can apply the constrained blow-up lemma
  (Lemma~\ref{lem:blowup}) for each $p$-dense pair $(U_i,V_i)$
  with $i\in[r]$ individually and embed $H[\ti U_i\dcup\ti V_i]$ into this
  pair. For this we fix $i\in[r]$. We will first set up special $\Delta$-sets
  $\hyper{H}_i$ and forbidden $\Delta$-sets $\hyper{B}_i$ for the application
  of Lemma~\ref{lem:blowup}. The idea is as follows. With the help of
  Lemma~\ref{lem:blowup} we will embed all vertices in $\ti U_i\dcup\ti V_i$.
  But all connecting and balancing vertices of~$H$ remain unembedded. They will
  be handled by the connection lemma, Lemma~\ref{lem:CL}, later on.
  However, these two lemmas cannot operate independently. If, for example, a
  connecting vertex $\ti y$ has three neighbours in $\ti  V_i$, then these
  neighbours will be already mapped to vertices $v_1,v_2,v_3$ in $V_i$
  (by the blow-up lemma) when we want to embed $\ti y$. Accordingly the
  image of $\ti y$ in the embedding is confined to the common neighbourhood of
  the vertices $v_1,v_2,v_3$ in $G$. In other words, this common
  neighbourhood will be the candidate set $C(\ti y)$ in the application of
  Lemma~\ref{lem:CL}. This lemma requires, however, that candidate sets are
  not too small (condition~\ref{lem:CL:Cbig} of Lemma~\ref{lem:CL}) and, in
  addition, that candidate sets of any two adjacent vertices induce $p$-dense
  pairs (condition~\ref{lem:CL:Cdense}). Hence we need to be prepared for these
  requirements. This will be done via the special and forbidden sets. The
  family of special sets $\hyper{H}_i$ will contain neighbourhoods in $\ti V_i$
  of connecting or balancing vertices $\ti y$ of $H$ (observe that such
  vertices do not have neighbours in $\ti U_i$\,, see
  Figure~\ref{fig:backbone}). The family of forbidden sets $\hyper{B}_i$ will
  consist of sets in $V_i$ which are ``bad'' for the embedding of these
  neighbourhoods in view of~\ref{lem:CL:Cbig} and~\ref{lem:CL:Cdense} of
  Lemma~\ref{lem:CL} (recall that Lemma~\ref{lem:blowup} does not map special
  sets to forbidden sets). Accordingly, $\hyper{B}_i$ contains $\Delta$-sets
  that have small common neighbourhoods or do not induce $p$-dense pairs in one
  of the relevant balancing or connecting clusters.
  We will next give the details of this construction
  of $\hyper{H}_i$ and $\hyper{B}_i$.

  We start with the special $\Delta$-sets $\hyper{H}_i$. As explained, we would
  like to include in the family~$\hyper{H}_i$ all neighbourhoods of vertices $\ti w$ of
  vertices outside $\ti U_i\dcup\ti V_i$. Such neighbourhoods clearly lie
  entirely in the set~$\ti X_i$ provided by Lemma~\ref{lem:H}.
  However, they need not necessarily be $\Delta$-sets (in fact,
  by~\ref{lem:H:deg} of Lemma~\ref{lem:H}, they are of size at most
  $\Delta-1$). Therefore we have to ``pad'' these neighbourhoods in order to
  obtain $\Delta$-sets. This is done as follows. We start by picking
  an arbitrary set of $\Delta|\ti X_i|$ vertices (which will be used for the
  ``padding'') in $\ti V_i\setminus \ti X_i$. We add these vertices to $\ti X_i$
  and call the resulting set~$\ti{X}'_i$. This is possible
  because~\ref{lem:H:X} of Lemma~\ref{lem:H} and~\eqref{eq:proof:eta} imply that
    $|\ti{X}'_i|\le(\Delta+1)|\ti{X}_i|
    \le (\Delta+1)\eta_\isubsc{H}|\ti{V}_i|
    \le|\ti{V}_i|$.
    
  Now let $\ti{Y}_i$ be the set of vertices in $\ti{B}_i\dcup\ti{C}_i$ with
  neighbours in $\ti{V}_i$. These are the vertices for whose neighbourhoods we
  will include $\Delta$-sets in~$\hyper H_i$.
  It follows from the definition of $\ti{X}_i$ that
  $|\ti{Y}_i|\le\Delta|\ti{X}_i|$.  
  Let $\ti{y}\in\ti{Y}_i\subset\ti{B}_i\cup\ti{C}_i$.   
  By the definition of $\ti{X}_i$ we have $N_H(\ti{y})\subset\ti{X}_i$. 
  Next, we let
  \begin{equation}\label{eq:proof:Xy}
    \ti X_{\ti y}\ \text{be the set of neighbours of $\ti{y}$ in $\ti V_i$}\,.
  \end{equation}
  As explained, $\ti{y}$ has strictly less than $\Delta$ neighbours in $\ti{V}_i$
  and hence
  we choose additional vertices from $\ti{X}'_i\setminus\ti{X}_i$. In this way
  we obtain for each $\ti{y}\in\ti{Y}_i$ a $\Delta$-set $N_{\ti{y}}\in\ti{X}'_i$ with
  \begin{equation}\label{eq:proof:Ny}
    N_{\ti{X}_i}(\ti y)=N_{\ti{V}_i}(\ti y)=\ti X_{\ti y}\subset N_{\ti{y}}\,.
  \end{equation}
  We make sure, in this process, that for any two
  different $\ti{y}$ and $\ti{y}'$ we never include the same additional vertex
  from $\ti{X}'_i\setminus\ti{X}_i$. 
  This is possible because
  $|\ti{X}'_i\setminus\ti{X}_i|\ge\Delta|\ti{X}_i|\ge|\ti{Y}_i|$.
  We can thus guarantee that
  \begin{equation}\label{eq:proof:special}
    \text{each vertex in $\ti{X}'_{i}$ is contained in at most
      $\Delta$ sets $N_{\ti{y}}$.}
  \end{equation}
  The family of special $\Delta$-sets for the application of
  Lemma~\ref{lem:blowup} on $(U_i,V_i)$ is then
  \begin{equation}\label{eq:proof:Hi}
    \hyper{H}_i:=\{N_{\ti{y}}\colon \ti{y}\in\ti{Y}_i\}\,.
  \end{equation}
  Note that this is indeed a family of $\Delta$-sets encoding all
  neighbourhoods in $\ti U_i\dcup\ti V_i$ of vertices 
  outside this set.

  Now we turn to the family~$\hyper B_i$ of forbidden $\Delta$-sets. Recall
  that this family should contain sets that are forbidden for the embedding of
  the special $\Delta$-sets because their common neighbourhood in a
  (relevant) balancing or connecting cluster is small or does not induce a
  $p$-dense pair. More precisely, we are interested in $\Delta$-sets~$S$ 
  that have one of the following properties. Either $S$
  has a small common neighbourhood in some cluster from $B_i$ or from $C_i$
  (observe that only balancing vertices from $\ti B_i$ and connecting
  vertices from $\ti C_i$ have neighbours in $\ti V_i$). Or the neighbourhood
  $\coN_D(S)$ of $S$ in a cluster $D$ from $B_i$ or $C_i$, respectively, is such
  that $(\coN_D(S),D')$ is not $p$-dense for some cluster $D'$ from $B'_i\cup
  B'_{i+1}$ or $C'_i\cup C'_{i+1}$ (observe that edges between balancing
  vertices run only between $\ti B_i$ and $\ti B'_i\cup \ti B'_{i+1}$ and edges
  between connecting vertices only between $\ti C_i$ and $\ti C'_i\cup \ti
  C'_{i+1}$).
  
  For technical reasons, however, we need to digress from this strategy
  slightly: We want to bound the number of $\Delta$-sets in $\hyper B_i$ with
  the help of the inheritance lemma for $\ell$-sets, Lemma~\ref{lem:Bad},
  later. Notice that, thanks to the lower bound on $n_2$ in
  Lemma~\ref{lem:Bad}, this lemma cannot be applied (in a meaningful way) for
  $\Delta$-sets. But it can be applied for $(\Delta-1)$-sets.  
  Therefore, we will not consider $\Delta$-sets directly but
  first construct an auxiliary family of $(\Delta-1)$-sets and then, again,
  ``pad'' these sets to obtain a family of $\Delta$-sets. Observe that the
  strategy outlined while setting up the special sets~$\hyper{H}_i$
  still works with these
  $(\Delta-1)$-sets: neighbourhoods of connecting or balancing
  vertices in $\ti V_i$ are of size at most $\Delta-1$ by~\ref{lem:H:deg} of Lemma~\ref{lem:H}.
  
  But now let us finally give the details.
  We first define the auxiliary family of $(\Delta-1)$-sets as follows:
  \begin{equation}\label{eq:proof:bad}
  \begin{split}
    \hyper{B}'_i:=
    &\bigcup_{\substack{i'\in\{i,i+1\},j,j'\in[2t] \\
      (c_{i,j},c'_{i',j'})\in\MON_{r,t}}}
      \Bad{\Delta-1}{\eps_{\subsc{CL}}}{d}
      (V_i,C_{i,j},C'_{i',j'})
    \quad\cup
      \bigcup_{\substack{j,j'\in[2t]\\ (b_{i,j},b'_{i,j'})\in\MON_{r,t}}}
      \Bad{\Delta-1}{\eps_{\subsc{CL}}}{d}
      (V_i,B_{i,j},B_{i,j'}).
  \end{split}\end{equation}

  We will next bound the size of this family by appealing to property
  (L\ref{lem:Bad}), and hence Lemma~\ref{lem:Bad}, with the
  tripartite graphs $G[V_i,C_{i,j},C'_{i',j'}]$ and $G[V_i,B_{i,j},B'_{i,j'}]$
  with indices as in the definition of~$\hyper{B}'_i$. For
  this we need to check the conditions appearing in this lemma. By the definition of $\MON_{r,t}$ and
  \ref{lem:G:reg} of Lemma~\ref{lem:G} all pairs $(C_{i,j},C'_{i',j'})$ and
  $(B_{i,j},B'_{i,j'})$ appearing in the definition of $\hyper{B}'_i$
  as well as the pairs
  $(V_{i},C_{i,j})$ and $(V_{i},B_{i,j})$ with $j\in[2t]$ are
  $(\eps,d,p)$-dense. For the vertex sets of these dense pairs we know
  $|V_{i}|,|C'_{i',j'}|,|B'_{i,j'}| \ge\eta'_\isubsc{G} n/2r
  \ge\xi_{\subref{lem:Bad}}p^{\Delta-1}n$ and $|C_{i,j}|,|B_{i,j}|
  \ge\eta'_\isubsc{G} n/2r =\xi_{\subref{lem:Bad}}n$ by~\ref{lem:G:Vi}
  and~\ref{lem:G:Ci} of Lemma~\ref{lem:G}
  and~\eqref{eq:proof:xi}.
  Thus, since $\eps\le\eps_{\subref{lem:Bad}}$, property (L\ref{lem:Bad})
  implies
  that the family
  \begin{align*}
    \Bad{\Delta-1}{\eps_{\subsc{CL}}}{d}
     (V_i,C_{i,j},C'_{i',j'}), \text{ and }
    \Bad{\Delta-1}{\eps_{\subsc{CL}}}{d}
     (V_i,B_{i,j},B'_{i,j'})
  \end{align*}
  is of size $\mu|V_i|^{\Delta-1}$ at most. It follows
  from~\eqref{eq:proof:bad} that $|\hyper{B}'_i|\le 8t^2\mu|V_i|^{\Delta-1}$ which is at most $\mu_{\rm
  BL}|V_i|^{\Delta-1}$ by~\eqref{eq:proof:mu}. The family of forbidden
  $\Delta$-sets is then defined by
  \begin{equation}
    \hyper{B}_i:=\hyper{B}'_i\times V_i \quad\text{and we have}\quad
    |\hyper{B}_i|\le\mu_{\subsc{BL}}|V_i|^\Delta\,.
  \label{eq:proof:corrupt}
  \end{equation}

  Having defined the special and forbidden $\Delta$-sets we are now ready to
  appeal to (L\ref{lem:blowup}) and use the constrained blow-up lemma
  (Lemma~\ref{lem:blowup}) 
  with parameters $\Delta$, $d$, $\eta/2$, $\eps_\subsc{BL}$, $\mu_\subsc{BL}$,
  $\hat{r}_{\subsc{BL}}$, and $r_{\subsc{BL}}$ separately for each pair of graphs
  $G_i:=(U_i,V_i)$ and $H_i:=H[\ti{U_i}\dcup\ti{V_i}]$. 
  Let us quickly check that the constant~$r_{\subsc{BL}}$ and the graphs $G_i$
  and $H_i$ satisfy the required conditions. Observe first, that 
  $1\le r_\subsc{BL} = 2r/(1-\eta_\isubsc{G})
  \le 2r_1/(1-\eta_\isubsc{G})\le
  \hat{r}_{\subsc{BL}}$ by~\eqref{eq:proof:r} and~\eqref{eq:proof:rhat}.
  Moreover $(U_i,V_i)$ is an
  $(\eps_\subsc{BL},d,p)$-dense pair by~\ref{lem:G:reg} of
  Lemma~\ref{lem:G} and~\eqref{eq:proof:eps}.
  \ref{lem:G:Vi} implies
  \begin{equation*}
    |U_i|\ge (1-\eta_\isubsc{G})\frac{n}{2r}
    \eqByRef{eq:proof:r} \frac{n}{r_{\subsc{BL}}}
  \end{equation*}
  and similarly $|V_i|\ge n/r_{\subsc{BL}}$.
  By~\ref{lem:H:Vi} of Lemma~\ref{lem:H} we have
  \begin{equation*}\begin{split}
    |\ti{U}_i| &\le(1+\eta_\isubsc{H})\frac{m}{2r}
    \le(1+\eta_\isubsc{H})(1-\eta)\frac{n}{2r}
    \le (1+\eta_\isubsc{H}-\eta)\frac{n}{2r}
    \leBy{\eqref{eq:proof:etar0},\eqref{eq:proof:eta}}
      (1-\tfrac12\eta-\eta_\isubsc{G})\frac{n}{2r} \\
    & \le
      (1-\tfrac12{\eta})(1-\eta_\isubsc{G})\frac{n}{2r}
    \eqByRef{eq:proof:r}
      (1-\tfrac12{\eta})\frac{n}{r_{\subsc{BL}}}
  \end{split}
  \end{equation*}
  and similarly $|\ti{V}_i|\le(1-\frac{\eta}{2})n/r_{\subsc{BL}}$.
  For the application of Lemma~\ref{lem:blowup}, let the families of 
  special and forbidden $\Delta$-sets be
  defined in~\eqref{eq:proof:Hi} and~\eqref{eq:proof:corrupt},
  respectively. Observe that~\eqref{eq:proof:special}
  and~\eqref{eq:proof:corrupt} guarantee that the required conditions (of
  Lemma~\ref{lem:blowup}) are satisfied. Consequently there is an embedding of $H_i$
  into $G_i$ for each $i\in[r]$ such that no special $\Delta$-set is mapped to a forbidden $\Delta$-set.
  Denote the united embedding resulting from these $r$ applications of the
  constrained blow-up lemma by $f_{\subsc{BL}}:\bigcup_{i\in[r]}\ti{U}_i\cup
  \ti{V}_i \to \bigcup_{i\in[r]}U_i\cup V_i$. 
  
  It remains to verify that $f_{\subsc{BL}}$ can be extended to an embedding of
  all vertices of $H$ into $G$. We still need to take care of the balancing and
  connecting vertices. For this purpose we will, again, fix $i\in[r]$ and use
  property (L\ref{lem:CL}) which states that the conclusion of the
  connection lemma (Lemma~\ref{lem:CL}) holds for parameters
  $\Delta$, $2t$, $d$, $\eps_{\subsc{CL}}$, $\xi_{\subsc{CL}}$, and
  $\hat{r}_{\subsc{CL}}$. We will apply this lemma with input
  $r_{\subsc{CL}}$ to the graphs $G'_i:=G[W_i]$ and $H'_i:=H[\ti{W}_i]$ where
  $W_i$ and $\ti{W}_i$ and their partitions for the application of the connection
  lemma are as follows (see Figure~\ref{fig:connect}). Let
  $W_i:=W_{i,1}\dcup\dots\dcup W_{i,8t}$ where for
  all $j\in[t],k\in[2t]$ we set
  \begin{align*}
    W_{i,j}&:=C_{i,t+j}\,, &
    W_{i,t+j}&:=C_{i+1,j}\,, & 
    W_{i,2t+j}&:=C'_{i,t+j}\,, \\
    W_{i,3t+j}&:=C'_{i+1,j}\,, & 
    W_{i,4t+k}&:=B_{i,k}\,, &
    W_{i,6t+k}&:=B'_{i,k}\,.
  \end{align*}
  ({\it This means that we propose the clusters in the following order to the
  connection lemma. The connecting clusters without primes come first, then the
  connecting clusters with primes, then the balancing clusters without primes,
  and finally the balancing clusters with primes.}\,)

  The partition
  $\ti{W}_i:=\ti{W}_{i,1}\dcup\dots\dcup \ti{W}_{i,8t}$ of the vertex set
  $\ti{W}_i$ of $H'_i$ is defined accordingly, i.e.,  for all
  $j\in[t],k\in[2t]$ we set
  \begin{align*}
    \ti W_{i,j}&:=\ti C_{i,t+j}\,, &
    \ti W_{i,t+j}&:=\ti C_{i+1,j}\,, & 
    \ti W_{i,2t+j}&:=\ti C'_{i,t+j}\,, \\
    \ti W_{i,3t+j}&:=\ti C'_{i+1,j}\,, & 
    \ti W_{i,4t+k}&:=\ti B_{i,k}\,, &
    \ti W_{i,6t+k}&:=\ti B'_{i,k}\,.
  \end{align*}  
  \begin{figure}
  \begin{center}
    \psfrag{u2}{\sml{$u_{i}$}}
    \psfrag{v2}{\sml{$v_{i}$}}
    \psfrag{u3}{\sml{$u_{i+1}$}}
    \psfrag{v3}{\sml{$v_{i+1}$}}
    \psfrag{c23}{\sml{$c_{i,3}$}} 
    \psfrag{W1}{$W_{i,1}$}
    \psfrag{c24}{\sml{$c_{i,4}$}} 
    \psfrag{W2}{$W_{i,2}$}
    \psfrag{c'23}{\sml{$c'_{i,3}$}} 
    \psfrag{W5}{$W_{i,5}$}
    \psfrag{c'24}{\sml{$c'_{i,4}$}} 
    \psfrag{W6}{$W_{i,6}$}
    \psfrag{b21}{\sml{$b_{i,1}$}} 
    \psfrag{W9}{$W_{i,9}$}
    \psfrag{b22}{\sml{$b_{i,2}$}} 
    \psfrag{W10}{$W_{i,10}$}
    \psfrag{b23}{\sml{$b_{i,3}$}} 
    \psfrag{W11}{\hspace{-2mm}$W_{i,11}$}
    \psfrag{b24}{\sml{$b_{i,4}$}} 
    \psfrag{W12}{$W_{i,12}$}
    \psfrag{b'21}{\sml{$b'_{i,1}$}} 
    \psfrag{W13}{$W_{i,13}$}
    \psfrag{b'22}{\sml{$b'_{i,2}$}} 
    \psfrag{W14}{$W_{i,14}$}
    \psfrag{b'23}{\sml{$b'_{i,3}$}} 
    \psfrag{W15}{$W_{i,15}$}
    \psfrag{b'24}{\sml{$b'_{i,4}$}} 
    \psfrag{W16}{$W_{i,16}$}
    \psfrag{c31}{\sml{$c_{i+1,1}$}} 
    \psfrag{W3}{$W_{i,3}$}
    \psfrag{c32}{\sml{$c_{i+1,2}$}} 
    \psfrag{W4}{$W_{i,4}$}
    \psfrag{c'31}{\sml{$c'_{i+1,1}$}} 
    \psfrag{W7}{$W_{i,7}$}
    \psfrag{c'32}{\sml{$c'_{i+1,2}$}} 
    \psfrag{W8}{$W_{i,8}$}
    \includegraphics[scale=1.7]{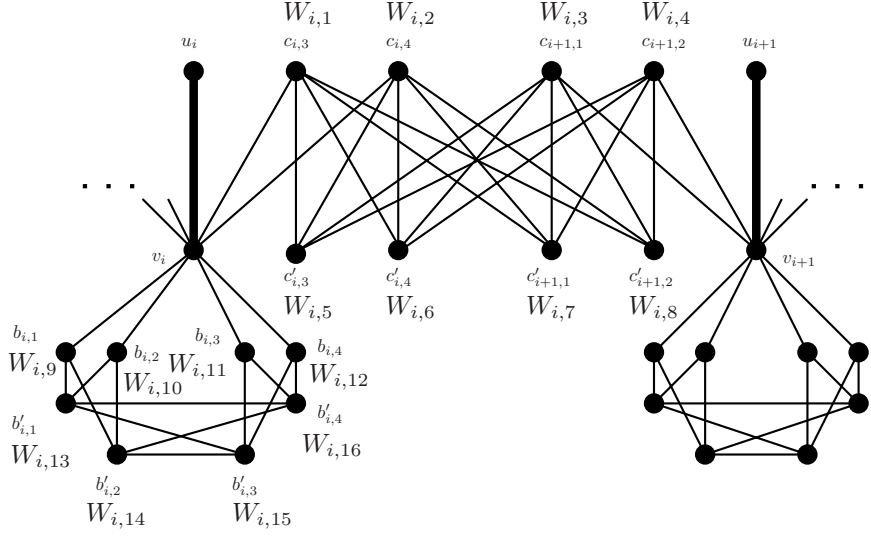}
  \end{center}
  \caption{The partition $W_i=W_{i,1}\dcup\dots\dcup W_{i,8t}$ of
  $G'_i=G[W_i]$ for the special case $t=2$.}
  \label{fig:connect}
  \end{figure}
  To check whether we can apply the connecting lemma observe first that 
  $$
    1\le 2r/\eta_\isubsc{G}\le 2r_1/\eta_\isubsc{G} \le \hat{r}_{\subsc{CL}}
  $$
  by~\eqref{eq:proof:rhat}. For $\ti{y}\in\ti{W}_{i,j}$ with $j\in[8t]$ 
  recall from~\eqref{eq:proof:Xy} (using that each vertex in $H$ has
  neighbours in at most one set~$\ti V_{i'}$, see Figure~\ref{fig:backbone}) that
  \begin{equation}\label{eq:proof:XXy}
    \text{$\ti{X}_{\ti{y}}$ is the set of neighbours of $\ti{y}$ in
      $\ti{V}_i\cup\ti{V}_{i+1}$ and
      set $X_{\ti{y}}:=f_{\subsc{BL}}(\ti{X}_{\ti{y}})$}.
  \end{equation}  
  Then the indexed set system
  $\big(\ti{X}_{\ti{y}}\colon\ti{y}\in\ti{W}_{i,j}\big)$ consists of pairwise
  disjoint sets because $\ti{W}_{i,j}$ is $3$-independent in~$H$
  by~\ref{lem:H:Cindep} of Lemma~\ref{lem:H}. Thus also
  $\big(X_{\ti{y}}:\ti{y}\in\ti{W}_{i,j}\big)$ consists of pairwise disjoint
  sets, as required by Lemma~\ref{lem:CL}. Now let the external degree and the
  candidate set of $\ti{y}\in\ti W_{i,j}$ be defined as in Lemma~\ref{lem:CL},
  i.e., 
  \begin{equation}\label{eq:proof:cand}
    \edeg(\ti{y}):=|X_{\ti{y}}| \qquad\text{and}\qquad
    C(\ti{y}):=\coN_{W_{i,j}}(X_{\ti{y}})\,.
  \end{equation}
  Observe that this implies $C(\ti{y})=W_{i,j}$ if $\ti{X}_{\ti{y}}=\emptyset$
  and hence $X_{\ti{y}}=\emptyset$. 
  Now we will check that conditions \ref{lem:CL:Wi}--\ref{lem:CL:Cdense} of
  Lemma~\ref{lem:CL} are satisfied. From~\ref{lem:G:Ci} of Lemma~\ref{lem:G}
  and~\ref{lem:H:Ci} of Lemma~\ref{lem:H}  
  it follows that
  \begin{align*}
    |W_{i,j}|
      &\geBy{\ref{lem:G:Ci}} \eta'_\isubsc{G}\frac{n}{2r}
      \eqByRef{eq:proof:r} \frac{n}{r_{\subsc{CL}}}
    \qquad\text{and} \\
    |\ti{W}_{i,j}|
      &\leBy{\ref{lem:H:Ci}}
      \eta_\isubsc{H}\frac{m}{2r}\le\eta_\isubsc{H}\frac{n}{2r}
      \eqByRef{eq:proof:r} \frac{\eta_\isubsc{H}}{\eta'_\isubsc{G}} \frac{n}{r_{\subsc{CL}}}
      \leByRef{eq:proof:eta}\xi_{\subsc{CL}} \frac{n}{r_{\subsc{CL}}}
  \end{align*}
  and thus we have condition~\ref{lem:CL:Wi}. By~\ref{lem:H:Cindep} of
  Lemma~\ref{lem:H} we also get condition~\ref{lem:CL:indep} of
  Lemma~\ref{lem:CL}. Further, it follows from~\ref{lem:H:deg} of
  Lemma~\ref{lem:H} that $\edeg(\ti{y})=\edeg(\ti{y}')$ and
  $\ldeg(\ti{y})=\ldeg(\ti{y}')$ for all $\ti{y},\ti{y}'\in\ti{W}_{i,j}$ with
  $j\in[8t]$. In addition $\Delta(H)\le\Delta$ and hence
  \begin{equation*}\begin{split}
    \deg_{H_i'}(\ti y) + \edeg(\ti y)
    & \eqByRef{eq:proof:cand} |N_{\ti W_i}(\ti y)|+|X_{\ti y}| \\
    & \eqByRef{eq:proof:XXy} 
      |N_{\ti W_i}(\ti y)|+|N_{\ti V_i\cup\ti V_{i+1}}(\ti y)|
    \le\deg_H(\ti y)\le\Delta
  \end{split}
  \end{equation*}
  and thus condition~\ref{lem:CL:deg} of Lemma~\ref{lem:CL} is
  satisfied.
  To check conditions~\ref{lem:CL:Cbig} and~\ref{lem:CL:Cdense} of
  Lemma~\ref{lem:CL} observe that for all $\ti{y}\in\ti{C}'_{i',j}$ with 
  $i'\in\{i,i+1\}$ and $j\in[2t]$ we have
  $C(\ti{y})=C'_{i',j}$ as $\ti{y}$ has no neighbours in $\ti{V}_i$ or
  $\ti{V}_{i+1}$ and hence the external $\edeg(\ti{y})=0$
  (see~\eqref{eq:proof:XXy} and~\eqref{eq:proof:cand}).
  Thus~\ref{lem:CL:Cbig} is
  satisfied for $\ti{y}\in\ti{C}'_{i',j}$, and similarly for
  $\ti{y}\in\ti{B'}_{i',j}$. 
  For all $\ti{y}\in\ti{C}_{i,j}$ with $t<j\le 2t$ on the other hand
  we have $\ti{X}_{\ti{y}}\subset N_{\ti{y}}\in\binom{\ti{V}_{i}}{\Delta}$
  by~\eqref{eq:proof:Xy}.
  Recall that $N_{\ti{y}}$ was a special $\Delta$-set in the application of
  the restricted blow-up lemma on $G_i=(U_i,V_i)$ and
  $H_i=H[\ti{U_i}\dcup\ti{V_i}]$ owing to~\eqref{eq:proof:Hi}. Therefore
  $N_{\ti{y}}$ is not mapped to a forbidden $\Delta$-set in
  $\hyper{B}_{i}\subset\binom{V_i}{\Delta}$ by $f_{\subsc{BL}}$ and thus,
  by~\eqref{eq:proof:bad}, to no
  $\Delta$-set in 
  $\Bad{\Delta-1}{\eps_{\subsc{CL}}}{d}(V_i,C_{i,j},C'_{i',j'})\times V_i$
  with $i'\in\{i,i+1\},j,j'\in[2t]$ and $(c_{i,j},c'_{i',j'})\in\MON_{r,t}$.
  We infer that the set $f_{\subsc{BL}}(\ti{X}_{\ti{y}})= 
  X_{\ti{y}}\in\binom{V_{i(\ti{y})}}{\edeg(\ti{y})}$ satisfies
  $|\coN_{C_{i,j}}(X_{\ti{y}})| \ge(d-\eps_{\rm
  CL})^{\edeg(\ti{y})}p^{\edeg(\ti{y})}|C_{i,j}|$ and is such that
  \begin{multline}\label{eq:proof:dense}
    (\coN_{C_{i,j}}(X_{\ti{y}}),C'_{i',j'}) \text{ is $(\eps_{\subsc{CL}},d,p)$-dense} \\
    \text{for all $i'\in\{i,i+1\},j,j'\in[2t]$ with
    $(c_{i,j},c'_{i',j'})\in\MON_{r,t}$}.
    \qquad
  \end{multline}
  Since we chose $C(\ti{y})=\coN(X_{\ti{y}})\cap{C_{i,j}}$
  in~\eqref{eq:proof:cand} we get condition~\ref{lem:CL:Cbig} of
  Lemma~\ref{lem:CL} also for $\ti{y}\in\ti{C}_{i,j}$ with $t<j\le 2t$.
  For $\ti{y}\in\ti{C}_{i+1,j}$ with $j\in[t]$
  the same argument applies with 
  $\ti{X}_{\ti{y}}\subset N_{\ti{y}}\in\binom{\ti{V}_{i+1}}{\Delta}$,
  and for $\ti{y}\in\ti{B}_{i,j}$
  with $j\in[2t]$
  the same argument applies with 
  $\ti{X}_{\ti{y}}\subset N_{\ti{y}}\in\binom{\ti{V}_{i}}{\Delta}$.
  
  Now it will be easy to see that we get~\ref{lem:CL:Cdense} of
  Lemma~\ref{lem:CL}. Indeed, recall again that $C(\ti{y})=C'_{i',j'}$ for all
  $\ti{y}\in\ti{C}'_{i',j'}$ and $C(\ti{y})=B'_{i',j'}$ for all $\ti{y}\in\ti{B}'_{i',j'}$ with
  $i'\in\{i,i+1\}$ and $j\in[2t]$. In addition, the mapping $h$ constructed by
  Lemma~\ref{lem:H} is a homomorphism from $H$ to $\MON_{r,t}$. Hence~\eqref{eq:proof:dense} and 
  property~\ref{lem:G:reg} of Lemma~\ref{lem:G} assert
  that condition~\ref{lem:CL:Cdense} of Lemma~\ref{lem:CL} is satisfied for 
  all edges $\ti{y}\ti{y}'$ of $H'_i=H[\ti{W}_i]$ with at least one end, say
  $\ti{y}$, in a cluster $\ti C'_{i',j'}$ or $\ti B'_{i',j'}$. This is
  true because then $C(\ti{y})=W_{i,k}$ where $\ti{W}_{i,k}$ is the cluster
  containing $\ti{y}$, and $C(\ti{y}')=\coN(X_{\ti{y}'})\cap {W_{i,k'}}$ where
  $\ti{W}_{i,k'}$ is the cluster containing $\ti{y}'$.
  Moreover, since $h$ is a homomorphism all edges $\ti{y}\ti{y}'$ in
  $H'_i=H[\ti{W}_i]$ have at least one end in a cluster $\ti C'_{i',j'}$ or
  $\ti B'_{i',j'}$.

  So conditions \ref{lem:CL:Wi}--\ref{lem:CL:Cdense} are satisfied and we can
  apply Lemma~\ref{lem:CL} to get embeddings of $H'_i=H[\ti{W}_i]$ into
  $G'_i=G[W_i]$ for all $i\in[r]$ that map vertices $\ti{y}\in\ti{W}_i$ (i.e.\
  connecting and balancing vertices) to vertices $y\in W_i$ in their candidate
  sets~$C(\ti y)$. Let $f_{\subsc{CL}}$ be the united embedding resulting from
  these $r$ applications of the connection lemma and denote the embedding that unites
  $f_{\subsc{BL}}$ and $f_{\subsc{CL}}$ by $f$.
  
  To finish the proof we verify that $f$ is an embedding of $H$ into $G$.
  Let $\ti{x}\ti{y}$ be an edge of $H$. By definition of the spin graph
  $\MON_{r,t}$ and since the mapping~$h$ constructed by  
  Lemma~\ref{lem:H} is a homomorphism from $H$ to
  $\MON_{r,t}$ we only need to distinguish the following cases for $i\in[r]$
  and $j,j'\in[2t]$ (see also Figure~\ref{fig:backbone}):
  \begin{enumerate}[leftmargin=*, label={\rm case \arabic{*}:}]
    \item If $\ti{x}\in\ti{V}_i$ and $\ti{y}\in\ti{U}_i$,
      then $f(\ti{x})=f_{\subsc{BL}}(\ti{x})$ and
      $f(\ti{y})=f_{\subsc{BL}}(\ti{y})$ and thus the constrained blow-up lemma
      guarantees  that $f(\ti{x})f(\ti{y})$ is an edge of $G_i$.
    \item If $\ti{x}\in\ti{W}_i$ and $\ti{y}\in\ti{W}_i$,
      then
      $f(\ti{x})=f_{\subsc{CL}}(\ti{x})$ and $f(\ti{y})=f_{\subsc{CL}}(\ti{y})$
      and thus the connection lemma guarantees that $f(\ti{x})f(\ti{y})$ is an
      edge of $G'_i$.
    \item If $\ti{x}\in\ti{V}_i$ and $\ti{y}\in\ti{W}_i$, then either
      $\ti{y}\in\ti{C}_{i,j}$ or $\ti{y}\in\ti{B}_{i,j}$ for some $j$.
      Moreover, $f(\ti{x})=f_{\subsc{BL}}(\ti{x})$ and therefore
      by~\eqref{eq:proof:cand} the candidate set $C(\ti{y})$ of $\ti{y}$ satisfies 
      $C(\ti{y})\subset N_{C_{i,j}}(f(\ti{x}))$ or $C(\ti{y})\subset
      N_{B_{i,j}}(f(\ti{x}))$, respectively. As $f(\ti{y})=f_{\subsc{CL}}(\ti{y})\in
      C(\ti{y})$ we also get that $f(\ti{x})f(\ti{y})$ is an edge of $G$ in
      this case.
  \end{enumerate}
  It follows that $f$ maps all edges of $H$ to edges of $G$, which finishes the
  proof of the theorem.
\end{proof}


\section{A \texorpdfstring{$p$}{p}-dense partition of~\texorpdfstring{$G$}{G}}\label{sec:G}

For the proof of the Lemma for~$G$ we shall apply the minimum degree
version of the sparse regularity lemma (Lemma~\ref{lem:reduced}).
Observe that this lemma guarantees
that the reduced graph of the regular partition we obtain is dense. Thus we
can apply Theorem~\ref{thm:bandwidth} to this
reduced graph.
In the proof of Lemma~\ref{lem:G} we use this theorem to find a copy of the
ladder $\LDR_r$ in the reduced graph (the graphs $\LDR_r$ and $\MON_{r,t}$ are
defined in Section~\ref{def:spin} on
page~\pageref{def:spin}, see also Figure~\ref{fig:backbone}).
Then we further partition the clusters in this ladder to obtain a regular
partition whose reduced graph contains a spin graph $\MON_{r,t}$. 
Recall that this partition will consist of a series of so-called
\emph{big clusters} which we denote by $U_i$ and $V_i$, and a
series of smaller clusters called \emph{balancing
clusters} $B_{i,j}$, $B'_{i,j}$ and \emph{connecting
clusters} $C_{i,j}$, $C'_{i,j}$ with $i\in[r]$, $j\in[2t]$. We will now give
the details.

\begin{proof}[Proof of Lemma~\ref{lem:G}]
  Given $t$, $r_0$, $\eta_\isubsc{G}$, and $\gamma$ choose $\eta'_\isubsc{G}$
  such that
  \begin{equation}\label{eq:G:eta}
    \frac{\eta_\isubsc{G}}5+\left(\frac{4}{\gamma}+2\right)t\cdot\eta'_\isubsc{G}\le\frac{\eta_\isubsc{G}}{2}
  \end{equation}
  and set $d:=\gamma/4$. 
  Apply Theorem~\ref{thm:bandwidth} with input $r_\subsc{BK}:=2$,
  $\Delta=3$ and $\gamma/2$ to obtain the constants 
  $\beta$ and $k_\subsc{BK}:=n_0$.
%
  For input $\eps$ set
  \begin{equation}\label{eq:lem:G:r0}
    r'_0:=\max\{2r_0+1,k_\subsc{BK},3/\beta,6/\gamma,2/\eps,10/\eta_\isubsc{G}\}      
  \end{equation}
  and choose
  $\eps'$ such that
  \begin{equation}\label{eq:G:eps}
    \eps'/\eta'_\isubsc{G}\le\eps/2, \qquad\text{and}\qquad
    \eps'\le\min\{\gamma/4,\eta_\isubsc{G}/10\}.
  \end{equation}
  Lemma~\ref{lem:reduced} applied with
  $\alpha:=\frac12+\gamma$, $\eps'$, $r'_0$ then gives us the missing constant
  $r_1$.

  Assume that~$\Gamma$ is a typical graph from~$\Gnp$ with
  $\log^4n/(pn)=\smallo(1)$, in the sense that it satisfies the conclusion of
  Lemma~\ref{lem:reduced}, and let $G=(V,E)\subset\Gamma$ satisfy
  $\deg_G(v)\ge(\frac12+\gamma)\deg_\Gamma(v)$ for all $v\in V$.
  Lemma~\ref{lem:reduced} applied with $\alpha=\frac12+\gamma$, $\eps'$,
  $r'_0$, and $d$ to~$G$ gives us an $(\eps',d,p)$-dense partition $V=V'_0\dcup V'_1
  \dcup\dots\dcup V'_{r'}$ of $G$ with reduced graph $R'$ with 
  $|V(R')|=r'$ such that
  $2r_0+1\le
  r'_0\le r'\le r_1$ and with minimum degree at least
  $(\frac12+\gamma-d-\eps')r'\ge(\frac12+\frac\gamma2)r'$ by~\eqref{eq:G:eps}.
  If $r'$ is odd, then set $V_0:=V'_0\dcup V'_{r'}$ and $r:=(r'-1)/2$, otherwise
  set $V_0:=V'_0$ and $r:=r'/2$. Clearly $r_0\le r\le r_1$, the graph
  $R:=R'[2r]$ still has minimum degree at least $(\frac12+\frac\gamma3)2r$ and
  $|V_0|\le\eps'n+(n/r'_0)\le(\eta_\isubsc{G}/5)n$ by the choice of $r'_0$ and $\eps'$. It
  follows from Theorem~\ref{thm:bandwidth} applied with $\Delta=3$ and
  $\gamma/2$ that $R$ contains a copy of the ladder $\LDR_r$ on $2r$ vertices
  ($\LDR_r$ has bandwidth $2\le\beta\cdot2r$ by the choice of
  $r'_0$ in~\eqref{eq:lem:G:r0}).%
\comment{AT1}
Hence we can rename the vertices of the graph
  $R=R'[2r]$ with $u_1,v_1,\dots,u_r,v_r$
  according to the spanning copy of $\LDR_r$.
  This naturally defines an equipartite
  mapping $f$ from $V\setminus V_0$ to the vertices of the ladder $\LDR_r$, where $f$ maps 
  all vertices in some cluster $V_i$ with $i\in [2r]$ to a vertex~$u_{i'}$ or $v_{i'}$ 
  of~ $\LDR_r$ for some index $i'\in [r]$. 
  We will show that subdividing the clusters
  $f^{-1}(x)$ for all $x\in V(\LDR_r)$ will give the desired mapping $g$.
  
  \begin{figure}[ht]
  \begin{center}
    \psfrag{b}{\sml{$B_{i,j}$}}
    \psfrag{b'}{\sml{$B'_{i,j'}$}}
    \psfrag{bt}{\sml{$B_{i,t+j'''}$}}
    \psfrag{bt'}{\hspace{-1mm}\sml{$B'_{i,t+j''}$}}
    \psfrag{u}{$f^{-1}(u_i)$}
    \psfrag{v}{$f^{-1}(v_i)$}
    \psfrag{w}{\hspace{-2mm}$f^{-1}(w_i)$}
    \includegraphics[scale=1.5]{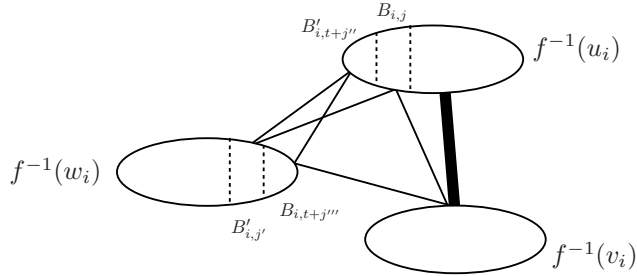}
  \end{center}
  \caption{Cutting off a set of balancing clusters from $f^{-1}(u_i)$ and
    $f^{-1}(w_i)$. These clusters build $p$-dense pairs (thanks to the
    triangle $u_iv_iw_i$ in~$R$) in the form of a $C_5$.}
  \label{fig:G:ballance}
  \end{figure}
  

   We will now construct the
  balancing clusters $B_{i,j}$ and $B'_{i,j}$ with $i\in[r]$, $j\in[2t]$ and
  afterwards turn to the connecting clusters $C_{i,j}$ and $C'_{i,j}$ 
  and big clusters $U_{i}$ and $V_{i}$ 
  with
  $i\in[r]$, $j\in[2t]$.
  
  Notice that $\delta(R)\ge(\frac12+\frac\gamma3)2r$ implies
  that every edge $u_iv_i$ of $\LDR_r\subset R$ is contained in more than
  $\gamma r$ triangles in~$R$. Therefore, we can choose vertices $w_i$ of $R$
  for all $i\in[r]$ such that $u_iv_iw_i$ forms a triangle in~$R$ and no vertex
  of~$R$ serves as $w_i$ more than $2/\gamma$ times. We continue by
  choosing in cluster $f^{-1}(u_i)$ arbitrary disjoint vertex sets $B_{i,1}$,
  \dots, $B_{i,t}$, $B'_{i,t+1}$, \dots, $B'_{i,2t}$\,,
  of size $\eta'_\isubsc{G}n/(2r)$ each, for all $i\in[r]$. We
  will show below that $f^{-1}(u_i)$ is large enough so that these
  sets can be chosen. We then remove all vertices in these sets from $f^{-1}(u_i)$.
  Similarly, we choose in cluster
  $f^{-1}(w_i)$ arbitrary disjoint vertex sets
  $B_{i,t+1}$, \dots,$B_{i,2t}$,
  $B'_{i,1}$, \dots,$B'_{i,t}$\,,  
  of size $\eta'_\isubsc{G}n/(2r)$ each,  for all $i\in[r]$. 
  We also remove these sets from
  $f^{-1}(w_i)$.
  Observe that this construction asserts the following property. For all
  $i\in[r]$ and $j,j',j'',j'''\in[t]$
  each of
  the pairs $(f^{-1}(v_i),B_{i,j})$, $(B_{i,j},B'_{i,j'})$,
  $(B'_{i,j'},B'_{i,t+j''})$, $(B'_{i,t+j''},B_{i,t+j'''})$, and
  $(B_{i,t+j'''},f^{-1}(v_i))$ is a sub-pair of a $p$-dense pair corresponding
  to an edge of $R[\{u_i,v_i,w_i\}]$ (see Figure~\ref{fig:G:ballance}).
  Accordingly this is a sequence of $p$-dense pairs in the form of a $C_5$, as
  needed for the balancing clusters 
  in view of condition~\ref{lem:G:reg}
  (see also Figure~\ref{fig:backbone}). Hence we call the sets $B_{i,j}$
  and $B'_{i,j}$ with $i\in[r]$, $j\in[2t]$ balancing clusters from now on and
  claim that they have the required properties. This claim will be verified below.
  
  We now turn to the construction of the connecting clusters and big clusters.
  Recall that we already removed balancing clusters from all clusters
  $f^{-1}(u_i)$ and possibly from some clusters $f^{-1}(v_i)$ (because $v_i$
  might have served as $w_{i'}$) with $i\in[r]$.
  For each $i\in[r]$ we arbitrarily partition the remaining vertices of cluster
  $f^{-1}(u_i)$ into sets $C_{i,1}\dcup\dots\dcup C_{i,2t} \dcup U_i$ and the 
  remaining vertices of cluster $f^{-1}(v_i)$ into sets
  $C'_{i,1}\dcup\dots\dcup C'_{i,2t}\dcup V_i$ such that
  $|C_{i,j}|,|C'_{i,j}|=\eta'_\isubsc{G} n/(2r)$ for all $i\in[r],j\in[2t]$.
  This gives us
  the connecting and the big clusters and we claim that also these clusters
  have the required properties. Observe, again, 
  that for all $i\in[r]$, $i'\in\{i-1,i,i+1\}\setminus\{0\}$, $j,j'\in[2t]$
  each of the pairs $(U_i,V_i)$, $(C_{i',j},V_i)$, and $(C_{i,j},C'_{i,j'})$
  is a sub-pair of a $p$-dense pair corresponding to an edge of $\LDR_r$
  (see Figure~\ref{fig:G:connect}).
  \begin{figure}[ht]
  \begin{center}
    \psfrag{R}{}
    \psfrag{u1}{$U_{i-1}$}
    \psfrag{v1}{$V_{i-1}$}
    \psfrag{u2}{$U_{i}$}
    \psfrag{v2}{$V_{i}$}
    \psfrag{u3}{$U_{i+1}$}
    \psfrag{v3}{$V_{i+1}$}
    \psfrag{fu1}{\sml{$f^{-1}(u_{i-1})$}}
    \psfrag{fu3}{\sml{$f^{-1}(u_{i+1})$}}
    \psfrag{fv1}{\sml{$f^{-1}(v_{i-1})$}}
    \psfrag{fv3}{\sml{$f^{-1}(v_{i+1})$}}
    \psfrag{c11}{\sml{$C_{i-1,2}$}}
    \psfrag{c12}{\sml{$C_{i,1}$}}
    \psfrag{c'11}{\sml{$C'_{i-1,2}$}}
    \psfrag{c'12}{\sml{$C'_{i,1}$}}
    \psfrag{c21}{\sml{$C_{i,2}$}}
    \psfrag{c22}{\sml{\hspace{-2mm}$C_{i+1,1}$}}
    \psfrag{c'21}{\sml{$C'_{i,2}$}}
    \psfrag{c'22}{\sml{\hspace{-2mm}$C'_{i+1,1}$}}
    \includegraphics[scale=1.0]{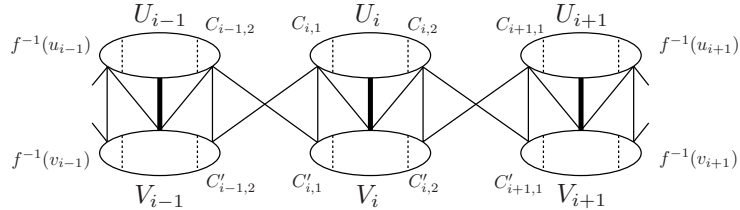}
  \end{center}
  \caption{Partitioning the remaining vertices of cluster $f^{-1}(u_i)$ and
  $f^{-1}(v_i)$ into sets $C_{i,1}\dcup C_{i,2} \dcup U_i$ and $C'_{i,1}\dcup
  C'_{i,2} \dcup V_i$ (for the special case $t=1$).
  These clusters form $p$-dense pairs (thanks to the ladder
  $\LDR_r$ in~$R$) as indicated by the edges.}
\label{fig:G:connect}
\end{figure}

  We will now show that the balancing clusters, connecting clusters and big
  clusters satisfy conditions~\ref{lem:G:Vi}--\ref{lem:G:reg}.
  Note that condition~\ref{lem:G:Ci} concerning the sizes of the
  connecting and balancing clusters is satisfied by construction.
  To determine the sizes of the big clusters observe that from each cluster
  $V'_j$ with $j\in[2r]$ vertices for at most $2t\cdot 2/\gamma$
  balancing clusters were removed. In addition, at most $2t$ connecting
  clusters were split off from $V'_j$.
  Since $|V\setminus V_0| \ge (1-\eta_\isubsc{G}/5)n$ we get
  \begin{equation*}
    |V_i|,|U_i|
      \ge \left(1-\frac{\eta_\isubsc{G}}5\right)\frac{n}{2r}
        -\left(\frac{4}{\gamma}+2\right)t\cdot\eta'_\isubsc{G}\frac{n}{2r}
      \ge(1-\eta_\isubsc{G})\frac{n}{2r}
  \end{equation*}
  by~\eqref{eq:G:eta}.\comment{AT2} This is
  condition~\ref{lem:G:Vi}. It remains to verify
  condition~\ref{lem:G:reg}. It can easily be checked that for all
  $xy\in E(\MON_{r,t})$ the corresponding pair $(g^{-1}(x),g^{-1}(y))$ is a
  sub-pair of some cluster pair $(f^{-1}(x'),f^{-1}(y'))$ with $x'y'\in E(R)$ by
  construction. In addition, all big, connecting, and balancing clusters are of
  size at least $\eta'_\isubsc{G}n/(2r)$. Hence we have $|g^{-1}(x)|\ge\eta'_\isubsc{G}|f^{-1}(x')|$
  and $|g^{-1}(y)|\ge\eta'_\isubsc{G}|f^{-1}(y')|$. 
  We conclude from Proposition~\ref{prop:subpairs} that
  $(g^{-1}(x),g^{-1}(y))$ is $(\eps,d,p)$-dense since $\eps'/\eta'_\isubsc{G}\le\eps$
  by~\eqref{eq:G:eps}. This finishes the verification
  of~\ref{lem:G:reg}.
\end{proof}


\section{A partition of~\texorpdfstring{$H$}{H}}
\label{sec:H}

Hajnal and Szemer\'edi determined the minimum degree that forces a certain
number of vertex disjoint~$K_r$ copies in~$G$. In addition their result
guarantees that the remaining vertices can be covered by copies of $K_{r-1}$. 
Another way to express this, which actually resembles the original formulation,
is obtained by considering the complement~$\bar G$ of~$G$ and its maximum degree. 
Then, so the theorem asserts, the graph~$\bar G$ contains a certain number of
vertex disjoint independent sets of almost equal sizes. In other words, $\bar
G$ admits a vertex colouring such that the sizes of the colour classes differ
by at most $1$. Such a colouring is also called \emph{equitable colouring}.

\begin{theorem}[Hajnal \& Szemer\'edi~\cite{HajSze}]
\label{thm:HajSze}
  Let~$\bar G$ be a graph on~$n$ vertices with maximum degree $\Delta(\bar
  G)\le \Delta$. Then there is an equitable vertex colouring of~$G$ with
  $\Delta+1$ colours.
\qed
\end{theorem}

In the proof of Lemma~\ref{lem:H} that we shall present in this section
we will use this theorem in order to guarantee
property~\ref{lem:H:Cindep}. This will be the very last step in the
proof, however. First, we need to take care of the remaining properties.

Before we start,
let us agree on some terminology that will turn out to be useful 
in the proof of Lemma~\ref{lem:H}. When defining a homomorphism $h$
from a graph $H$ to a graph $R$, we write 
$h(S):=z$ for a set $S$ of vertices in $H$ and a vertex $z$ in $R$
to say that all vertices from $S$ are mapped to $z$.
Recall that we have a bandwidth hypothesis on~$H$. Consider an ordering of the
vertices of~$H$ achieving its bandwidth.
Then we can deal with the vertices of $H$ in this order. 
In particular, we can refer to vertices as the \emph{first}
or \emph{last} vertices in some set, meaning that they are
the vertices with the smallest or largest label from this set. 

We start with the following proposition.

 \begin{figure}
  \begin{center}
    \psfrag{z0}{$z^0$}
    \psfrag{z1}{$z^1$}
    \psfrag{z2}{$z^2$}
    \psfrag{z3}{$z^3$}
    \psfrag{z4}{$z^4$}
    \psfrag{z5}{$z^5$}
    \includegraphics[scale=1]{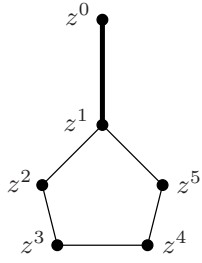}
  \end{center}
  \caption{The graph $\bar R$ in Proposition~\ref{prop:lolly}.}
  \label{fig:lolly}
  \end{figure}

  \begin{proposition}
  \label{prop:lolly} Let $\bar R$ be the following graph with six vertices and six edges: 
  $$
  \bar R:= \left( \{z^0,z^1,\dots,z^5\},\{z^0z^1,z^1z^2, z^2z^3, z^3z^4, z^4z^5, z^5z^1\} \right), 
  $$  
  see Figure~\ref{fig:lolly} for a picture of $\bar R$.
  For every real $\bar \eta>0$ there exists a real $\bar \beta>0$ such that the following holds:
  Consider an arbitrary bipartite graph 
  $\bar H$ with $\bar m$ vertices, colour classes $Z^0$ and $Z^1$, and 
  $\bw(\bar H)\le \bar \beta  \bar m$ and denote by $T$ the union of the first $\bar\beta \bar m$
  vertices and the last $\bar\beta \bar m$ vertices of $H$. Then there exists a 
  homomorphism $\bar h\colon V(\bar H) \to V(\bar R)$ from
  $\bar H$ to $\bar R$ such that for all $j\in \{0,1\}$ and all $k\in [2,5]$
  \begin{align}
  \label{eq:lolly1}
  \frac{\bar m}{2}-5\bar \eta \bar m \le |\bar h^{-1}(z^j)|&\le \frac{\bar
  m}{2}+\bar \eta \bar m\,, \\
  \label{eq:lolly2} 
  |\bar h^{-1}(z^k)|&\le \bar \eta \bar m\,, \\
  \label{eq:lolly3}
  h(T\cap Z^j) &= z^j.
  \end{align}
  \end{proposition}

  Roughly speaking, Proposition~\ref{prop:lolly} shows that we can find a
  homomorphism from a bipartite graph $\bar H$ to a graph $\bar R$ which consists
  of an edge $z^0z^1$ which has an attached 5-cycle 
in such a way that most of the vertices of $\bar H$ are mapped about evenly to
the vertices $z^0$ and $z^1$. If we knew that the colour classes of $\bar H$ were
of almost equal size, then this would be a trivial task, but since this is not
guaranteed, we will have to make use of the additional vertices $z^2,\dots,z^5$.
  
  \begin{proof}[Proof of Proposition~\ref{prop:lolly}] 
  Given $\bar \eta$, choose an integer $\ell\ge 6$ and a real $\bar \beta>0$ such that
  \begin{equation}
  \label{eq:lollybeta}
  \frac{5}{\ell} < \bar\eta \text{ \quad and \quad} \bar \beta:= \frac{1}{\ell^2}.
  \end{equation}
  For the sake of a simpler exposition we assume that $\bar m/\ell$ and $\bar
  \beta \bar m$ are integers. Now consider a graph $\bar H$ as given in the
  statement of the proposition.
  Partition 
  $V(\bar H)$ along the ordering induced by the bandwidth labelling into sets $\bar W_1,\dots,\bar W_\ell$ of 
  sizes $|\bar W_i|=\bar m/\ell$ for $i\in [\ell]$. For each $\bar W_i$, consider 
its last $5\bar \beta \bar m$ vertices and partition them into sets 
$X_{i,1},\dots,X_{i,5}$ of size $|X_{i,k}|=\bar \beta \bar m$. 
For $i\in [\ell]$ let
$$
W_i:= \bar W_i \setminus (X_{i,1}\cup\dots\cup X_{i,5}), \quad
W:= \bigcup_{i=1}^\ell W_i,
$$
and note that
$$
L:=|W_i| = \frac{\bar m}{\ell} - 5\bar\beta \bar m \eqByRef{eq:lollybeta} 
\left( \frac1\ell - \frac5{\ell^2}\right) \bar m \ge \frac{1}{\ell^2} \bar m 
\eqByRef{eq:lollybeta} \bar\beta \bar m.
$$
For $i\in [\ell]$, $j\in\{0,1\}$, and $1\le k\le 5$ let
$$
W_i^j := W_i \cap Z^j, \quad
X_{i,k}^j:= X_{i,k} \cap Z^j.
$$
Thanks to the fact that 
$\bw(\bar H)\le \bar\beta \bar m$, we know that there are no edges between $W_i$ and $W_{i'}$ for 
$i\neq i'\in [\ell]$. 
In a first round, for each $i\in[\ell]$ we will either map all vertices from
$W_i^j$ to $z^j$ for both $j\in\{0,1\}$ (call such a mapping a
\emph{normal} embedding of $W_i$) or we map all vertices
from $W_i^j$ to $z^{1-j}$ for both $j\in\{0,1\}$ (call this an
\emph{inverted} embedding). We will do this in such a
way that the difference between the number of vertices that get sent to $z^0$ and the number of those that get sent to $z^1$ is as small as possible. Since
$|W_i|\le L$ the difference is therefore at most $L$. If, in addition, we
guarantee that both $W_1$ and $W_\ell$ receive a normal embedding, it is at most
$2L$. So, to summarize and to describe the mapping more precisely: there exist
integers $\phi_i\in\{0,1\}$ for all $i\in [\ell]$ such that $\phi_1=0=\phi_\ell$
and the function $h: W \to \{z^0,z^1\}$ defined by
$$
h(W_i^j ):= 
\begin{cases}
z^j     & \text{ if } \phi_i=0, \\        
z^{1-j} & \text{ if } \phi_i=1,         
\end{cases} 
$$  
is a homomorphism from $\bar H[W]$ to $\bar R[\{z^0,z^1\}]$,
satisfying that for both $j\in\{0,1\}$
\begin{equation}
  \begin{split}
|h^{-1}(z^j)| 
&\le \frac{\ell L}{2} + 2L 
= \left(\frac{\ell}{2} +2\right)\frac{\bar m}{\ell} 
  -  \left(\frac{\ell}{2} +2\right)5\bar \beta\bar m \\
&\eqByRef{eq:lollybeta}
   \frac{\bar m}{2} + \bar m\left( \frac2\ell - \frac52 \frac1\ell 
   - \frac{10}{\ell^2}\right) \le \frac{\bar m}{2}\,.
\label{eq:hz0}  
  \end{split}
\end{equation}

In the second round we extend this homomorphism to the vertices in the classes 
$X_{i,k}$. Recall that these vertices are by definition situated after those 
in $W_i$ and before those in $W_{i+1}$.
The idea for the extension is simple. If $W_i$ and $W_{i+1}$ have been embedded in the same way by $h$ (either both normal or both inverted), then we 
map all the vertices from all $X_{i,k}$ to $z^0$ and $z^1$ accordingly. 
If they have been embedded in different ways (one normal and one inverted), then we walk around the 5-cycle 
$z^1,\dots,z^5,z^1$ to switch colour classes. 

Here is the precise definition. Consider an arbitrary $i\in [\ell]$.
Since $h(W_i^0)$ and $h(W_i^1)$ are already defined, choose (and fix) $j\in\{0,1\}$ in such a way that $h(W_i^j)=z^1$. Note that this implies that $h(W_i^{1-j})=z^0$. 
Now define $h_i: \bigcup_{k=0}^5 X_{i,k} \to \bigcup_{k=1}^5 \{z^k\}$ as follows:

Suppose first that $\phi_i=\phi_{i+1}$. Observe that in this case we must also have 
$h(W_{i+1}^j)=z^1$ and $h(W_{i+1}^{1-j})=z^0$. So we can happily define for all $k\in [5]$
$$
h_i(X_{i,k}^j) = z^1 \text{ \quad and \quad } h_i(X_{i,k}^{1-j}) = z^0.
$$


Now suppose that $\phi_i\neq\phi_{i+1}$. 
Since we are still assuming that $j$ is such that $h(W_i^j)=z^1$ and thus $h(W_i^{1-j})=z^0$, 
the fact that $\phi_i\neq\phi_{i+1}$ implies that $h(W_{i+1}^j)=z^0$ and $h(W_{i+1}^{1-j})=z^1$. 
In this case we define $h_i$ as follows:
$$
\begin{array}{l|l|l|l|l|l|l}
h(W_i^{1-j}) & h_i(X_{i,1}^{1-j}) & h_i(X_{i,2}^{1-j}) & h_i(X_{i,3}^{1-j}) 
                 & h_i(X_{i,4}^{1-j}) & h_i(X_{i,5}^{1-j}) & h(W_{i+1}^{1-j}) \\
=z^0 & :=z^2 & :=z^2 & :=z^4 
     & :=z^4 & :=z^1 & =z^1 \\
\hline 
h(W_i^j) & h_i(X_{i,1}^j) & h_i(X_{i,2}^j) & h_i(X_{i,3}^j) 
             & h_i(X_{i,4}^j) & h_i(X_{i,5}^j) & h(W_{i+1}^j) \\
=z^1 & :=z^1 & :=z^3 & :=z^3 
             & :=z^5 & :=z^5 & =z^0 \\
\end{array}
$$

Finally, we set $\bar h: V(\bar H)\to V(\bar R)$ by letting
$\bar h(x):= h(x)$ if $x\in W_i$ for some $i\in [\ell]$ and 
$\bar h(x):= h_i(x)$ if $x\in X_{i,k}$ for some $i\in [\ell]$ and
$k\in [5]$.

In order to verify that this is a
 homomorphism from $\bar H$ to the sets $\bar R$, 
 we first let 
$$
X_{i,0}^0:=W_i^0, X_{i,0}^1:=W_i^1, X_{i,6}^0:=W_{i+1}^0, X_{i,6}^1:=W_{i+1}^1. 
$$
Using this notation, it is clear that 
any edge $xx'$ in $\bar H[W_i\cup \bigcup_{k=1}^5 X_{i,k} \cup W_{i+1}]$
with $x\in Z^j$ and $x'\in Z^{1-j}$ is of the form 
$$
xx' \in 
(X_{i,k}^j \times X_{i,k}^{1-j}) \cup
(X_{i,k}^j \times X_{i,k+1}^{1-j}) \cup
(X_{i,k+1}^j \times X_{i,k}^{1-j})
$$
for some $k\in [0,6]$.   
It is therefore easy to check in the above table that $\bar h$ maps $xx'$ to an edge of $R$.


We conclude the proof by showing that the cardinalities of the preimages of the vertices in $R$ match the required sizes. In the second round we mapped a total of   
$$
\ell \cdot 5\bar\beta \bar m \eqByRef{eq:lollybeta} \frac{5}{\ell} \bar m 
\leByRef{eq:lollybeta} \bar\eta \bar m 
$$
additional vertices from $\bar H$ to the vertices of $\bar R$, which guarantees that
  $$
  |\bar h^{-1}(z^j)|\leByRef{eq:hz0} \frac{\bar m}{2} + \bar \eta\bar m 
\text{ for all } j\in \{0,1\}, \quad 
  |\bar h^{-1}(z^k)|\le \bar \eta \bar m \text{ for all } k\in [2,5].
  $$
Finally, the lower bound in (\ref{eq:lolly1}) immediately follows from the upper bounds:
$$
|\bar h^{-1}(z^j)| \ge \bar m - |\bar h^{-1}(z^{1-j})| - \sum_{k=2}^5|\bar h^{-1}(z^k)|
\ge \frac{\bar m}{2} - 5\bar\eta \bar m.  
$$  
\end{proof}

We remark that Proposition~\ref{prop:lolly} (and thus Lemma~\ref{lem:H}) would remain true if we replaced the 5-cycle in 
$\bar R$ by a 3-cycle. However, we need the properties of the 5-cycle in the proof of the main
theorem. 
Now we will prove Lemma~\ref{lem:H}.

\begin{proof}[Proof of Lemma~\ref{lem:H}.]
  Given the integer $\Delta$, set $t:=(\Delta+1)^3(\Delta^3+1)$. 
  Given a real $0<\eta_{\isubsc{H}}<1$
  and integers $m$ and $r$, set $\bar \eta:= \eta_\isubsc{H}/20<1/20$ and apply 
  Proposition~\ref{prop:lolly} to 
  obtain a real $\bar \beta>0$. Choose $\beta>0$ sufficiently small so that all the inequalities
  \begin{equation}
  \label{eq:lemHbeta}
  \frac1r - 4\beta \ge \beta / \bar\beta, \quad 
   4\beta r \le \frac{\eta_\isubsc{H}}{20r}, \quad
   16\Delta\beta r \le
   \eta_\isubsc{H}\left(\frac1r-4\beta\right)\left(\frac12-5\bar\eta\right)
  \end{equation}
  hold. Again, we assume that $m/r$ and $\beta m$ are integers.

  Next we consider the spin graph $R_{r,t}$ with $t=1$, i.e., let
  $R:=R_{r,1}$. For the sake of simpler reference, we will change the  names of
  its vertices as follows: For all $i\in [r]$ we set (see Figure~\ref{fig:lemH})
  $$
  z_{i}^{0}:=u_i,
  z_{i}^{1}:=v_i,
  z_{i}^{2}:=b_{i,1},
  z_{i}^{3}:=b'_{i,1},
  z_{i}^{4}:=b'_{i,2},
  z_{i}^{5}:=b_{i,2},
  $$  
  $$
  q_{i}^2:=c_{i,1},
  q_{i}^3:=c'_{i,1},
  q_{i}^4:=c_{i,2},
  q_{i}^5:=c'_{i,2}.
  $$
  Note that for every $i\in [r]$ the graph $R[\{z_i^0,\dots,z_i^5\}]$ is isomorphic to the 
  graph $\bar R$ defined in Proposition~\ref{prop:lolly}. 

 \begin{figure}
  \begin{center}
    \psfrag{u2}{\sml{$u_{i}$}}
    \psfrag{v2}{\sml{$v_{i}$}}
    \psfrag{u3}{\sml{$u_{i+1}$}}
    \psfrag{v3}{\sml{$v_{i+1}$}}
    \psfrag{b21}{\sml{$b_{i,1}$}}
    \psfrag{b22}{\sml{$b_{i,2}$}}
    \psfrag{b'21}{\sml{$b'_{i,1}$}}
    \psfrag{b'22}{\sml{$b'_{i,2}$}}
    \psfrag{c21}{\sml{$c_{i,1}$}}
    \psfrag{c22}{\sml{$c_{i,2}$}}
    \psfrag{c'21}{\sml{$c'_{i,1}$}}
    \psfrag{c'22}{\sml{$c'_{i,2}$}}
    \psfrag{c31}{\sml{$c_{i+1,1}$}}
    \psfrag{c32}{\sml{$c_{i+1,2}$}}
    \psfrag{c'31}{\sml{$c'_{i+1,1}$}}
    \psfrag{c'32}{\sml{$c'_{i+1,2}$}}
    \psfrag{b31}{\sml{$b_{i+2,1}$}}
    \psfrag{b32}{\sml{$b_{i+2,2}$}}
    \psfrag{b'31}{\sml{$b'_{i+2,1}$}}
    \psfrag{b'32}{\sml{$b'_{i+2,2}$}}
    \psfrag{z0}{$z_i^0$}
    \psfrag{z1}{$z_i^1$}
    \psfrag{z2}{$z_i^2$}
    \psfrag{z3}{$z_i^3$}
    \psfrag{z4}{$z_i^4$}
    \psfrag{z5}{$z_i^5$}
    \psfrag{z0+}{$z_{i+1}^0$}
    \psfrag{z1+}{$z_{i+1}^1$}
    \psfrag{z2+}{$z_{i+1}^2$}
    \psfrag{z3+}{$z_{i+1}^3$}
    \psfrag{z4+}{\hspace{-2mm}$z_{i+1}^4$}
    \psfrag{z5+}{$z_{i+1}^5$}
    \psfrag{q2}{$q_i^2$}
    \psfrag{q3}{$q_i^3$}
    \psfrag{q4}{$q_i^4$}
    \psfrag{q5}{$q_i^5$}
    \psfrag{q2+}{$q_{i+1}^2$}
    \psfrag{q3+}{$q_{i+1}^3$}
    \includegraphics[scale=1.5]{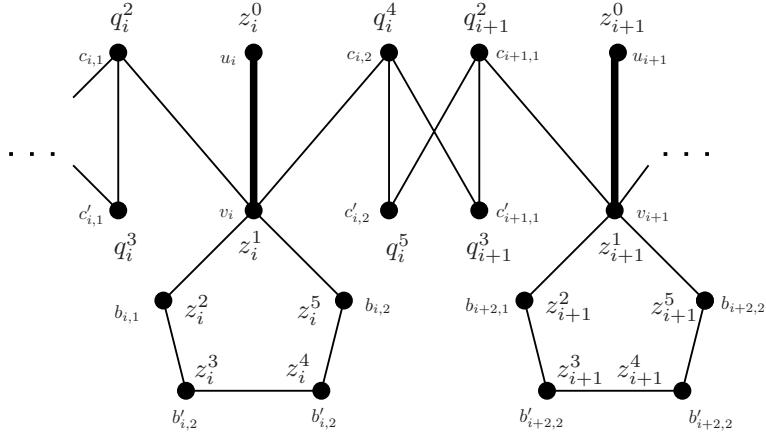}
  \end{center}
  \caption{The subgraph $R[\{z_i^0,\dots,z_i^5,q_i^4, q_i^5,
  q_{i+1}^2,q_{i+1}^3, z_{i+1}^0,\dots,z_{i+1}^5\}]$ of $\MON_{r,1}$ in the
  proof of Lemma~\ref{lem:H}.}
  \label{fig:lemH}
  \end{figure}

   Partition 
  $V(H)$ along the ordering (induced by the bandwidth labelling) into sets $\bar S_1,\dots,\bar S_r$ of 
  sizes $|\bar S_i|=m/r$ for $i\in [r]$. 

  Define sets $T_{i,k}$ for $i\in [r]$ and $k\in [0,5]$ with $|T_{i,k}|=\beta m$ such that
  $T_{i,0}\cup\dots\cup T_{i,4}$ contain the last $5\beta m$ vertices of $\bar S_i$ and 
  $T_{i,5}$ the first $\beta m$ vertices of $\bar S_{i+1}$ (according to the ordering).
  Set $S_i:= \bar S_i \setminus (T_{i,1}\cup\dots\cup T_{i,4})$ and observe that this implies
  that $T_{i,0}$ is the set of the last $\beta m$ vertices of $S_{i}$ and $T_{i,5}$ is the 
  set of the first $\beta m$ vertices in $\bar S_{i+1}$.  
  Set 
  \begin{equation}
  \label{eq:barm}
  \bar m:= |S_i| = (m/r) - 4\beta m = \left(\frac1r - 4\beta\right)m 
  \geByRef{eq:lemHbeta} \beta m/ \bar \beta,
  \text{ \quad thus \quad }
  \bar \beta \bar m \ge \beta m.
  \end{equation}
  Denote by $Z^0$ and $Z^1$ the two colour classes of the bipartite graph $H$.
  For $i\in [\ell]$, $k\in [0,5]$ and $j\in [0,1]$ let
  $$
  S_i^j := S_i \cap Z^j, \quad
  T_{i,k}^j:= T_{i,k} \cap Z^j.
  $$

  Now for each $i\in [r]$ apply Proposition~\ref{prop:lolly} to $\bar H_i:= H[S_i]$ and 
  $\bar R_i:=R[\{z_i^0,\dots,z_i^5\}]$.
  Observe that 
  $$\bw(\bar H_i) \le \bw(H) \le \beta m \leByRef{eq:barm} \bar\beta\bar m,$$
  so we obtain a homomorphism 
  $\bar h_i: S_i \to \{z_i^0,\dots,z_i^5\}$ of $\bar H_i$ to $\bar R_i$. Combining these yields a
  homomorphism 
  $$
  \bar h: \bigcup_{i=1}^{r} S_i \to \bigcup_{i=1}^{r} \{z_i^0,\dots,z_i^5\},
  $$
  $$
  \text{ \quad from \quad }
  H[\bigcup_{i=1}^{r} S_i] 
  \text{ \quad to \quad } 
  R[\bigcup_{i=1}^{r} \{z_i^0,\dots,z_i^5\}]
  $$
  with the property that for every $i\in [r]$, $j\in [0,1]$ and $k\in [2,5]$
  \begin{equation*}
  \label{eq:preim}
  \begin{split}
  \frac{\bar m}{2} -5 \bar \eta \bar m 
  \leByRef{eq:lolly1} |\bar h^{-1}(z_i^j)| 
  &\leByRef{eq:lolly1}\frac{\bar m}{2} + \bar \eta \bar m 
  \le \left(1+\frac{\eta_\isubsc{H}}{10}\right)\frac{m}{2r}
  \text{ \quad and} \\
  |\bar h^{-1}(z_i^k)| 
  &\leByRef{eq:lolly2} \bar \eta \bar m 
  \le \frac{\eta_\isubsc{H}}{10}\frac{m}{2r}.  
  \end{split}
  \end{equation*}
  Thanks to (\ref{eq:barm}), we know that $\bar\beta \bar m \ge \beta m$, and therefore applying
  the information from (\ref{eq:lolly3}) in Proposition~\ref{prop:lolly} yields
  that for all $i\in [r]$ and $j\in [0,1]$
  $$
  \bar h(T_{i,0}^j)=z_i^j \text{ \quad and \quad } \bar h(T_{i,5}^j)=z_{i+1}^j.
  $$
  
  In the second round, our task is to extend this homomorphism to the vertices in $\bar S_i \setminus S_i$ by
     defining a function 
  $$
  h_i: T_{i,1}\cup\dots\cup T_{i,4} \to \{z_i^1,q_i^4,q_i^5,q_{i+1}^2,q_{i+1}^3,z_{i+1}^1\}
  $$
  for each $i\in [r]$ as follows: 
  
{\footnotesize
$$
\begin{array}{l|l|l|l|l|l}
\bar h(T_{i,0}^0)=z_i^0     & h_i(T_{i,1}^{0}):=q_i^4 & h_i(T_{i,2}^{0}):=q_{i}^4 &
   h_i(T_{i,3}^0):=q_{i+1}^2  & h_i(T_{i,4}^{0}):=q_{i+1}^2 & \bar h(T_{i,5}^{0})=z_{i+1}^0 \\[1mm]
\hline \\[-2mm]
\bar h(T_{i,0}^1)=z_i^1     & h_i(T_{i,1}^{1}):=z_i^1 & h_i(T_{i,2}^{1}):=q_{i}^5 &
   h_i(T_{i,3}^1):=q_{i+1}^3  & h_i(T_{i,4}^{1}):=z_{i+1}^1 & \bar h(T_{i,5}^{1})=z_{i+1}^1 \\
\end{array}
$$
}%
Now set $h(x):= \bar h(x)$ if $x\in S_i$ for some $i\in [r]$ and 
$h(x):=h_i(x)$ if $x\in T_{i,k}$ for some $i\in [r]$ and $k\in [4]$. 

Let us verify that $h$ is a homomorphism from $H$ to $R$. For edges $xx'$ with
both endpoints inside a set $S_i$ we do not need to check anything because here
$h(x)=\bar h(x)$ and $h(x')=\bar h(x')$ and we know from
Proposition~\ref{prop:lolly} that $\bar h$ is a homomorphism. Due to the
bandwidth condition $\bw(H)\le \beta m$, any other edge $xx'$ with $x\in Z^0$ and
$x'\in Z^1$ is of the form
$$
xx' \in 
(T_{i,k}^0 \times T_{i,k}^{1}) \cup
(T_{i,k}^0 \times T_{i,k+1}^{1}) \cup
(T_{i,k+1}^0 \times T_{i,k}^{1})
$$
for some $i\in [\ell]$ and $0\le k,k+1 \le 5$.  
It is therefore easy to check in the above table 
that $h$ maps $xx'$ to an edge of $R$. 

What can we say about the cardinalities of the preimages? 
In the second round we have mapped $4\beta m r$
additional vertices from $H$ to vertices in $R$, hence for 
any vertex $z$ in $R$ with $z\not\in\{z_i^0,z_i^1\}$, $i\in[\ell]$ we have
\begin{equation}
\label{eq:H:X}
|h_i^{-1}(z)| \le 4\beta m r \leByRef{eq:lemHbeta} \frac{\eta_\isubsc{H}}{10}\frac{m}{2r},
\end{equation}
and therefore the required upper bounds immediately follow from (\ref{eq:preim}).

At this point we have found a homomorphism $h$ from $H$ to $R=\MON_{r,1}$ 
of which we know that it satisfies properties \ref{lem:H:Vi} and \ref{lem:H:Ci}.

So far we have been working with the graph $R=\MON_{r,1}$, and therefore 
we know which vertices have been mapped to $u_i=z_i^0$ and $v_i=z_i^1$:
$$
\ti{U_i} := h^{-1}(u_i) = h^{-1}(z_i^0) \text{ \quad and \quad }
\ti{V_i} := h^{-1}(v_i) = h^{-1}(z_i^1). 
$$
Moreover for $i\in[r]$ and $k\in [2,5]$ set
$$
Z_i^k:= h^{-1}(z_i^k) \text{ \quad and \quad } Q_i^k:=h^{-1}(q_i^k).
$$


Let us 
deal with property \ref{lem:H:X} next. By definition, a vertex in 
$\ti{X}_i \subset \ti{V}_i$ must have at least one neighbour in 
$Q_i^2$ or $Q_i^4$ or $Z_i^2$ or $Z_i^5$.
We know from (\ref{eq:H:X}) that the two latter sets contain at most $4\beta m r$ vertices each,
and each of their vertices has at most $\Delta$ neighbours. Thus
\begin{eqnarray*}
|\ti{X}_i| &\le& \Delta \cdot 16\beta m r \leByRef{eq:lemHbeta} 
\eta_\isubsc{H} \left(\frac1r - 4\beta\right)\left(\frac12-5\bar\eta\right)m 
\eqByRef{eq:barm} \eta_\isubsc{H} \bar m \left(\frac12-5\bar\eta\right)\\
&\leByRef{eq:preim}&  \eta_\isubsc{H} |\bar h^{-1}(z_i^1)| \le \eta_\isubsc{H} |h^{-1}(z_i^1)| = \eta_\isubsc{H} |\ti{V}_i|,
\end{eqnarray*}
which shows that \ref{lem:H:X} is also satisfied.

Next we would like to split up the sets $Z_i^k$ and $Q_i^k$ for $i\in [r]$ and $k\in [2,5]$
into smaller sets 
in order to meet
the additional requirements \ref{lem:H:Cindep} and \ref{lem:H:deg}.
This means that we need to partition them further into sets 
of vertices which have no path of length 1, 2, or 3 between them
and which have the same degree into certain sets.

To achieve this, first denote by $H^3$ the 3rd power of $H$. Then an upper bound on 
the maximum degree of $H^3$ is obviously given by 
$$
\Delta+\Delta(\Delta-1) + \Delta(\Delta-1)(\Delta-1) \le \Delta^3.
$$
Hence $H^3$ has a vertex colouring $c:V(H)\to \mathbb{N}$ with at most 
$\Delta^3+1$ colours. 
Notice that a set of vertices that receives the same colour by $c$ forms a 3-independent set in 
$H$. 
To formalize this argument, we define a `fingerprint' function
$$
f:\bigcup_{i=1}^r \bigcup_{k=2}^5 (Z_i^k\cup Q_i^k) 
\to [0,\Delta] \times  [0,\Delta] \times [0,\Delta] \times [\Delta^3+1]
$$
as follows:
$$
f(y):=\begin{cases}
\left(\deg_{\ti{V}_i}(y),\deg_{Q_i^2\cup Q_i^4}(y),\deg_{Z_i^2}(y),c(y)\right) 
& \text{if } 
y \in \left(\bigcup_{k=2}^5 (Q_i^k \cup Z_i^k) \right) \setminus Z_i^4, \\
\left(\deg_{\ti{V}_i}(y),\deg_{Q_i^2\cup Q_i^4}(y),\deg_{Z_i^3\cup Z_i^5}(y),c(y)\right)
& \text{if } 
y \in Z_i^4, 
\end{cases}
$$
for some $i\in [r]$.

Recall that we defined $t:=(\Delta+1)^3(\Delta^3+1)$, so let us
identify the codomain of $f$ with the set $[t]$. Now for $i\in [r]$ and $j\in
[t]$ we set $$
\ti{B}_{i,j}:= Z_i^2\cap f^{-1}(j), \qquad \ti{B}_{i,t+j}:= Z_i^5\cap f^{-1}(j),
$$
$$
\ti{B}_{i,j}':= Z_i^3\cap f^{-1}(j), \qquad \ti{B}_{i,t+j}':= Z_i^4\cap f^{-1}(j),
$$
$$
\ti{C}_{i,j}:= Q_i^2\cap f^{-1}(j), \qquad \ti{C}_{i,t+j}:= Q_i^4\cap f^{-1}(j),
$$
$$
\ti{C}_{i,j}':= Q_i^3\cap f^{-1}(j), \qquad \ti{C}_{i,t+j}':= Q_i^5\cap f^{-1}(j).
$$
Observe, for example, that for $y\in \ti{B}_{i,j}$
the third component of $f(y)$ is exactly equal to $\deg_{L(i,j)}(y)$.
Now, for any $$
yy'\in\binom{\ti{C}_{i,j}}{2}\cup\binom{\ti{B}_{i,j}}{2}
\cup \binom{\ti{C}_{i,j}'}{2}\cup\binom{\ti{B}_{i,j}'}{2},
$$ 
we have
$f(y)=j=f(y')$ and hence any of the parameters required in \ref{lem:H:Cindep} and \ref{lem:H:deg}
have the same value for $y$ and $y'$.  

The only thing missing before the proof of Lemma~\ref{lem:H} is complete is that we need 
to guarantee that every $y\in Z_i^2\cup Z_i^5\cup Q_i^2\cup Q_i^4$ has at most $\Delta-1$
neighbours in $\ti{V}_i$, 
as required in the first line of \ref{lem:H:deg}. 
If a vertex $y$ does not satisfy this, it must have 
\emph{all} its $\Delta$ neighbours in $\ti{V}_i$. 
Since by definition of $\ti{V}_i$ these neighbours have been mapped to $z_i^1$, we can map $y$ to $z_i^0$  
(instead of mapping it to $z_i^2$, $z_i^5$, $q_i^2$ or $q_i^4$).

Even if, in this way, all of the vertices in $Z_i^2\cup Z_i^5\cup Q_i^2\cup Q_i^4$
would have to be mapped to $z_i^0$, 
(\ref{eq:H:X}) assures us that these are at most 
$4\frac{\eta_\isubsc{H}}{10}\frac{m}{2r}$ vertices.
Since
by (\ref{eq:preim}) at most $(1+\frac{\eta_\isubsc{H}}{10})\frac{m}{2r}$ have already been mapped to $z_i^0$
in the first round and by (\ref{eq:H:X}) at most $\frac{\eta_\isubsc{H}}{10}\frac{m}{2r}$ in the second round,
this does not violate the upper bound in \ref{lem:H:Vi}.
\end{proof}


\section{The constrained blow-up lemma}
\label{sec:blowup}

As explained earlier, the proof of the constrained blow-up lemma uses 
techniques developed in~\cite{millennium,RoRuTa} adapted to our setting.
In fact, the proof we present here follows the embedding strategy used
in the proof of~\cite[Theorem~1.5]{millennium}.
This strategy is roughly as follows. Assume we want to embed
the bipartite graph $H$ on vertex set $\ti U\dcup\ti V$ into the host graph $G$ on vertex
set $U\dcup V$. Then we consider injective mappings $f\colon\ti V\to V$, and
try to find one
that can be extended to $\ti U$ such that the resulting mapping is an
embedding of $H$ into $G$. For determining whether a particular mapping $f$
can be extended in this way we shall construct an auxiliary bipartite graph
$B_f$, the so-called candidate graph (see
Definition~\ref{def:cand}), which contains a
matching covering one of its partition classes if and only if $f$ can be
extended. Accordingly, our goal will be to check 
whether $B_f$ contains such a matching $M$ which we will do
by appealing to Hall's condition.
On page~\pageref{sec:bl:match} we will explain the details of this part
of the proof, determine necessary
conditions for the application of Hall's theorem, and collect them in form of a
matching lemma (Lemma~\ref{lem:matching}). 
It will then remain to show that
there is a mapping $f$ such that $B_f$ satisfies the conditions of this
matching lemma. This will require most of the work. The idea here is as follows.

We will show that mappings $f$ usually have the necessary
properties as long as they do not map neighbourhoods $N_H(\ti u)\subset\ti V$ of
vertices in $\ti u\in\ti U$ to certain ``bad'' spots in $V$. The existence of
(many) mappings that avoid these ``bad'' spots is verified with the help of a
hypergraph packing lemma (Lemma~\ref{lem:pack}). This lemma states
that half of all possible mappings $f$ avoid \emph{almost all} ``bad'' spots and
can easily be turned into mappings $f'$ avoiding \emph{all} ``bad'' spots with
the help of so-called switchings.

\subsection{Candidate graphs}
\label{sec:bl:cand}

If we have injective mappings $f\colon\ti V\to V$ as described in the previous
paragraph we would like to decide whether $f$ can be extended to an
embedding of $H$ into $G$. Observe that in such an embedding each vertex
$\ti u\in \ti U$ has to be embedded to a vertex $u\in U$ such that the
following holds. The neighbourhood $N_{H}(\ti u)$ has its image
$f(N_{H}(\ti u))$ in the set $N_G(u)$. 
Determining which vertices~$u$ are ``candidates'' for the embedding of~$\ti
u$ in this sense gives rise to the following bipartite graph.

\begin{definition}[candidate graph]
\label{def:cand}
  Let $H$ and $G$ be bipartite graphs on vertex sets $\ti{U}\dcup\ti{V}$ and
  $U\dcup V$, respectively. For an injective function $f\colon\ti{V}\to V$ we
  say that a vertex $u\in U$ is an
  \emph{$f$-candidate} for $\ti{u}\in\ti U$ if and
  only if $f(N_H(\ti{u}))\subset N_G(u)$. 
  
  The \emph{candidate graph}
  $B_f(H,G):=(\ti{U}\dcup U,E_f)$ for~$f$ is the bipartite graph with edge set 
  \begin{equation*}
    E_f:=\left\{
      \ti{u}u\in\ti{U}\times U \colon\, u\text{ is an $f$-candidate for }\ti{u}
    \right\}\,.
  \end{equation*}
\end{definition}

Now it is easy to see that the mapping $f$ described above can be extended
to an embedding of $H$ into $G$ if and only if the corresponding candidate
graph has a matching covering $\ti{U}$.
Clearly, if the candidate graph $B_f(H,G)$ of $f$ has vertices $\ti u\in\ti U$ of
degree $0$, then $B_f(H,G)$ has no such matching and hence $f$ cannot be
extended. More generally we would like to avoid that $\deg_{B_f(H,G)}(\ti u)$
is too small. Notice that this means precisely that $f$ should not map $N_H(\ti u)$ 
to a set $B\subset V$ that has a small common
neighbourhood in $G$. These sets $B$ are the ``bad'' spots (see
the beginning of this section) that should be avoided by~$f$.

We explained above that, in order to avoid ``bad'' spots,
we will have to change certain mappings~$f$ slightly. The exact definition of this
operation is as follows.

\begin{definition}[switching]
  Let $f,f'\colon X\to Y$ be injective functions. We say that $f'$ is obtained
  from $f$ by a \emph{switching} if there are
  $u,v\in X$ with $f'(u)=f(v)$ and $f'(v)=f(u)$ and $f(w)=f'(w)$ for all
  $w\not\in\{u,v\}$. The \emph{switching distance} \emph{$\swdist(f,f')$}
  of $f$ and $f'$ is at most $s$ if the mapping $f'$ can be obtained from $f$
  by a sequence of at most $s$ switchings.
\end{definition}

These switchings will alter the candidate graph corresponding to the
injective function slightly (but not much, see Lemma~\ref{lem:switch}). In order
to quantify this, we further define the neighbourhood distance between two
bipartite graphs $B$ and $B'$ which determines the number of vertices (in one
partition class) whose neighbourhoods differ in $B$ and $B'$.

\begin{definition}[neighbourhood distance]
Let $B=(U\dcup\ti{U},E)$, $B'=(U\dcup\ti{U},E')$ be bipartite graphs. We define
the \emph{neighbourhood distance} of $B$ and $B'$ with respect to $\ti{U}$
as
\begin{equation*}
  \ndist(B,B'):=\big|\{ \ti{u}\in\ti{U}\colon\, N_B(\ti{u})\neq
  N_{B'}(\ti{u})\}\big|.
\end{equation*}
\end{definition}

The next simple lemma now examines the effect of switchings on the neighbourhood
distance of candidate graphs and shows that functions with small switching
distance correspond to candidate graphs with small neighbourhood distance.

\begin{lemma}[switching lemma]\label{lem:switch}
  Let $H$ and $G$ be bipartite graphs on vertex sets $\ti{U}\dcup\ti{V}$ and
  $U\dcup V$, respectively, such that $\deg_H(\ti{v})\le\Delta$ for all
  $\ti{v}\in\ti{V}$ and let $f,f'\colon\ti{V}\to V$ be injective functions with
  switching distance $\swdist(f,f')\le s$. 
  Then the neighbourhood distance of the candidate graphs $B_f(H,G)$ and
  $B_{f'}(H,G)$ satisfies
  \begin{equation*}
    \ndist\Big(B_f(H,G),B_{f'}(H,G)\Big)\le 2s\Delta\,.
  \end{equation*}
\end{lemma}
\begin{proof}
  We proceed by induction on $s$. For $s=0$ the lemma is trivially true. Thus,
  consider $s>0$ and let $g$ be a function with $\swdist(f,g)\le s-1$ and
  $\swdist(g,f')=1$. Define 
  \begin{equation*}
    N(f,f'):=\left\{
      \ti{u}\in\ti{U} \colon N_{B_f(H,G)}(\ti{u})\neq N_{B_{f'}(H,G)}(\ti{u})
    \right\}.
  \end{equation*}  
  Clearly, $|N(f,f')|=\ndist(B_f(H,G),B_{f'}(H,G))$ and $N(f,f')\subset
  N(f,g)\cup N(g,f')$. By induction hypothesis we have $|N(f,g)|\le2(s-1)\Delta$.
  The remaining switching from $g$ to $f'$ interchanges only the images of two
  vertices from $\ti{V}$, say $\ti{v}_1$ and $\ti{v}_2$. It follows that 
  \begin{equation*}
 	  N(g,f')=\left\{
	    \ti{u}\in N_H(\ti{v}_1)\cup N_H(\ti{v}_2) \colon
	    N_{B_g(H,G)}(\ti{u})\neq N_{B_{f'}(H,G)}(\ti{u}) 
	  \right\},
  \end{equation*}
  which implies $|N(g,f')|\le2\Delta$ and therefore we get
  $|N(f,f')|\le2s\Delta$.
\end{proof}

\subsection{A hypergraph packing lemma}
\label{sec:bl:pack}

The main ingredient to the proof of the constrained blow-up lemma is the
following hypergraph packing result (Lemma~\ref{lem:pack}). To understand what
this lemma says and how we will apply it, recall that we would like to embed the
vertex set $\ti U$ of $H$ into the vertex set $U$ of $G$ such that subsets of
$\ti U$ that form neighbourhoods in the graph $H$ avoiding certain ``bad'' spots
in $U$. If $H$ is a $\Delta$-regular graph, then these neighbourhoods form
$\Delta$-sets. In this case, as we will see, also the ``bad'' spots
form~$\Delta$-sets. Accordingly, we have to solve the problem of packing the
neighbourhood $\Delta$-sets $\hyper{N}$ and the ``bad'' $\Delta$-sets $\hyper{B}$, 
which is a hypergraph packing 
problem. Lemma~\ref{lem:pack} below
states that this is possible under certain conditions. One of these conditions is that the
``bad'' sets should not ``cluster'' too much (although there might be many of
them). The following definition makes this precise.

\begin{definition}[corrupted sets]\label{def:corrupt}
  For $\Delta\in\NATS$ and a set $V$ let
  $\hyper{B}\subset\binom{V}{\Delta}$ be a collection of $\Delta$-sets in $V$
  and let $x$ be a positive real. We say that all $B\in\hyper{B}$ are
  \emph{$x$-corrupted} by $\hyper{B}$.
  Recursively, for $i\in[\Delta-1]$ an $i$-set
  $B\in\binom{V}{i}$ in $V$ is called $x$-corrupted by $\hyper{B}$ if it
  is contained in more than $x$ of the $(i+1)$-sets that are
  $x$-corrupted by $\hyper{B}$.

  Observe that, if a vertex $v\in V$ is not $x$-corrupted by $\hyper{B}$, then
  it is also not $x'$-corrupted by $\hyper{B}$ for any $x'>x$.
\end{definition}

The hypergraph packing lemma now implies that $\hyper{N}$ and $\hyper{B}$
can be packed if $\hyper{B}$ contains no corrupted sets. In fact this
lemma states that \emph{half of all} possible ways to map the vertices of
$\hyper{N}$ to $\hyper{B}$ can be turned into such a packing by performing
a sequence of few switchings.

\begin{lemma}[hypergraph packing lemma~\cite{RoRuTa}]
\label{lem:pack}
  For all integers $\Delta\ge 2$ and $\ell\ge 1$ 
  and every positive~$\sigma$
  there are positive constants 
  $\eta_{\subref{lem:pack}}$, and $n_{\subref{lem:pack}}$ such that the following holds. 
  Let $\hyper{B}$ be a $\Delta$-uniform
  hypergraph on $n'\ge n_{\subref{lem:pack}}$ vertices such that
  no vertex of $\hyper{B}$ is $\eta_{\subref{lem:pack}}n'$-corrupted by
  $\hyper{B}$. 
  Let $\hyper{N}$ be a $\Delta$-uniform hypergraph on 
  $n\le n'$ vertices such that no vertex of $\hyper{N}$ is
  contained in more than $\ell$ edges of $\hyper{N}$.
  
  Then for at least half of all injective functions $f\colon V(\hyper{N})\to
  V(\hyper{B})$ there are packings $f'$ of $\hyper{N}$ and $\hyper{B}$
  with switching distance $\swdist(f,f')\le\sigma n$.
\qed
\end{lemma}

When applying this lemma we further make use of the following lemma which 
helps us to bound corruption.

\begin{lemma}[corruption lemma]
\label{lem:corrupt}
  Let $n\text{,}\Delta>0$ be integers and $\mu$ and $\eta$ be positive reals. Let
  $V$ be a set of size $n$ and $\hyper{B}\subset\binom{V}{\Delta}$ be a family
  of $\Delta$-sets of size at most $\mu n^\Delta$. Then at most
  $(\Delta!/\eta^{\Delta-1})\mu n$ vertices are $\eta n$-corrupted by $\hyper{B}$.
\end{lemma}
\begin{proof}
  For $i\in[\Delta]$ let $\hyper{B}_i$ be the
  family of all those $i$-sets $B'\in\binom{V}{i}$ that are $\eta n$-corrupted
  by $\hyper{B}$. We will prove by induction on $i$ (starting at $i=\Delta$) that
  \begin{equation}
  \label{eq:corrupt}
    |\hyper{B}_i|\le\frac{\Delta!/i!}{\eta^{\Delta-i}}\mu n^i.
  \end{equation} 
  For $i=1$ this establishes the lemma. For $i=\Delta$ the assertion is true by
  assumption. Now assume that \eqref{eq:corrupt} is true for $i>1$. By
  definition every $B'\in\hyper{B}_{i-1}$ is contained in more than $\eta n$
  sets $B\in\hyper{B}_{i}$. On the other hand, clearly every
  $B\in\hyper{B}_{i}$ contains at most $i$ sets from $\hyper{B}_{i-1}$.
  Double counting thus gives
  \begin{equation*}
    \eta n\left|\hyper{B}_{i-1}\right| 
    \le \left|\big\{(B',B):B'\in\hyper{B}_{i-1},B\in\hyper{B}_{i},B'\subset
    B\big\}\right| 
    \le i \left|\hyper{B}_{i}\right|
    \leByRef{eq:corrupt} i\frac{\Delta!/i!}{\eta^{\Delta-i}} \mu n^{i},
  \end{equation*}
  which implies \eqref{eq:corrupt} for $i$ replaced by $i-1$.
\end{proof}

\subsection{A matching lemma}
\label{sec:bl:match}

We indicated earlier that we are interested in determining whether
a candidate graph has a matching covering one of its partition classes.
In order to do so we will make use of the following matching lemma which is an
easy consequence of Hall's theorem. This lemma takes two graphs $B$ and $B'$ as
input that have small neighbourhood distance.
In our application these two graphs will be candidate graphs
that correspond to two injective
mappings $f$ and $f'$ with small switching distance (such as promised by the
hypergraph packing lemma, Lemma~\ref{lem:pack}). Recall that
Lemma~\ref{lem:switch} guarantees that mappings with small switching
distance correspond to candidate graphs with small neighbourhood distance.

The matching lemma asserts that $B'$ has the desired matching if certain
vertex degree and neighbourhood conditions are satisfied. These conditions are
somewhat technical. They are tailored exactly to match the conditions that we
establish for candidate graphs in the proof of the constrained blow-up lemma
(see Claims~\ref{cl:bl:small}--\ref{cl:bl:big}).

\begin{lemma}[matching lemma]
\label{lem:matching}
  Let $B=(\ti{U}\dcup U,E)$ and $B'=(\ti{U}\dcup U,E')$ be bipartite graphs with
  $|U|\ge|\ti{U}|$ and $\ndist(B,B')\le s$. If there are
  positive integers $x$ and $n_1,n_2,n_3$ such that
  \begin{enumerate}[label={\rm (\roman{*})}]
	\item \label{lem:matching:i} $\deg_{B'}(\ti{u})\ge n_1$ for all $\ti{u}\in\ti{U}$,
    \item \label{lem:matching:ii} $|N_{B'}(\ti{S})|\ge x|\ti{S}|$ for all
      $\ti{S}\subset\ti{U}$ with $|\ti{S}|\le n_2$
	\item \label{lem:matching:iii} $e_{B'}(\ti{S},S)\le\frac{n_1}{n_3}|\ti{S}||S|$
	  for all $\ti{S}\subset\ti{U}$, $S\subset U$ with $xn_2\le|S|<|\ti{S}|<n_3$,
	\item \label{lem:matching:iv} $|N_{B}(S)\cap\ti{S}|>s$ for all
	  $\ti{S}\subset\ti{U}$, $S\subset U$ with $|\ti{S}|\ge n_3$ and
	  $|S|>|U|-|\ti{S}|$,
  \end{enumerate}
  then $B'$ has a matching covering $\ti{U}$.
\end{lemma}
\begin{proof}
  We will check Hall's condition in $B'$ for all sets $\ti{S}\subset\ti{U}$.
  We clearly have $|N_{B'}(\ti{S})|\ge|\ti{S}|$ for $|\ti{S}|\le xn_2$
  by~\ref{lem:matching:ii} (if $|\ti{S}|>n_2$, then consider a subset
  of~$\ti{S}$ of size $n_2$). 
  
  Next, consider the case $xn_2<|\ti{S}|<n_3$. Set
  $S:=N_{B'}(\ti{S})$ and assume, for a contradiction, that $|S|<|\ti{S}|$.
  Since $|S|<|\ti{S}|<n_3$ we have $|S|/n_3<1$. Therefore,
  applying~\ref{lem:matching:i}, we can conclude that
  \begin{equation*}
     e_{B'}(\ti{S},S)=\sum_{\ti{u}\in\ti{S}}\left|N_{B'}(\ti{u})\right|
       \ge n_1|\ti{S}|>\frac{n_1}{n_3}|\ti{S}||S|,
  \end{equation*}
  which is a contradiction to~\ref{lem:matching:ii}. Thus
  $|N_{B'}(\ti{S})|\ge|\ti{S}|$.
  
  Finally, for sets $\ti{S}$ of size at least $n_3$ set $S:=U\setminus
  N_{B'}(\ti{S})$ and assume, again for a contradiction, that
  $|N_{B'}(\ti{S})|<|\ti{S}|$. This implies $|S|>|U|-|\ti{S}|$. Accordingly we can
  apply~\ref{lem:matching:iv} to $\ti{S}$ and $S$ and infer that
  $|N_{B}(S)\cap\ti{S}|>s$. Since $\ndist(B,B')\le s$, at most $s$ vertices from
  $\ti{U}$ have different neighbourhoods in $B$ and $B'$ and so
  \begin{equation*}\begin{split}
    \left|N_{B'}(S)\cap\ti{S}\right| 
    &= \bigg|\left\{\ti{u}\in\ti{S}:N_{B'}(\ti{u})\cap S\neq\emptyset\right\}\bigg|
    \\ 
    &\ge \bigg|\left\{\ti{u}\in\ti{S}:N_{B}(\ti{u})\cap
         S\neq\emptyset\right\}\bigg|-s
    = \left|N_{B}(S)\cap\ti{S}\right|-s
    > 0,
  \end{split}\end{equation*}
  which is a contradiction as $S=U\setminus N_{B'}(\ti{S})$.
\end{proof}

\subsection{Proof of Lemma~\ref{lem:blowup}}

Now we are almost ready to present the proof of the constrained blow-up
lemma (Lemma~\ref{lem:blowup}). We just need one further 
technical lemma as preparation. This lemma considers a family of
pairwise disjoint $\Delta$-sets $\hyper{S}$ in a set $S$ and
states that a random injective function from $S$ to a set $T$
usually has the following property. The images $f(\hyper{S})$ 
of sets in $\hyper{S}$ ``almost'' avoid a small family of ``bad'' sets
$\hyper{T}$ in $T$.

\begin{lemma}\label{lem:bl:sets}
  For all positive integers $\Delta$ and positive reals $\beta$ and $\mu_\isubsc{S}$ there is
  $\mu_\isubsc{T}>0$ such that the following holds. Let $S$ and $T$ be disjoint sets,
  $\hyper{S}\subset\binom{S}{\Delta}$ be a family of pairwise disjoint
  $\Delta$-sets in $S$ with $|\hyper{S}|\le\frac1\Delta(1-\mu_\isubsc{S})|T|$, and
  $\hyper{T}\subset\binom{T}{\Delta}$ be a family of $\Delta$-sets in $T$ with
  $|\hyper{T}|\le\mu_\isubsc{T}|T|^\Delta$.
  
  Then a random injective function $f\colon S\to T$ satisfies
  $|f(\hyper{S})\setminus\hyper{T}|>(1-\beta)|\hyper{S}|$ with probability at
  least $1-\beta^{|\hyper{S}|}$.
\end{lemma}
\begin{proof}
  Given $\Delta$, $\beta$, and $\mu_\isubsc{S}$ choose 
  \begin{equation}\label{eq:bl:sets:muT}
    \mu_\isubsc{T}:=\sqrt[\beta]{\beta}
      \left(\frac{e}{\beta}\left(\frac{\Delta}{\mu_\isubsc{S}}\right)^\Delta\right)^{-1}.
  \end{equation} 
  Let $S$, $T$, $\hyper{S}$, and $\hyper{T}$ be as required and let $f$ be a
  random injective function from $S$ to $T$. We consider $f$ as a consecutive
  random selection (without replacement) of images for the elements of $S$ where
  the images of the elements of the (disjoint) sets in $\hyper{S}$ are chosen
  first. Let $S_i$ be the $i$-th such set in $\hyper{S}$. Then the probability
  that $f$ maps $S_i$ to a set in $\hyper{T}$, which we denote by $p_i$, is
  at most
  \begin{equation*}
    p_i \le \frac{|\hyper{T}|}{\binom{|T|-(i-1)\Delta}{\Delta}}
    \le \frac{\mu_\isubsc{T}|T|^\Delta}{\binom{\mu_\isubsc{S}|T|}{\Delta}}
    \le \mu_\isubsc{T} \frac{ |T|^{\Delta} }
                   { \left(\frac{\mu_\isubsc{S}|T|}{\Delta}\right)^{\Delta} }
    = \mu_\isubsc{T}\left(\frac{\Delta}{\mu_\isubsc{S}}\right)^\Delta =: p\,,
  \end{equation*}
  where the second inequality follows from
  $(i-1)\Delta\le|\bigcup\hyper{S}|\le(1-\mu_\isubsc{S})|T|$. 
  Let $Z$ be a random variable with distribution $\Bin(|\hyper{S}|,p)$.  
  It follows that $\Prob[|f(\hyper{S})\cap\hyper{T}|\ge z]\le\Prob[Z\ge z]$.
  Since
  \begin{equation*}
    \Prob[Z\ge z]\le\binom{|\hyper{S}|}{z}p^z
      < \left(\frac{e|\hyper{S}|p}{z}
    \right)^z\,,
  \end{equation*}
  we infer that
  \begin{equation*}
    \Prob\Big[|f(\hyper{S})\cap\hyper{T}|\ge\beta|\hyper{S}|\Big]
    < \left(\frac{e p}{\beta}\right)^{\beta|\hyper{S}|}
    =\left( \frac{e\mu_\isubsc{T}}{\beta}
       \left(\frac{\Delta}{\mu_\isubsc{S}}\right)^\Delta
     \right)^{\beta|\hyper{S}|}
    \eqByRef{eq:bl:sets:muT}\beta^{|\hyper{S}|},
  \end{equation*}
  which proves the lemma since $|f(\hyper{S})\cap\hyper{T}|\ge\beta|\hyper{S}|$
  holds iff $|f(\hyper{S})\setminus\hyper{T}|\le(1-\beta)|\hyper{S}|$.
\end{proof}

Now we can finally give the proof of Lemma~\ref{lem:blowup}.

\begin{proof}[Proof of Lemma~\ref{lem:blowup}]
  We first define a sequence of constants.
  Given $\Delta$, $d$, 
  and $\eta$ 
  fix $\Delta':=\Delta^2+1$. Choose $\beta$ and $\sigma$ such that
  \begin{equation}\label{eq:bl:betasigma}
    \beta^{\frac17(\frac d2)^\Delta}\le\frac15  \qquad\text{and}\qquad
    \frac{(1-\beta)d^\Delta}{100^\Delta}\ge 2\sigma
  \end{equation} 
  Apply the hypergraph packing lemma, Lemma~\ref{lem:pack}, with input $\Delta$,
  $\ell=2\Delta+1$,
  and $\sigma$ to obtain constants
  $\eta_{\subref{lem:pack}}$, and
  $n_{\subref{lem:pack}}$. Next, choose $\eta'_{\subref{lem:pack}}$,
  $\mu_\subsc{BL}$, and $\mu_\isubsc{S}$ such that
  \begin{equation}\label{eq:bl:alphaetamu}
    \frac{\eta'_{\subref{lem:pack}}}{1-\eta}\le
    \eta_{\subref{lem:pack}}\,,
    \qquad
      \frac{\Delta!\cdot 2\mu_{\subsc{BL}}}
        {(\eta'_{\subref{lem:pack}})^{\Delta-1}} \le \eta\,,
    \qquad
      \frac1{\Delta'}\le\frac1\Delta(1-\mu_\isubsc{S})\,.
  \end{equation}
  Lemma~\ref{lem:bl:sets} with input $\Delta$, $\beta$, $\mu_\isubsc{S}$
  provides us with a constant $\mu_\isubsc{T}$. We apply
  Lemma~\ref{lem:joint} two times, once with input $\Delta=\ell$, $d$,
  $\eps':=\frac12 d$, and
  $\mu=\mu_{\subsc{BL}}/\Delta'$ and once with input $\Delta=\ell$, $d$,
  $\eps':=\frac12 d$, and $\mu=\mu_\isubsc{T}$ and get constants 
  $\eps_{\subref{lem:joint}}$ and $\ti\eps_{\subref{lem:joint}}$, respectively.
  Now we can fix the promised constant $\eps$ such that
  \begin{equation}\label{eq:bl:eps}
    \eps\le\min\left\{ \frac{\eps_{\subref{lem:joint}}}{\Delta'}\,,\ 
      \frac{d}{2\Delta}\right\},
    \qquad\text{and}\qquad
    \frac{\eps\Delta'}{\eta(1-\eta)}
     < \min\{d,\ti\eps_{\subref{lem:joint}}\}.
  \end{equation}
  As last input let $r_1$ be given and set 
  \begin{equation}\label{eq:bl:xi}
    \xi_{\subref{lem:joint}}:=\eta(1-\eta)/(r_1\Delta').
  \end{equation}
  Next let $c_{\subref{lem:joint}}$ be the maximum of the two constants obtained
  from the two applications of Lemma~\ref{lem:joint}, that we started above,
  with the additional parameter $\xi_{\subref{lem:joint}}$.
  Further, let $\nu$ and $c_{\subref{lem:exp}}$ be the constants 
  from Lemma~\ref{lem:exp} for input $\Delta$, $d$, and $\eps$,
  and let $c_{\subref{lem:stars-big}}$ be the constant from
  Lemma~\ref{lem:stars-big} for input $\Delta$ and $\nu$. Finally, we
  choose $c=\max\{c_{\subref{lem:joint}},c_{\subref{lem:exp}},c_{\subref{lem:stars-big}}\}$.
  With this we defined all necessary constants.

  Now assume we are given any $1\le r\le r_1$,
  and a random graph $\Gamma=\Gnp$ with $p\ge c(\log n/n)^{1/\Delta}$, 
  where, without loss of generality, $n$ is such that
  \begin{equation}\label{eq:bl:n}
    (1-\eta')\tfrac{n}{r} \ge n_{\subref{lem:pack}}.
  \end{equation}
  Then, with high probability, the graph $\Gamma$ satisfies the assertion
  of the different lemmas concerning random graphs, that we started to apply
  in the definition of the constants. More precisely, by the choice of the
  constants above,
  \begin{enumerate}[label={\rm (P\arabic{*})},leftmargin=*]
    \item\label{bl:Pstars} 
      $\Gamma$ satisfies the assertion of Lemma~\ref{lem:stars-big}
      for parameters $\Delta$ and $\nu$, i.e., for any set $X$ and any
      family~$\cF$ with the conditions required in this lemma, the
      conclusion of the lemma holds.
    \item\label{bl:Pjoint}
      Similarly $\Gamma$ satisfies the assertion of Lemma~\ref{lem:joint}
      for parameters $\Delta=\ell$, $d$, $\eps'=\frac12 d$, $\mu=\mu_{\rm
        BL}/\Delta'$, $\eps_{\subref{lem:joint}}$, and
      $\xi_{\subref{lem:joint}}$. The same holds for parameters
      $\Delta=\ell$, $d$, $\eps'=\frac12 d$, $\mu=\mu_\isubsc{T}$, 
      $\ti\eps_{\subref{lem:joint}}$, and $\xi_{\subref{lem:joint}}$.
    \item\label{bl:Pexp}
      $\Gamma$ satisfies the assertion of Lemma~\ref{lem:exp}
      for parameters $\Delta$, $d$, $\eps$, and $\nu$.
  \end{enumerate}
  In the following we will assume that $\Gamma$ has these properties and
  show that it then also satisfies the conclusion of the
  constrained blow-up lemma, Lemma~\ref{lem:blowup}.

  Let $G\subset\Gamma$ and $H$ be two bipartite graphs on vertex sets
  $U\dcup V$ and $\ti{U}\dcup\ti{V}$, respectively, that fulfil the requirements of
  Lemma~\ref{lem:blowup}. Moreover, let 
  $\hyper{H}\subset\binom{\ti{V}}{\Delta}$ be the family of special
  $\Delta$-sets, and $\hyper{B}\subset\binom{V}{\Delta}$ be the family of
  forbidden $\Delta$-sets. It is not difficult to see that, by possibly
  adding some edges to $H$, we can assume that the following holds.
  \begin{enumerate}[label={\rm ($\tilde{\text{\Alph{*}}}$)},start=21,leftmargin=*]
    \item\label{bl:U} All vertices in $\ti U$ have degree \emph{exactly} $\Delta$.
    \item\label{bl:V} All vertices in $\ti V$ have degree maximal $\Delta+1$.
  \end{enumerate}  

  Our next step will be to split the partition class $U$ of
  $G$ and the corresponding partition class $\ti U$ of $H$ into $\Delta'$ 
  parts of equal size.
  From the partition of~$H$ we require that no two vertices in
  one part have a common neighbour. This will guarantee that the
  neighbourhoods of two different vertices from one part form disjoint vertex
  sets (which we need because we would like to apply Lemma~\ref{lem:exp}
  later, in the proof of Claim~\ref{cl:bl:small}, 
  and Lemma~\ref{lem:exp} asserts certain properties for
  families of disjoint vertex sets).
   
  Let us now explain precisely how we split $U$ and $\ti U$.
  We assume for simplicity that $|\ti{U}|$ and
  $|U|$ are divisible by $\Delta'$ and
  partition the sets $U$ arbitrarily into $\Delta'$ parts
  $U=U_1\dcup\dots U_{\Delta'}$ of equal size, i.e., sets of size
  at least $n/(r\Delta')$. Similarly let $\ti{U}=\ti{U}_1\dcup\dots\dcup
  \ti{U}_{\Delta'}$ be a partition of $\ti{U}$ into sets of equal size such
  that each $\ti{U}_j$ is $2$-independent in $H$. Such a partition exists by
  the Theorem of Hajnal and Szemer\'edi (Theorem~\ref{thm:HajSze})
  applied to $H^2[\ti{U}]$ because  the maximum degree of $H^2$ is less than
  $\Delta'=\Delta^2+1$.

  In Claim~\ref{cl:bl:matching} below we will assert
  that there is an embedding $f'$
  of $\ti V$ into $V$ that can be extended to \emph{each} of the $\ti U_j$
  separately such that we obtain an embedding of $H$ into $G$. To this end we
  will consider the candidate graphs $B_{f'}(H_j,G_j)$ defined by $f'$ (see
  Definition~\ref{def:cand}) and show, that there is an~$f'$ such that each
  $B_{f'}(H_j,G_j)$ has a matching covering $\ti U_j$.
  This, as discussed earlier, will ensure the existence of the desired
  embedding. For preparing this argument, we first need to exclude some vertices of $V$
  which are not suitable for such an embedding. For identifying these vertices, we
  define the following family of $\Delta$-sets which contains
  $\hyper{B}$ and all sets in $V$ that have a small common neighbourhood
  in some $\ti U_j$.

  Define $\hyper{B}':=\hyper{B}\cup\bigcup_{j\in\Delta'}\hyper{B}_j$ where
  \begin{equation}
  \label{eq:bl:B}
    \hyper{B}_j:=
    \left\{B\in\binom{V}{\Delta}\colon\,
    \big|\coN_G(B)\cap U_j\big|<(\tfrac12 d)^{\Delta}p^{\Delta}|U_j| \right\}
    \eqByRef{eq:BAD} \BAD{\Delta}{d/2}{d}(V,U_j).
  \end{equation}
  We claim that we obtain a set $\hyper{B}'$ that is not much larger than
  $\hyper{B}$. Indeed, by Proposition~\ref{prop:subpairs} the pair 
  \begin{equation}\label{eq:bl:dense}
    \text{$(V,U_j)$ is
      $(\eps\Delta',d,p)$-dense for all $j\in[\Delta']$,}
  \end{equation}
  and $\eps\Delta'\le\eps_{\subref{lem:joint}}$ by~\eqref{eq:bl:eps}.
  Moreover we have
  $|U_j|\ge n/(r\Delta')\ge n/(r_1\Delta')\ge\xi_{\subref{lem:joint}} n$
  by~\eqref{eq:bl:xi}.
  We can thus use the fact that our random graph $\Gamma$ satisfies
  property~\ref{bl:Pjoint} (with $\mu=\mu_{\subsc{BL}}/\Delta'$) on the
  bipartite subgraph $G[V\dcup U_j]$ and
  conclude that $|\hyper{B}_j|\le\mu_{\subsc{BL}}|V|^\Delta/\Delta'$. 
  Since $|\hyper{B}|\le\mu_{\subsc{BL}} |V|^\Delta$ by assumption we infer
  \begin{equation*}
    |\hyper{B}'|\le 
      \mu_{\subsc{BL}} |V|^\Delta + \Delta'\cdot\mu_{\subsc{BL}}|V|^\Delta/\Delta'
      = 2\mu_{\subsc{BL}} |V|^\Delta.
  \end{equation*}
  Set
  \begin{equation}\label{eq:bl:corrupt}
    V':=V\setminus V'' \quad\text{with}\quad
    V'':=\Big\{v\in V\colon\, v 
      \text{ is $\eta'_{\subref{lem:pack}}|V|$-corrupted by }
      \hyper{B}'\Big\}
  \end{equation}
  and delete all sets from $\hyper{B}'$ that contain vertices from $V''$.
  This determines the set $V''$ of vertices that we exclude from $V$ for
  the embedding. We will next show that we did not exclude too many
  vertices in this process. For this we use the corruption lemma,
  Lemma~\ref{lem:corrupt}.
  Indeed, Lemma~\ref{lem:corrupt} applied with $n$ replaced by $|V|$,
  with $\Delta$,  $\mu=2\mu_{\subsc{BL}}$, and
  $\eta'_{\subref{lem:pack}}$ to $V$ and $\hyper{B}'$ implies that
  \begin{equation}\label{eq:bl:V'}
    |V''| 
    \le\frac{\Delta!}{(\eta'_{\subref{lem:pack}})^{\Delta-1}}2\mu_{\subsc{BL}}|V|
    \leByRef{eq:bl:alphaetamu} \eta|V|
    \quad\text{and thus}\quad
    n':=|V'|\ge(1-\eta)|V|.
  \end{equation}
  Let 
  \begin{equation*}
    H_j:=H\big[\ti{U}_j\dcup\ti{V}\big] \qquad\text{and}\qquad
    G_j:=G\big[U_j\dcup V'\big].
  \end{equation*}
  Now we are ready to state the claim announced above, which asserts that
  there is an embedding $f'$ of the vertices in $\ti V$ to the vertices in
  $V'$ such that the corresponding candidate graphs 
  $B_{f'}(H_j,G_j)$ have matchings covering $\ti U_j$. 
  As we will shall show, this claim implies the assertion of the constrained
  blow-up lemma. Its proof, which we will provide thereafter, 
  requires the matching lemma
  (Lemma~\ref{lem:matching}), and the hypergraph packing lemma
  (Lemma~\ref{lem:pack}).
  \begin{claim}\label{cl:bl:matching}
    There is an injection $f':\ti{V}\to V'$
    with 
    $f'(T)\not\in\hyper{B}$ for all $T\in\hyper{H}$
    such that for \emph{all} $j\in[\Delta']$ the candidate graph
    $B_{f'}(H_j,G_j)$ has a matching covering $\ti{U}_j$.  
  \end{claim}
  Let us show that proving this claim suffices to establish the constrained blow-up
  lemma. Indeed, let $f':\ti{V}\to V'$ be such an injection and denote by
  $M_j:\ti{U}_j\to U_j$  the corresponding matching in
  $B_{f'}(H_j,G_j)$ for $j\in[\Delta]$. We claim that the function
  $g\colon\ti{U}\dcup\ti{V}\to U\dcup V$, defined by
  \begin{equation*}
    g(\ti{w}) = \begin{cases}
       M_j(\ti{w}) & \ti{w}\in\ti{U}_j, \\
       f'(\ti{w}) & \ti{w}\in\ti{V},
    \end{cases}
  \end{equation*}
  is an embedding of $H$ into $G$. To see this, notice first that $g$ is
  injective since $f'$ is an injection and all $M_j$ are matchings.
  Furthermore, consider an edge $\ti{u}\ti{v}$ of $H$ with $\ti{u}\in\ti{U}_j$ for
  some $j\in[\Delta']$ and $\ti{v}\in\ti{V}$ and let
  \begin{equation*}
    u:=g(\ti{u})=M_j(\ti{u}) \qquad\text{and}\qquad
    v:=g(\ti{v})=f'(\ti{v}). 
  \end{equation*}
  It follows from the definition of $M_j$ that $\ti{u}u$ is an edge of the
  candidate graph $B_{f'}(H_j,G_j)$. Hence, by the definition of
  $B_{f'}(H_j,G_j)$, $u$ is an $f'$-candidate for $\ti{u}$, i.e.,
  \begin{equation*}
    f'\left(N_{H_j}(\ti{u})\right) \subset N_{G_j}(u).
  \end{equation*}
  Since $v=f'(\ti{v})\in f'\left(N_{H_j}(\ti{u})\right)$ this implies that $uv$ is
  an edge of $G$. Because $f'$ also satisfies 
  $f'(T)\not\in\hyper{B}$ for all
  $T\in\hyper{H}$ the embedding $g$ also meets the remaining requirement of the
  constrained blow-up lemma that no special $\Delta$-set is mapped to a
  forbidden $\Delta$-set.
\end{proof}

  For completing the proof of Lemma~\ref{lem:blowup}, we still need
  to prove Claim~\ref{cl:bl:matching} which we shall be occupied with for the
  remainder of this section. We will assume throughout that we have the same
  setup as in the preceding proof. In particular all constants, sets, and
  graphs are defined as there.

  For proving Claim~\ref{cl:bl:matching} we will use the matching lemma
  (Lemma~\ref{lem:matching}) on candidate graphs $B=B_{f}(H_j,G_j)$ and
  $B'=B_{f'}(H_j,G_j)$ for injections $f,f'\colon\ti{V}\to V'$. 
  As we will see, the following three claims imply that there are suitable
  $f$ and $f'$ such that the conditions of this lemma are satisfied.
  More precisely, Claim~\ref{cl:bl:small} will take care of
  conditions~\ref{lem:matching:i} and~\ref{lem:matching:ii} in this lemma,
  Claim~\ref{cl:bl:middle} of condition~\ref{lem:matching:iii}, and
  Claim~\ref{cl:bl:big} of condition~\ref{lem:matching:iv}.
  Before proving these claims we will show that they imply
  Claim~\ref{cl:bl:matching}.

  The first claim states that \emph{many} injective mappings $f\colon\ti V\to
  V'$ can be turned into injective mappings $f'$ (with the help of a few
  switchings) such that the candidate graphs ${B_{f'}}(H_j,G_j)$ for $f'$ satisfy certain
  degree and expansion properties.

  \begin{claim}\label{cl:bl:small}
    For \emph{at least half} of all injections $f:\ti{V}\to V'$ there is an
    injection $f':\ti{V}\to V'$ with $\swdist(f,f')\le\sigma n/r$ such that the
    following is satisfied for all $j\in[\Delta']$. For all $\ti{u}\in\ti{U}_j$
    and all $\ti{S}\subset\ti{U}_j$ with $|\ti{S}|\le p^{-\Delta}$ we have
    \begin{equation}\label{eq:bl:small}
      \deg_{{B_{f'}}(H_j,G_j)}(\ti{u})\ge (\tfrac{d}{2})^{\Delta}p^\Delta|U_j| 
      \qquad\text{and}\qquad
      |N_{{B_{f'}}(H_j,G_j)}(\ti{S})|\ge \nu np^\Delta|\ti{S}|.
    \end{equation}
     Further, 
     no special $\Delta$-set from $\hyper{H}$ is mapped to a forbidden
     $\Delta$-set from $\hyper{B}$ by $f'$.
  \end{claim}

  The second claim asserts that \emph{all} injective mappings $f'$ are such
  that the candidate graphs ${B_{f'}}(H_j,G_j)$ do not contain sets of certain
  sizes with too many edges between them.

  \begin{claim}\label{cl:bl:middle}
	\emph{All} injections $f':\ti{V}\to V'$ satisfy the following for all
	$j\in[\Delta']$ and all $S\subset U_j$, $\ti{S}\subset\ti{U}_j$.
    If $\nu n\le|S|<|\ti{S}|< \frac17(\frac d 2)^\Delta|U_j|$,
	then
	\begin{equation*}
      e_{B_{f'}(H_j,G_j)}(\ti{S},S)\le 7p^\Delta|\ti{S}||S|.
    \end{equation*}
  \end{claim}

  The last of the three claims states that for \emph{random} injective
  mappings~$f$ the graphs ${B_{f'}}(H_j,G_j)$ have edges between any pair of
  large enough sets $S\subset U_j$ and $\ti S\subset\ti{U}_j$.

  \begin{claim}\label{cl:bl:big}
    A \emph{random} injection $f:\ti{V}\to V'$ \aas\ satisfies the following.
    For all $j\in[\Delta']$ and all $S\subset U_j$, $\ti{S}\subset\ti{U}_j$ with
	$|\ti{S}|\ge\frac17(\frac d 2 )^\Delta|U_j|$ and $|S|>|U_j|-|\ti{S}|$ we have
    \begin{equation*}
      \left|N_{B_{f}(H_j,G_j)}(S)\cap\ti{S}\right|>2\sigma n/r.
    \end{equation*}
  \end{claim}

  \begin{proof}[Proof of Claim~\ref{cl:bl:matching}]
  Our aim is to apply the matching lemma (Lemma~\ref{lem:matching}) to the
  candidate graphs $B_f(H_j,G_j)$ and $B_{f'}(H_j,G_j)$ for all
  $j\in[\Delta']$ with carefully chosen injections $f$ and $f'$.
  
  Let $f:\ti{V}\to V'$ be an injection satisfying the assertions of
  Claim~\ref{cl:bl:small} and Claim~\ref{cl:bl:big} and let $f'$ be the injection
  promised by Claim~\ref{cl:bl:small} for this $f$. Such an $f$ exists as
  at least half of all injections from $\ti{V}$ to
  $V'$ satisfy the assertion of Claim~\ref{cl:bl:small} and almost all of those
  satisfy the assertion of Claim~\ref{cl:bl:big}. We
  will now show that for all $j\in[\Delta']$ the conditions of
  Lemma~\ref{lem:matching} are satisfied for input 
  \begin{gather*}
    B=B_f(H_j,G_j), \qquad B'=B_{f'}(H_j,G_j), \qquad
    s=2\sigma n/r\,,  \\
    x=\nu np^\Delta, \qquad
    n_1=(\tfrac{d}{2})^\Delta p^\Delta|U_j|, \qquad 
    n_2=p^{-\Delta}, \qquad
    n_3=\tfrac17(\tfrac d 2)^\Delta|U_j|, \qquad
  \end{gather*}
  Claim~\ref{cl:bl:small} asserts that $\swdist(f,f')\le\sigma
  n/r$. Since $\ti{U}_j$ is $2$-independent in $H$ we have
  $\deg_{H_j}(\ti{v})\le 1$ for all $\ti{v}\in\ti{V}$.  
  Thus the switching lemma, Lemma~\ref{lem:switch}, applied to $H_j$ and $G_j$
  and with $s$ replaced by $\sigma n/r$ implies
  \begin{equation*}
    \ndist[\ti{U}_j](B,B')=
    \ndist[\ti{U}_j]\Big(B_f(H_j,G_j),B_{f'}(H_j,G_j)\Big)\le 2\sigma n/r
    =s.
  \end{equation*}
  
  Moreover, by Claim~\ref{cl:bl:small}, for all
  $\ti{u}\in\ti{U}_j$ we have
  \begin{equation*}
    \deg_{B'}(\ti{u}) =  \deg_{B_{f'}(H_j,G_j)}(\ti{u}) 
    \ge (\tfrac{d}{2})^\Delta p^\Delta|U_j| = n_1
  \end{equation*}
  and thus condition~\ref{lem:matching:i} of Lemma~\ref{lem:matching} holds
  true. Further, we conclude from Claim~\ref{cl:bl:small} that
  $|N_{B'}(\ti{S})|\ge x|\ti{S}|$ for all $\ti{S}\subset\ti{U}_j$ with
  $|\ti{S}|<p^{-\Delta}=n_2$. This gives
  condition~\ref{lem:matching:ii} of Lemma~\ref{lem:matching}. In
  addition, Claim~\ref{cl:bl:middle} states that for all $S\subset U_j$, $\ti{S}\subset\ti{U}_j$ with
  $xn_2=\nu n\le|S|<|\ti{S}|<\tfrac17(\tfrac d 2)^\Delta|U_j|=n_3$ we have
  \begin{equation*}
       e_{B'}(\ti{S},S)=e_{B_{f'}(H_j,G_j)}(\ti{S},S)
       \le 7p^\Delta|\ti{S}||S|=\frac{n_1}{n_3}|\ti{S}||S|
  \end{equation*}
  and accordingly also condition~\ref{lem:matching:iii} of
  Lemma~\ref{lem:matching} is satisfied. To see~\ref{lem:matching:iv}, 
  observe that the choice of $f$ and Claim~\ref{cl:bl:big} assert
  \begin{equation*}
    \left|N_{B}(S)\cap\ti{S}\right| =
    \left|N_{B_{f}(H_j,G_j)}(S)\cap\ti{S}\right| >
    2\sigma n/r=s
  \end{equation*}
  for all $S\subset U_j$, $\ti{S}\subset\ti{U}_j$ with
  $|\ti{S}|\ge\tfrac17(\tfrac d 2)^\Delta|U_j|=n_3$ and $|S|>|U|-|\ti{S}|$.
  Therefore, all conditions of Lemma~\ref{lem:matching} are satisfied and we
  infer that for \emph{all} $j\in[\Delta']$ the candidate graph
  $B_{f'}(H_j,G_j)$ with $f'$ as chosen above has a matching covering $\ti{U}$.
  Moreover, by Claim~\ref{cl:bl:small}, 
  $f'$ maps no special $\Delta$-set to a forbidden $\Delta$-set. This
  establishes Claim~\ref{cl:bl:matching}.
  \end{proof}

  It remains to show Claims~\ref{cl:bl:small}--\ref{cl:bl:big}.
  We start with Claim~\ref{cl:bl:small}. For the proof of this claim we
  apply the hypergraph packing lemma (Lemma~\ref{lem:pack}).

  \begin{proof}[Proof of Claim~\ref{cl:bl:small}]
  Notice that~\ref{bl:U} on page~\pageref{bl:U} implies that $N_H(\ti{u})$
  contains exactly $\Delta$ elements for each $\ti{u}\in\ti{U}$. Hence we
  may define the following family of $\Delta$-sets.
  Let
  \begin{equation*}
    \hyper{N}:=\left\{N_H(\ti{u})\colon\; \ti{u}\in\ti{U}\right\}\cup\hyper{H} 
    \subset \binom{\ti{V}}{\Delta}.
  \end{equation*}
  We want to apply the hypergraph packing lemma (Lemma~\ref{lem:pack}) with $\Delta$,
  with $\ell$ replaced by $2\Delta+1$, 
  and with
  $\sigma$  to the hypergraphs with vertex sets $\ti{V}$ and $V'$ and edge sets
  $\hyper{N}$ and $\hyper{B}'$, respectively (see~\eqref{eq:bl:B} on
  page~\pageref{eq:bl:B}). We will first check that the necessary conditions are
  satisfied.

  Observe that
  \begin{equation*}
    |V'|\geByRef{eq:bl:V'} (1-\eta')|V|\ge(1-\eta')n/r
    \geByRef{eq:bl:n} n_{\subref{lem:pack}}, \qquad\text{and}\qquad
    |\ti{V}|\le|V'|\,.
  \end{equation*}
  Furthermore, a vertex $\ti{v}\in\ti{V}$ is neither contained in more than
  $\Delta$ sets from $\hyper{H}$ nor is $\ti{v}$ in $N_H(\ti{u})$ for more than
  $\Delta+1$ vertices $\ti{u}\in\ti{U}$ (by~\ref{bl:V} on
  page~\pageref{bl:V}). Therefore the condition Lemma~\ref{lem:pack} 
  imposes on $\hyper{N}$ is satisfied with $\ell$ replaced by $2\Delta+1$.
  Moreover, according to~\eqref{eq:bl:corrupt} no vertex in $V'$ is
  $\eta'_{\subref{lem:pack}}|V|$-corrupted by $\hyper{B}'$.
  Since 
  \begin{equation*}
    \eta'_{\subref{lem:pack}}|V| \leByRef{eq:bl:V'}
    \eta'_{\subref{lem:pack}}(1-\eta)^{-1}n' \leByRef{eq:bl:alphaetamu}
      \eta_{\subref{lem:pack}}n', 
  \end{equation*}
  this (together with the observation in Definition~\ref{def:corrupt}) implies
  that no vertex in $V'$ is $\eta_{\subref{lem:pack}}n'$-corrupted by
  $\hyper{B}'$ 
  and therefore all
  prerequisites of Lemma~\ref{lem:pack} are satisfied.

  It follows that the conclusion of Lemma~\ref{lem:pack}
  holds for at least half of all 
  injective functions $f\colon\ti{V}\to V'$,
  namely that 
  there are packings $f'$ of (the
  hypergraphs with edges) $\hyper{N}$ and $\hyper{B}$ with switching distance
  $\swdist(f,f')\le\sigma|\ti{V}|\le\sigma n/r$.
  Clearly, such a packing~$f'$ does not send
  any special $\Delta$-set from~$\hyper{H}$ to any forbidden $\Delta$-set
  from~$\hyper{B}$. Our next goal is to show that~$f'$ satisfies the first part
  of~\eqref{eq:bl:small} for all $j\in[\Delta']$ and $\ti{u}\in\ti{U}_j$. For
  this purpose, fix~$j$ and~$\ti{u}$. The definition of the candidate graph
  $B_{f'}(H_j,G_j)$, Definition~\ref{def:cand}, implies
  \begin{equation*}\begin{split}
    \deg_{{B_{f'}}(H_j,G_j)}(\ti{u}) &=
    \left|\Big\{ 
      u\in U_j\colon\, f'\left(N_{H_j}(\ti{u})\right)\subset N_{G_j}(u)
    \Big\}\right| \\
    &=
    \left|\bigg\{ 
      u\in U_j\colon\,
      u\in\coN_{G_j}\!\Big(f'\left(N_{H_j}(\ti{u})\right)\!\Big) \bigg\}\right| \\
    &=\left|\coN_{G_j}\!\Big(f'\left(N_{H_j}(\ti{u})\right)\!\Big)\right|
    \ge (\tfrac12 d)^\Delta p^\Delta|U_j|\,.
  \end{split}\end{equation*}
  where the first inequality follows from the fact that
  $N_{H_j}(\ti{u})\in\hyper{N}$ and thus, as $f'$ is a packing of $\hyper{N}$
  and $\hyper{B}'$, we have
  $f'(N_{H_j}(\ti{u}))\not\in\BAD{\Delta}{
    d/2}{d}(V,U_j)\subset\hyper{B'}$ (see the definition of $\hyper{B'}$
  in~\eqref{eq:bl:B}).
  This in turn means that all $\Delta$-sets $f'(N_{H_j}(\ti{u}))$ with
  $\ti{u}\in\ti{U}_j$ are $p$-good (see Definition~\ref{def:bad}) in
  $(V,U_j)$, because~$(V,U_j)$ has $p$-density at least
  $d-\eps\Delta'\ge\tfrac{d}2$ by~\eqref{eq:bl:dense} and~\eqref{eq:bl:eps}. With this information at hand we
  can proceed to prove the second part of~\eqref{eq:bl:small}. Let $\ti{S}\subset\ti{U}_j$ with $\ti{S}<1/p^\Delta$
  and consider the family $\hyper{F}\subset\binom{V}{\Delta}$ with
  $$\hyper{F}:=\{ f'(N_H(\ti{u}))\colon\, \ti{u}\in\ti{S} \}.$$
  Because $U_j$ is $2$-independent in $H$ the sets $N_H(\ti{u})$ with
  $\ti{u}\in\ti{S}$ form a family of disjoint $\Delta$-sets in $\ti{V}$. 
  It follows that also the sets $f'(N_H(\ti{u}))$ with
  $\ti{u}\in\ti{S}$ form a family of disjoint $\Delta$-sets in $V$.
  By~\ref{bl:Pexp} on page~\pageref{bl:Pexp} the conclusion of
  Lemma~\ref{lem:exp} holds for $\Gamma$. We conclude that the pair $(V,U_j)$ is
  $(1/p^\Delta,\nu np^\Delta)$-expanding. Since
  $|\hyper{F}|=|\ti{S}|<1/p^\Delta$ by assumption and all members of~$\hyper{F}$
  are $p$-good in $(V,U_j)$ this implies that
  $|\coN_{U_j}(\hyper{F})|\ge\nu np^\Delta|\hyper{F}|$. 
  On the other
  hand $\coN_{U_j}(\hyper{F})=N_{B_{f'}(H_j,G_j)}(\ti{S})$ by the definition of
  $B_{f'}(H_j,G_j)$ and $\hyper{F}$ and thus we get the second part
  of~\eqref{eq:bl:small}.
  \end{proof}
	
  Recall
  that~property~\ref{bl:Pstars} states that
  $\Gamma$ satisfies the conclusion of Lemma~\ref{lem:stars-big} for
  certain parameters. 
  We will use this fact to prove Claim~\ref{cl:bl:middle}.

  \begin{proof}[Proof of Claim~\ref{cl:bl:middle}]
  Fix $f':\ti{V}\to V'$,
  $j\in[\Delta']$, $S\subset U_j$, and $\ti{S}\subset\ti{U}_j$ with
  $\nu n\le|S|<|\ti{S}|<\frac17 (\frac d 2)^\Delta|U_j|$. For the candidate
  graphs $B_{f'}(H_j,G_j)$ of $f'$ we have
  \begin{equation*}\begin{split}
    e_{B_{f'}(H_j,G_j)}(\ti{S},S) 
    &= \left|\bigg\{ \ti{u}u\in\ti{S}\times S\colon\,
      f'\big(N_H(\ti{u})\big) \subset N_G(u) \bigg\}\right| \\
    &\eqByRef{eq:stars} \stars\!\bigg( 
      S, \Big\{ f'\big(N_H(\ti{u})\big)\colon\, \ti{u}\in\ti{S} \Big\} \bigg) \\
    &\le \stars[\Gamma]\! \bigg(
      S, \Big\{ f'\big(N_H(\ti{u})\big)\colon\, \ti{u}\in\ti{S} \Big\} \bigg)
    = \stars[\Gamma]\!(S,\hyper{F'}),
  \end{split}\end{equation*}
  where $\hyper{F'}:=\{ f'\big(N_H(\ti{u})\big)\colon\, \ti{u}\in\ti{S}\}$.
  As before the sets $f'(N_H(\ti{u}))$ with
  $\ti{u}\in\ti{S}$ form a family of $|\ti{S}|$ disjoint $\Delta$-sets in $V'$.
  Since 
  $\nu n\le|S|<|\ti{S}|=|\hyper{F}'|\le n$ we can appeal to
  property~\ref{bl:Pstars} (and hence Lemma~\ref{lem:stars-big}) with the set $X:=S$
  and the family $\hyper{F}'$ and infer that
  \begin{equation*}
    e_{B_{f'}(H_j,G_j)}(\ti{S},S) 
     \le \stars[\Gamma]\!(S,\hyper{F'})
     \le 7p^\Delta|\hyper{F'}||S|
     = 7p^\Delta|\ti{S}||S|
  \end{equation*}
  as required.
  \end{proof}

  Finally, we prove Claim~\ref{cl:bl:big}. For this proof we will use the fact
  that $\Delta$-sets in $p$-dense graphs have big common
  neighbourhoods (the conclusion of Lemma~\ref{lem:joint} holds by
  property~\ref{bl:Pjoint}) together with Lemma~\ref{lem:bl:sets}.

  \begin{proof}[Proof of Claim~\ref{cl:bl:big}]
  Let $f$ be an
  injective function from $\ti{V}$ to $V'$.
  First, consider a fixed
  $j\in[\Delta']$ and fixed sets $S\subset U_j$, $\ti{S}\subset\ti{U}_j$ with
  $|\ti{S}|\ge\frac17(\frac d 2)^\Delta|U_j|$ and $|S|>|U_j|-|\ti{S}|$. Define
  \begin{equation*}
    \hyper{S}:=\{N_{H_j}(\ti{u})\colon\, \ti{u}\in\ti{S} \}
    \qquad\text{and}\qquad \hyper{T}:=\BAD{\Delta}{d/2}{d}(V',S).
  \end{equation*}
  and observe that
  \begin{equation*}\begin{split}
    \left| N_{B_{f}(H_j,G_j)}(S)\cap\ti{S} \right|
    &= \left| \Big\{ \ti{u}\in\ti{S}\colon\, \exists u\in S\text{ with }
        f\left(N_{H_j}(\ti{u})\right)\subset N_{G_j}(u) \Big\} \right|\\ 
    &= \left| \bigg\{ \ti{u}\in\ti{S}\colon\,
         \coN_{G_j}\!\Big(f\left(N_{H_j}(\ti{u})\right)\!\Big)
         \cap S \neq \emptyset \bigg\} \right|\\
    &\ge \left| \Big\{ \ti{u}\in\ti{S}\colon\,
         f\left(N_{H_j}(\ti{u})\right) \not\in \BAD{\Delta}{d/2}{d}(V',S)
       \Big\} \right|
    = \left| f(\hyper{S})\setminus\hyper{T} \right|
  \end{split}\end{equation*}
  since all $\Delta$-sets $B\not\in \BAD{\Delta}{d/2}{d}(V',S)$ satisfy
  $|\coN_{G_j}(B)\cap S|\ge(\frac d 2)^\Delta p^\Delta|S|>0$.
  Thus, for proving the claim, it suffices to show that a random injection
  $f:\ti{V}\to V'$ violates $|f(\hyper{S})\setminus\hyper{T}|>2\sigma n/r$ with probability at most
  $5^{-|U_j|}$ because this implies that $f$ violates the conclusion of
  Claim~\ref{cl:bl:big} for \emph{some} $j\in[\Delta']$, and \emph{some}
  $S\subset U_j$, $\ti{S}\subset\ti{U}_j$ with probability at most
  $\bigO(2^{|U_j|} 2^{|\ti{U}_j|}\cdot 5^{-|U_j|})=\smallo(1)$.
  For this purpose, we will use the fact that the pair $(V',S)$ is $p$-dense.
  Indeed, observe that 
  \begin{equation*} 
    |S|>|U_j|-|\ti{S}|>|U_j|-|\ti{U}_j| = \frac{|U|-|\ti{U}|}{\Delta'}
       \ge \frac{\eta|U|}{\Delta'}
  \end{equation*}
  by the assumptions of the constrained blow-up lemma, Lemma~\ref{lem:blowup}.
  As $|V'|\ge(1-\eta)|V|$ by~\eqref{eq:bl:V'} we can apply
  Proposition~\ref{prop:subpairs} twice to infer from the
  $(\eps,d,p)$-density of $(V,U)$ that $(V',S)$ is $(\ti \eps,d,p)$-dense with
  $\ti \eps:=\eps\Delta'/(\eta(1-\eta))$. 
  Furthermore
  $\ti\eps\le\ti\eps_{\subref{lem:joint}}$
  by~\eqref{eq:bl:eps} and
  \begin{equation*}
    |V'|\geByRef{eq:bl:V'} (1-\eta)\frac{n}{r} 
    \geByRef{eq:bl:xi}\xi_{\subref{lem:joint}} n, \qquad\text{and}\qquad
    |S|>\frac{\eta|U|}{\Delta'}\ge\frac{\eta n}{r\Delta'} \geByRef{eq:bl:xi}
    \xi_{\subref{lem:joint}} n.
  \end{equation*}
  Hence we conclude from~\ref{bl:Pjoint} on page~\pageref{bl:Pjoint} (with
  $\mu=\mu_\isubsc{T}$) that
  $|\hyper{T}|=|\BAD{\Delta}{d/2}{d}(V',S)|\le\mu_\isubsc{T}|V'|^\Delta$.
  In addition
  \begin{equation}\label{eq:bl:HS}
    \tfrac 17 \left(\tfrac{d}{2}\right)^\Delta|U_j|
    \le |\ti{S}|=|\hyper{S}|
    \le|\ti{U}_j| \le (1-\eta)\frac{n}{\Delta'}
    \leByRef{eq:bl:V'} \frac{|V'|}{\Delta'} 
    \leByRef{eq:bl:alphaetamu} \frac1\Delta(1-\mu_\isubsc{S})|V'|.
  \end{equation}
  Thus, we can apply
  Lemma~\ref{lem:bl:sets} with $\Delta$, $\beta$, and $\mu_\isubsc{S}$ to
  $S=\ti{V}$, $T=V'$, and to $\hyper{S}$ and $\hyper{T}$ and conclude that $f$
  violates
  \begin{equation*}
    |f(\hyper{S})\setminus\hyper{T}|>(1-\beta)|\hyper{S}|
      \geByRef{eq:bl:HS} (1-\beta)
        \tfrac 17 \left(\tfrac{d}{2}\right)^\Delta|U_j|
      \ge \frac{(1-\beta)d^\Delta n}{7\cdot 2^\Delta r\Delta'}
      \ge \frac{(1-\beta)d^\Delta n}{100^\Delta r}
      \geByRef{eq:bl:betasigma} 2\sigma \frac{n}{r}
  \end{equation*}
  with probability at most
  $$\beta^{|\hyper{S}|} \le \beta^{\frac17(\frac d
  2)^\Delta|U_j|}\le5^{-|U_j|}$$ 
  where the first inequality follows
  from~\eqref{eq:bl:HS} and the second from~\eqref{eq:bl:betasigma}.
\end{proof}

\smallskip

\paragraph{\bf Acknowledgement}

We thank two anonymous referees for their comments.


\bibliographystyle{amsplain_yk}
\bibliography{sparse} 


\ifthen{\equal{\printappendix}{on}}{

\appendix 


\section{The connection lemma}\label{sec:CL}

The proof of Lemma~\ref{lem:CL} which we present in this appendix is inherent in
the proof of~\cite[Lemma~18]{KohRoeSchSze}. The only difference is that we have
a somewhat more special set-up here (given by the pre-defined partitions and candidate
sets). This set-up however is chosen exactly in such a way that this proof
continues to work if we adapt the parameters involved accordingly.

\begin{proof}[Proof of Lemma~\ref{lem:CL}]
  For the proof of Lemma~\ref{lem:CL} we use an inductive argument and 
  embed a partition class of $H$ into the corresponding partition
  class of $G$ one at a time. Before describing this strategy we will define
  two graph properties ${\rm D}_p(d_0,\eps',\mu,\eps,\xi)$ and ${\rm
  STAR}_p(k,\xi,\nu)$, which a random graph $\Gamma=\Gnp$ enjoys \aas\ for
  suitable sets of parameters. Then we will set up these parameters accordingly
  and define all other constants involved in the proof.

  For a fixed $n$-vertex graph $\Gamma$, fixed positive reals $d_0$, $\eps'$,
  $\mu$, $\eps$, $\xi$, and $\nu$, a fixed integer $k$, and a function $p=p(n)$
  we define the following properties of $\Gamma$.
  \begin{description}[leftmargin=0mm, listparindent=0mm, itemsep=2mm,
    topsep=2mm, style=sameline, font=\mdseries]
    \item[${\rm D}_p(d_0,\eps',\mu,\eps,\xi)$\,]\ 
      {\it We say that $\Gamma$ has property ${\rm
      D}_p(d_0,\eps',\mu,\eps,\xi)$ if it satisfies the property stated in
      Lemma~\ref{lem:reg-neighb} with these parameters and with $\Delta$,
      i.e., whenever $G=(X\dcup Y\dcup Z,E)$ is a tripartite subgraph of
      $\Gamma$ with the required properties, then it satisfies the conclusion of
      this lemma.}
    \item[${\rm STAR}_p(k,\xi,\nu)$\,]\ 
      {\it Similarly $\Gamma$ has
      property ${\rm STAR}_p(k,\xi,\nu)$ if $\Gamma$ has the property
      stated in Lemma~\ref{lem:stars-small} with $\Delta$ replaced by $k$,
      with parameters $\xi$, $\nu$, and for $p=p(n)$.}
  \end{description}

  Now we set up the constants.
  Let $\Delta,t$ and $d$ be given and assume without loss of generality that
  $d\le\frac14$.
  First we set
  \begin{equation}\label{eq:CL:mu}
    \mu=\frac{1}{4\Delta^2}
  \end{equation}
  and we fix $\eps_i$ for $i=t,t-1,\dots,0$ by setting
  \begin{equation}\label{eq:CL:eps}
  \begin{split}
    \eps_t &=\frac{d}{12\Delta t}\,,
    \qquad
    d_0:=d\,, 
    \qand \\
    \eps_{i-1}
      &=\min\big\{\eps(\Delta,d_0,\eps'=\eps_i,\mu),
          \,\eps_i\big\}\ \text{for}\ i=t,\dots,1\,,  
  \end{split}
  \end{equation}
  where $\eps(\Delta-1,d_0,\eps'=\eps_i,\mu)$ is given by
  Lemma~\ref{lem:reg-neighb}.
  We choose $\eps:=\eps_0$ and $\xi:=(d/100)^{\Delta}$ and receive $r_1$ as
  input. For each $k\in[\Delta]$ and each $r'\in[r_1]$
  Lemma~\ref{lem:stars-small} with $\Delta$ replaced by $k$ and with $\xi$
  replaced by $\xi/r'$ provides positive constants $\nu(k,r')$ and $c(k,r')$.
  Let $\nu$ be the minimum among the $\nu(k,r')$ and let
  $c_{\subref{lem:stars-small}}$ be the maximum among the $c(k,r')$ as we
  let both $k$ and $r'$ vary. Similarly
  Lemma~\ref{lem:reg-neighb} with input $\Delta-1$, $d_0$, $\eps'=\eps_i$,
  $\mu$, and $\xi$ replaced by $\xi/r'$ provides constants
  $c'(i,r')$ for $i\in[0,t]$ and $r'\in[r_1]$. We let
  $c_{\subref{lem:reg-neighb}}$ be the maximum among these $c'(i,r')$.
  Then we fix $c:=\max\{c_{\subref{lem:stars-small}},
  c_{\subref{lem:reg-neighb}}\}$, and receive $r\in[r_1]$ as input. Finally, we
  set
  \begin{equation}\label{eq:CL:xi}
    \xi_{\subref{lem:stars-small}}
    :=\xi_{\subref{lem:reg-neighb}}
    :=\xi/r
    =(d/100)^{\Delta}(1/r)\,.
  \end{equation}
  This finishes the definition of the constants.

  Let $p=p(n)\ge c(\log n/n)^{1/\Delta}$ and let 
  $\Gamma$ be a graph from $\Gnp$. By 
  Lemma~\ref{lem:stars-small}, Lemma~\ref{lem:reg-neighb}, and the choice of 
  constants the graph $\Gamma$ \aas\ satisfies 
  properties
  ${\rm D}_p(d,\eps_{i},\mu,\eps_{i-1},\xi_{\subref{lem:reg-neighb}})$ for
  all $i\in[t]$,
  and properties ${\rm STAR}_p(k,\xi_{\subref{lem:stars-small}},\nu)$ for
  all $k\in[\Delta]$. In the remainder of this proof we assume that $\Gamma$
  has these properties and show that then $\Gamma$ also satisfies the
  conclusion of Lemma~\ref{lem:CL}.
  
  Let $G\subset\Gamma$ and $H$ be arbitrary graphs satisfying the requirements
  stated in the lemma on vertex sets $W=W_1\dcup\dots\dcup
  W_t$ and $\ti{W}=\ti{W}_1\dcup\dots\dcup\ti{W}_t$, respectively.
  Let $h\colon\,\ti{W}\to[t]$ be the ``partition function'' for the vertex partition
  of $H$, i.e., 
  $$
    h(\ti{w})=j\quad\text{if and only if}\quad \ti{w}\in\ti{W}_j.
  $$
  For an integer $i\le h(\ti w)$ we denote by
  $$
    \ldeg^{i}(\ti w):=\big|N_H(\ti w)\cap \{\ti x\in \ti W\colon\, h(\ti x)\le
    i\}\big| 
  $$ 
  the left degree of $\ti w$ with respect to $\ti W_1\dcup\dots\dcup\ti W_i$.
  Clearly $\ldeg^{h(\ti w)}(\ti w)=\ldeg(\ti w)$.
  Before we continue, recall that each vertex $\ti w\in\ti W_i$ is equipped
  with a set $X_{\ti w}\subset V(\Gamma)\setminus W$
  and that we defined an external degree $\edeg(\ti{w})=|X_{\ti{w}}|$
  of $\ti{w}$ as well as a candidate set
  $C(\ti{w})=\coN_{W_i}(X_{\ti{w}})\subset W_i$. In the course of
  our embedding procedure, that we will describe below, we shall shrink this
  candidate set but keep certain invariants as we explain next.
 
  We proceed inductively and embed the vertex class $\ti{W}_i$ into $W_i$ one
  at a time, for $i=1,\dots,t$. To this end, we verify the following
  statement~\STi{} for $i=0,\dots,t$.
  
  \begin{enumerate}[label={\STi},leftmargin=*]
    \item There exists a partial embedding $\phi_{i}$ of
    $H[\bigcup_{j=1}^{i}\ti{W}_j]$ into $G[\bigcup_{j=1}^{i}W_j]$ such that
    for every $\ti{z}\in \bigcup_{j=i+1}^{t} \ti{W}_j$ there exists a candidate
    set $C_i(\ti{z})\subset C(\ti{z})$ given by
    \begin{enumerate}[label={\rm(\alph{*})},leftmargin=*]
      \item\label{ST:a}
      $C_i(\ti{z})=\bigcap\{N_{G}(\phi_{i}(\ti{x}))\colon\, \ti{x}\in
      N_H(\ti{z}) \text{ and } h(\ti x)\le i\}\cap C(\ti z)$,
    \end{enumerate}
    and satisfying
    \begin{enumerate}[label={\rm(\alph{*})},leftmargin=*,resume]
      \item\label{ST:b}
      $|C_i(\ti{z})|\geq (dp/2)^{\ldeg^{i}(\ti{z})}|C(\ti z)|$, and
      \item\label{ST:c}
      for every edge $\{\ti z,\ti z'\}\in E(H)$ with $h(\ti z),h(\ti z')>i$ the
      pair $\big(C_i(\ti z),C_i(\ti z')\big)$ is $(\eps_i,d,p)$-dense
      in $G$.
    \end{enumerate}
  \end{enumerate}

  Statement~\STi{} ensures the existence of a partial embedding of the first
  $i$ classes $\ti W_1,\dots,\ti W_i$ of $H$ into $G$ such that for every
  unembedded vertex $\ti z$ there exists a candidate set $C_i(\ti z)\subset
  C(\ti z)$ that is not too small (see part~\ref{ST:b}). Moreover, if we embed
  $\ti z$ into its candidate set, then its image will be adjacent to all
  vertices $\phi_i(\ti x)$ with $\ti x\in (\ti W_1\cup\dots\cup \ti W_i)\cap
  N_H(\ti z)$ (see part~\ref{ST:a}). The last property, part~\ref{ST:c}, says
  that for edges of $H$ such that none of the endvertices are embedded already
  the respective candidate sets induce $(\eps,d',p)$-dense pairs for some
  positive $d'$. This property will be crucial for the inductive proof.

  \begin{remark}
    \it In what follows we shall use the following convention. Since
    the embedding of $H$ into $G$ will be divided into $t$ rounds, we shall find
    it convenient to distinguish among the vertices of $H$.  We shall use $\ti
    x$ for vertices that have already been embedded, $\ti y$ for vertices that
    will be embedded in the current round, while $\ti z$ will denote vertices
    that we shall embed at a later step.
  \end{remark}

  Before we verify~\STi{} for $i=0,\dots,t$ by induction on $i$ we note
  that~\STi[t] implies that $H$ can be embedded into~$G$ in such a way that
  every vertex $\ti w\in\ti W$ is mapped to a vertex in its candidate set
  $C(\ti w)$. Consequently, verifying~\STi[t] concludes the proof of
  Lemma~\ref{lem:CL}.

  \paragraph{Basis of the induction: $i=0$.} We first verify~\STi[0]. In this
  case $\phi_0$ is the empty mapping and for every $\ti z\in\ti W$ we have, according
  to~\ref{ST:a}, $C_0(\ti z)=C(\ti z)$, as there is no vertex $\ti x\in N_H(\ti
  z)$ with $h(\ti x)\le 0$. Property~\ref{ST:b} holds because $C_0(\ti
  z)=C(\ti z)$ and $\ldeg^{0}(\ti{z})=0$ for every $\ti z\in\ti W$.  Finally,
  property~\ref{ST:c} follows from the property that $(C(\ti z),C(\ti z'))$ is
  $(\eps_0,d,p)$-dense by~\ref{lem:CL:Cdense} of Lemma~\ref{lem:CL}.

  \paragraph{Induction step: $i\to i+1$.} For the inductive step, we suppose
  that~$i<t$ and assume that statement~\STi{} holds; we have to construct
  $\phi_{i+1}$ with the required properties. Our strategy is as follows. In the
  first step, we find for every $\ti y\in\ti W_{i+1}$ an appropriate subset
  $C'(\ti y)\subseteq C_{i}(\ti y)$ of its candidate set such that if
  $\phi_{i+1}(\ti y)$ is chosen from $C'(\ti y)$, then the new candidate set
  $C_{i+1}(\ti z):= C_{i}(\ti z)\cap N_{G}(\phi_{i+1}(\ti y))$ of every
  ``right-neighbour'' $\ti z$ of $\ti y$ will not shrink too much and
  property~\ref{ST:c} will continue to hold.

  Note, however, that in general $|C'(\ti y)|\leq |C_i(\ti y)|=\smallo(n)\ll
  |\ti W_{i+1}|$ (if $\ldeg^{i}(\ti y)\ge 1$) and, hence, we cannot ``blindly''
  select $\phi_{i+1}(\ti y)$ from $C'(\ti y)$. Instead, in the second step, we
  shall verify Hall's condition to find a system of distinct representatives
  for the indexed set system $\big( C'(\ti y)\colon\,\ti
  y\in\ti W_{i+1}\big)$ and we let $\phi_{i+1}(\ti y)$ be the representative of
  $C'(\ti y)$. (A similar idea was used
  in~\cite{alon92:_spann_subgr_of_random_graph,RoRu}.) We now give the details
  of those two steps.

  \medskip
  \noindent{\sl First step:}
  For the first step, fix $\ti y\in\ti W_{i+1}$ and set 
  $$ N_H^{i+1}(\ti y):=\{\ti z\in N_H(\ti y)\colon\, h(\ti z)>i+1\}\,. $$ 
  A vertex $v\in C_{i}(\ti y)$ will be ``bad''
  (i.e., we shall not select $v$ for $C'(\ti y)$) if there exists a vertex
  $\ti z\in N_H^{i+1}(\ti y)$ for which $N_{G}(v)\cap C_{i}(\ti z)$, in view
  of~\ref{ST:b} and~\ref{ST:c} of \STi[i+1], cannot play the r\^ole of
  $C_{i+1}(\ti z)$.

  We first prepare for~\ref{ST:b} of \STi[i+1]. Fix a vertex $\ti z\in
  N_H^{i+1}(\ti y)$. Since $(C_i(\ti y),C_i(\ti z))$ is an
  $(\eps_i,d,p)$-dense 
  pair by~\ref{ST:c} of \STi{},
  Proposition~\ref{prop:typical} implies that there exist at most
  $\eps_{i}|C_{i}(\ti y)|\le\eps_{t}|C_{i}(\ti y)|$ vertices $v$ in $C_{i}(\ti y)$ such that 
  $$ |N_{G}(v)\cap C_i(\ti z)|<
    \big(d-\eps_t\big) p |C_i(\ti z)|\,. 
  $$ 
  Repeating the above for all $\ti z\in N_H^{i+1}(\ti y)$, we infer
  from~\ref{ST:a} and~\ref{ST:b} of \STi{}, that there are at most
  $\Delta\eps_t|C_i(\ti y)|$ vertices $v\in C_i(\ti y)$ such that the following
  fails to be true for some $\ti z\in N_H^{i+1}(\ti y)$:
  \begin{multline}\label{eq:CL:a'}
    |N_G(v)\cap C_i(\ti z)|\ge
    \big(d-\eps_t\big)p |C_i(\ti z)|\\
    \geBy{\text{\ref{ST:a},\ref{ST:b}}}
    \left(d-\eps_{t}\right)p
    \left(\frac{dp}{2}\right)^{\ldeg^{i}(\ti z)} |C(\ti z)|
    \geByRef{eq:CL:eps}
    \left(\frac{dp}{2}\right)^{\ldeg^{i+1}(\ti z)}|C(\ti z)|\,.
  \end{multline}
  
  For property~\ref{ST:c} of \STi[i+1], we fix an edge $e=\{\ti z,\ti z'\}$ with
  $h(\ti z)$, $h(\ti z')>i+1$ and with at least one end vertex in
  $N_H^{i+1}(\ti y)$. There are at most $\Delta(\Delta-1)< \Delta^2$ such
  edges. Note that if both vertices $\ti z$ and $\ti z'$ are neighbours of
  $\ti y$, i.e., $\ti z$, $\ti z'\in N_H^{i+1}(\ti y)$, then 
  $$
    \max\big\{
      \ldeg^i(\ti y)+\edeg(\ti y),
      \ldeg^i(\ti z)+\edeg(\ti z),
      \ldeg^i(\ti z')+\edeg(\ti z')
    \big\}\le\Delta-2\,,
  $$
  by~\ref{lem:CL:deg} of Lemma~\ref{lem:CL} and
  because all three vertices $\ti y$, $\ti z$, and $\ti z'$ have at least two
  neighbours in $\ti W_{i+1}\cup\dots\cup \ti W_{t}$. 
  From property~\ref{ST:b} of \STi{}, and~\ref{lem:CL:Wi}
  and~\ref{lem:CL:Cbig} of Lemma~\ref{lem:CL} we infer for all $\ti
  w\in\{\ti y,\ti z,\ti z'\}$ that
  \begin{equation*}
    |C_i(\ti w)|
    \geBy{\ref{ST:b}}\left(\frac{dp}{2}\right)^{\ldeg^i(\ti w)}|C(\ti w)|
    \geBy{\ref{lem:CL:Wi},\ref{lem:CL:Cbig}}
      \left(\frac{dp}{2}\right)^{\ldeg^i(\ti w)+\edeg(\ti w)}
      \frac{n}{r}
    \geByRef{eq:CL:xi} \xi_{\subref{lem:reg-neighb}} p^{\Delta-2}n.  
  \end{equation*}
  Furthermore, $\Gamma$ has property
  ${\rm D}_p(d,\eps_{i+1},\mu,\eps_{i},\xi_{\subref{lem:reg-neighb}})$ by
  assumption.
  This implies that there are at most $\mu |C_i(\ti y)|$ vertices $v\in
  C_i(\ti y)$ such that the pair $(N_G(v)\cap C_i(\ti z),
  N_G(v)\cap C_i(\ti z'))$ fails to be 
  $(\eps_{i+1},d,p)$-dense.

  If, on the other hand, say, only $\ti z\in N_H^{i+1}(\ti y)$ and
  $\ti z'\not\in N_H^{i+1}(y)$, then
  \begin{align*}
    \max\big\{
      \ldeg^i(\ti y)+\edeg(\ti y),
      \ldeg^i(\ti z')+\edeg(\ti z')
    \big\}
    & \le \Delta-1 \\
    \tand\quad
    \ldeg^{i}(\ti z)+\edeg(\ti z) &\le \Delta-2.
  \end{align*}
  Consequently, similarly as above,
  $$
    \min\Big\{|C_i(\ti y)|\,, |C_i(\ti z')|\Big\}
    \ge \xi_{\subref{lem:reg-neighb}} p^{\Delta-1}n
    \qand
    |C_i(\ti z)|\geq \xi_{\subref{lem:reg-neighb}} p^{\Delta-2}n
  $$
  and we can appeal to the fact that $\Gamma$ has property
  ${\rm D}_p(d,\eps_{i+1},\mu,\eps_{i},\xi_{\subref{lem:reg-neighb}})$ to
  infer that there are at most $\mu |C_i(\ti y)|$ vertices $v\in C_i(\ti y)$ such that
  $(N_G(v)\cap C_i(\ti z), C_i(\ti z'))$ fails to be
  $(\eps_{i+1},d,p)$-dense.  
  For a given~$v\in C_i(\ti y)$, let $\hat
  C_i(\ti z)=C_i(\ti z)\cap N_G(v)$ if $\ti z\in N_H^{i+1}(\ti y)$ and $\hat
  C_i(\ti z)=C_i(\ti z)$ if $\ti z\not\in N_H^{i+1}(\ti y)$, and define~$\hat
  C_i(\ti z')$ analogously.

  Summarizing the above we infer that there are at least
  \begin{equation}\label{eq:CL:finalC}
    (1-\Delta\eps_t-\Delta^2\mu)|C_i(\ti y)|
  \end{equation}
  vertices $v\in  C_i(\ti y)$ such that
  \begin{enumerate}[label={\rm(\alph{*}')},leftmargin=*,start=2]
    \item\label{ST:b'}
      $|N_G(v)\cap C_i(\ti z)|\ge (dp/2)^{\ldeg^{i+1}(\ti z)}|C(\ti z)|$ for
      every $\ti z\in N_H^{i+1}(\ti y)$ (see~\eqref{eq:CL:a'}) and
    \item\label{ST:c'}
      $(\hat C_i(\ti z), \hat C_i(\ti z'))$ is
      $(\eps_{i+1},d,p)$-dense 
      for all edges
      $\{\ti z,\ti z'\}$ of~$H$ with $h(\ti z)$, $h(\ti z')>i+1$ and
      $\{\ti z,\ti z'\}\cap N_H^{i+1}(\ti y)\neq\emptyset$.
  \end{enumerate}
  Let $C'(\ti y)$ be the set of those vertices~$v$ from $C_i(\ti y)$
  satisfying properties~\ref{ST:b'} and~\ref{ST:c'} above.  Recall
  that $\ldeg^i(\ti y)+\edeg(\ti y)=\ldeg^i(\ti y')+\edeg(\ti y')$ for all
  $\ti y$, $\ti y'\in\ti W_{i+1}$ and set
  \begin{equation}\label{eq:CL:k}
    k:=\ldeg^i(\ti y)+\edeg(\ti y)\ \text{for some $\ti y\in\ti W_{i+1}$.}
  \end{equation}
  Since $\ti y\in\ti W_{i+1}$ was arbitrary, we infer from
  property~\ref{ST:b} of \STi{},
  properties~\ref{lem:CL:Wi} and~\ref{lem:CL:Cbig} of Lemma~\ref{lem:CL},
  and the choices of $\mu$ and $\eps_t$ that
  \begin{multline}\label{eq:CL:finalC2}
    |C'(\ti y)|
    \geByRef{eq:CL:finalC} (1-\Delta\eps_t-\Delta^2\mu)|C_i(\ti y)|
    \geBy{\ref{ST:b}} (1-\Delta\eps_t-\Delta^2\mu)
      \left(\frac{dp}{2}\right)^{\ldeg^i(\ti y)}|C(\ti z)| \\
    \geBy{\ref{lem:CL:Wi},\ref{lem:CL:Cbig}} (1-\Delta\eps_t-\Delta^2\mu)
      \left(\frac{dp}{2}\right)^{k}\frac{n}{r}
    \geBy{\eqref{eq:CL:mu},\eqref{eq:CL:eps}}
      \left(\frac{dp}{10}\right)^k \frac{n}{r}\,.
  \end{multline}

  \medskip
  \noindent{\sl Second step:}
  We now turn to the aforementioned second part of the inductive step. Here we
  ensure the existence of a system of distinct representatives for the
  indexed set system
  $$\cC_{i+1}:=\Big( C'(\ti y) \colon\,\ti y\in\ti W_{i+1}\Big)\,.$$
  We shall appeal to \emph{Hall's condition} and show that for every subfamily
  $\cC'\subset \cC_{i+1}$ we have
  \begin{equation}\label{eq:CL:hall}
    |\cC'|\leq \bigg|\bigcup_{C'\in\cC'}C'\,\bigg|.
  \end{equation}
  Because of~\eqref{eq:CL:finalC2}, assertion~\eqref{eq:CL:hall}
  holds for all families $\cC'$ with $1\leq |\cC'|\leq (dp/10)^k n/r$.
  
  Thus, consider a family $\cC'\subset\cC_i$ with $|\cC'|>(dp/10)^k n/r$.
  For every $\ti y\in\ti W_{i+1}$ 
  we have a set $\ti K(\ti y)$ of $\ldeg^i(\ti y)$ already embedded
  vertices of $H$ such that $\ti K(\ti y)=N_H(\ti y)\setminus N_H^{i+1}(\ti y)$. 
  Let $K'(\ti y):=\phi_i(\ti K(\ti y))$ be the image of
  $\ti K(\ti y)$ in $G$ under $\phi_i$. Recall that $\ti y$ is equipped with a
  set $X_{\ti y}\subset V(\Gamma)\setminus W$ of size $\edeg(\ti y)$ in
  Lemma~\ref{lem:CL}.  We have $\ldeg^i(\ti y)+\edeg(\ti y)=k$ by~\eqref{eq:CL:k}.
  Hence, when we add the vertices of $X_{\ti y}$ to $K'(\ti y)$ we obtain a
  set $K(\ti y)=\{u_1(\ti y),\dots,u_k(\ti y)\}$ of $k$ vertices in $\Gamma$.
  Note that for two distinct vertices $\ti y$, $\ti y'\in\ti W_{i+1}$ the sets
  $\ti K(\ti y)$ and $\ti K(\ti y')$ are disjoint. This follows from the fact
  that the distance in~$H$ between $\ti y$ and $\ti y'$ is at least four
  by the $3$-independence of $\ti W_{i+1}$ (cf.~\ref{lem:CL:indep} of
  Lemma~\ref{lem:CL}) and if $\ti K(\ti y)\cap\ti K(\ti y')\neq\emptyset$, then
  this distance would be at most two. In addition $\big(X_{\ti y}:\ti y\in\ti
  W_{i+1}\big)$ consists of pairwise disjoint sets by
  hypothesis. Consequently, the sets $K(\ti
  y)$ and $K(\ti y')$ are disjoint as well and, therefore, 
  $$
    \cF:=\{K(\ti y)\colon\,C'(\ti y)\in\cC'\}
    \subset\{K(\ti y)\colon\,\ti y\in\ti W_{i+1}\}
    \subset \binom{V(\Gamma)}{k}
  $$ 
  is a family of $|\cC'|$
  pairwise disjoint $k$-sets in $V(\Gamma)$. Moreover, $C(\ti
  y)=\coN_{W_i}(X_{\ti y})$ by definition and so~\ref{ST:a} of \STi{}
  implies 
  $$ 
    C'(\ti y)
    \subset C(\ti y) \cap \bigcap_{v\in K'(\ti y)} N_\Gamma(v)
    =\bigcap_{v\in K(\ti y)} N_\Gamma(v)\,. 
  $$
  Let 
  $$ 
    U=\bigcup_{C'(\ti y)\in\cC'}C'(\ti y)\subset W_{i+1}\,, 
  $$
  and suppose for a contradiction that
  \begin{equation}
  \label{eq:CL:Hall_contra}
    |U|<|\cC'|=|\cF|.
  \end{equation}
  We now use the fact that $\Gamma$ has property ${\rm
  STAR}_p(k,\xi_{\subref{lem:stars-small}},\nu)$ and apply it to~$U$ and~$\cF$
  (see Lemma~\ref{lem:stars-small}). By assumption $|U|<|\cF|\le\nu
  np^k|\cF|$. We deduce that
  $$
    \stars[\Gamma](U,\cF)\leq 
      p^k|U||\cF|+6\xi_{\subref{lem:stars-small}} np^k|\cF|\,.
  $$
  On the other hand, because of~\eqref{eq:CL:finalC2}, we have
  $$
    \stars[\Gamma](U,\cF)\ge
      \left(\frac{dp}{10}\right)^k \frac{n}{r}|\cF|\,.    
  $$
  Combining the last two inequalities we infer from property~\ref{lem:CL:Wi} of
  Lemma~\ref{lem:CL} that
  $$
    |U|
    \geq \left(\left(\frac{d}{10}\right)^k\frac{1}{r}
      -6\xi_{\subref{lem:stars-small}} \right)n
    \geByRef{eq:CL:xi} \xi_{\subref{lem:stars-small}} n
    \eqByRef{eq:CL:xi} \xi\frac{n}{r}
    \geBy{\ref{lem:CL:Wi}} |\ti W_{i+1}|
    \geq|\cC'|,
  $$
  which contradicts~\eqref{eq:CL:Hall_contra}.  This contradiction shows
  that~\eqref{eq:CL:Hall_contra} does not hold, that is, Hall's
  condition~\eqref{eq:CL:hall} does hold.  Hence, there exists a system of
  representatives for $\cC_{i+1}$, i.e., an injective mapping $\psi\colon\ti
  W_{i+1}\to \bigcup_{\ti y\in \ti W_{i+1}}C'(\ti y)$ such that $\psi(\ti y)\in
  C'(\ti y)$ for every $\ti y\in\ti W_{i+1}$.
  
  Finally, we extend $\phi_i$. For that we set 
  $$
    \phi_{i+1}(\ti w)=\begin{cases}
      \phi_i(\ti w)\,,&\text{if $\ti w\in\bigcup_{j=1}^{i} \ti W_j$},\\
      \psi(\ti w)\,,&\text{if $\ti w\in \ti W_{i+1}$}\,.
    \end{cases}
  $$
  Note that every $\ti z\in\bigcup_{j=i+2}^t \ti W_j$ has at most one neighbour
  in $\ti W_{i+1}$, as otherwise there would be two vertices $\ti y$ and
  $\ti y'\in \ti W_{i+1}$ with distance at most $2$ in $H$, which contradicts
  property~\ref{lem:CL:indep} of Lemma~\ref{lem:CL}. Consequently,
  for every $\ti z\in\bigcup_{j=i+2}^t \ti W_j$ we have
  $$
    C_{i+1}(\ti z)=\begin{cases}
      C_i(\ti z)\,,&\text{if $N_H(\ti z)\cap \ti W_{i+1}=\emptyset$},\\
      C_i(\ti z)\cap N_G(\phi_{i+1}(\ti y)) \,,&
        \text{if $N_H(\ti z)\cap\ti W_{i+1}=\{\ti y\}$}.
    \end{cases}
  $$
  by~\ref{ST:a} of \STi[i+1].
  In what follows we show that $\phi_{i+1}$ and $C_{i+1}(\ti z)$ for every
  $\ti z\in\bigcup_{j=i+2}^t \ti W_j$ have the desired properties and validate
  \STi[i+1].

  First of all, from~\ref{ST:a} of \STi{}, combined with $\phi_{i+1}(\ti y)\in
  C'(\ti y)\subseteq C_i(\ti y)$ for every~$\ti y\in\ti W_{i+1}$ and the
  property that $\big(\phi_{i+1}(\ti y)\colon\,\ti y\in\ti W_{i+1}\big)$ is a
  system of distinct representatives, we infer that $\phi_{i+1}$ is indeed a partial
  embedding of $H[\bigcup_{j=1}^{i+1}W_j]$.

  Next we shall verify property~\ref{ST:b} of \STi[i+1]. So
  let $\ti z\in \bigcup_{j=i+2}^t \ti W_j$ be fixed. If $N_H(\ti z)\cap
  \ti W_{i+1}=\emptyset$, then $C_{i+1}(\ti z)=C_{i}(\ti z)$,
  $\ldeg^{i+1}(\ti z)=\ldeg^{i}(\ti z)$, which yields~\ref{ST:b}
  of \STi[i+1] for that case. If, on the other hand, $N_H(\ti z)\cap
  \ti W_{i+1}\neq\emptyset$, then there exists a unique neighbour $\ti y\in
  \ti W_{i+1}$ of~$H$ (owing to the $3$-independence
  of $W_{i+1}$ by property~\ref{lem:CL:indep} of Lemma~\ref{lem:CL}). As
  discussed above we have
  $C_{i+1}(\ti z)=C_{i}(\ti z)\cap N_G(\phi_{i+1}(\ti y))$ 
  in this case. Since $\phi_{i+1}(\ti y)\in
  C'(\ti y)$, we infer directly from~\ref{ST:b'} that~\ref{ST:b} of
  \STi[i+1] is satisfied.

  Finally, we verify property~\ref{ST:c} of \STi[i+1]. Let $\{\ti z,\ti z'\}$ be
  an edge of $H$ with $\ti z,\ti z'\in\bigcup_{j=i+2}^t \ti W_j$.  We consider
  three cases, depending on the size of $N_H(\ti z)\cap\ti W_{i+1}$ and of
  $N_H(\ti z')\cap\ti W_{i+1}$. If $N_H(\ti z)\cap\ti W_{i+1}=\emptyset$ and
  $N_H(\ti z')\cap\ti W_{i+1}=\emptyset$, then part~\ref{ST:c} of \STi[i+1]
  follows directly from part~\ref{ST:c} of \STi{} and $\eps_{i+1}\geq\eps_i$,
  combined with $C_{i+1}(\ti z)=C_{i}(\ti z)$, $C_{i+1}(\ti z')=C_{i}(\ti z')$. 
  If $N_H(\ti z)\cap\ti W_{i+1}=\{\ti y\}$ and $N_H(\ti z')\cap
  \ti W_{i+1}=\emptyset$, then~\ref{ST:c} of \STi[i+1] follows from~\ref{ST:c'}
  and the definition of $C_{i+1}(\ti z)$ and $C_{i+1}(\ti z')$.  If
  $N_H(\ti z)\cap\ti W_{i+1}=\{\ti y\}$ and $N_H(\ti z')\cap\ti
  W_{i+1}=\{\ti y'\}$, then $\ti y=\ti y'$, as otherwise there would be a
  $\ti y$-$\ti y'$-path in $H$ with three edges, contradicting
  the $3$-independence of $\ti W_{i+1}$. Consequently, \ref{ST:c} of \STi[i+1]
  follows from~\ref{ST:c'} and the definition of $C_{i+1}(\ti z)$ and
  $C_{i+1}(\ti z')$.

  We have therefore verified \ref{ST:a}--\ref{ST:c} of \STi{}, thus concluding
  the induction step. The proof of Lemma~\ref{lem:CL} follows by induction.
\end{proof}


\section{Proofs of auxiliary lemmas}\label{sec:aux}

In this section we provide all proofs that were postponed earlier, namely
those of Lemma~\ref{lem:reduced}, Lemma~\ref{lem:stars-small}, Lemma~\ref{lem:joint}, and
Lemma~\ref{lem:Bad}.

\subsection{Proof of
Lemma~\ref{lem:reduced}}\label{sec:aux:reduced}

For the proof of Lemma~\ref{lem:reduced} we need the following lemma which collects
some well known facts about the edge distribution in random graphs $\Gnp$ and
follows directly from the Chernoff bound for binomially distributed random
variables.

\begin{lemma}
  \label{lem:Gnp}
  If $\log^4n/(pn)=\smallo(1)$ then \aas\ the random graph $\Gamma=\Gnp$
  has the following properties. For all vertex sets $X$, $Y$, $Z\subset
  V(\Gamma)$ with $X\cap Y=\emptyset$ and $|X|,|Y|\ge\frac{n}{\log n}$
  we have
  \begin{enumerate}[label=\irom]
      \item\label{lem:Gnp:eX}
        $e_\Gamma(X)=(1\pm\frac{1}{\log n})p\binom{|X|}{2}$,
      \item\label{lem:Gnp:eXY}
        $e_\Gamma(X,Y)=(1\pm\frac{1}{\log n})p|X||Y|$,
      \item\label{lem:Gnp:deg}
        $\sum_{z\in Z}\deg_\Gamma(z)=(1\pm\frac{1}{\log n})p|Z|n$.
  \qed
  \end{enumerate}
\end{lemma}

\begin{proof}[Proof of Lemma~\ref{lem:reduced}]
  For the proof we will use the sparse regularity lemma
  (Lemma~\ref{lem:sparse-RL}) and the facts about the edge
  distribution in random graphs given by Lemma~\ref{lem:Gnp}.
  
  Given $\alpha$, $\eps$, and $r_0$ let $r_1$, $\nu$, and $n_0$ be as provided
  by Lemma~\ref{lem:sparse-RL} for input 
  $$
    \eps':=\eps^2/1000\,, \quad
    K:=1+\eps'\,, \quad\text{and}\quad
    r'_0:=\max\{2r_0,\lceil 1/\eps'\rceil\}\,.
  $$
  Let further~$d$ be given and assume that~$n$ is such that $n\ge
  n_0$, $\log n\ge 1/\eps'$, and $\log n\ge 1/\nu$.  Let~$\Gamma$ be a
  typical graph from~$\Gnp$ with $\log^4n/(pn)=\smallo(1)$, by which
  we mean here that~$\Gamma$ should satisfy properties
  \ref{lem:Gnp:eX}--\ref{lem:Gnp:deg} of Lemma~\ref{lem:Gnp}. We will
  show that, then, $\Gamma$ also satisfies the conclusion of
  Lemma~\ref{lem:reduced}.
  
  To this end we consider an arbitrary
  subgraph~$G=(V,E)$ of~$\Gamma$ that satisfies the assumptions of
  Lemma~\ref{lem:reduced}.
  By property~\ref{lem:Gnp:eXY} of Lemma~\ref{lem:Gnp} the graph
  $G\subset\Gamma$ is $(1/\log n, 1+1/\log n)$-bounded with respect to $p$.
  Since we have $1+1/\log n\le 1+\eps'=K$,
  the sparse regularity lemma (Lemma~\ref{lem:sparse-RL}) with input $\eps'$,
  $K$, and $r'_0$ asserts that $G$ has an $(\eps',p)$-regular
  $\epsilon$-equipartition $V=V'_0\dcup V'_1\dcup\dots\dcup V'_{r'}$ for some
  $r'_0\le r'\le r_1$. Observe that there are at most $r'\sqrt{\eps'}$
  clusters in this partition which are contained in more than $r'\sqrt{\eps'}$
  pairs that are not $(\eps',p)$-regular. We add all these clusters to $V'_0$,
  denote the resulting set by $V_0$ and let the remaining clusters be
  $V_1,\dots,V_r$. Then $r_0\le r'/2 \le r\le r_1$.  We claim that the
  partition $V=V_0\dcup V_1\dcup\dots\dcup V_{r}$ has the desired properties. 

  Indeed, $|V_0|\le\eps'n+r'\sqrt{\eps'}(n/r')\le\eps n$ and the
  number of pairs in $V_1\dcup\dots\dcup V_{r}$ which are not
  $(\eps,p)$-regular is at most $r\cdot r'\sqrt{\eps'}\le
  2r^2\sqrt{\eps'}\le\eps\binom{r}{2}$.  It follows that
  $V_1\dcup\dots\dcup V_{r}$ is an $(\eps,p)$-regular partition.  Let
  $R$ be the (edge maximal) reduced graph for the given paramter~$d$,
  so that, $R$ has vertex set $[r]$ and edges $ij$ for exactly all the
  $(\eps,d,p)$-dense pairs $(V_i,V_j)$ with $i$, $j\in[r]$.  It
  remains to show that we have $\delta(R)\ge(\alpha-d-\eps)|R|$.
  
  To see this, define $L:=|V_i|$ ($i\in[r]$) and consider arbitrary
  disjoint sets $X,Y\subset V(G)$.
  Then $\sum_{x\in X}
  \deg_G(x)=2e_G(X)+e_G(X,Y)+e_G(X,V\setminus(X\cup Y))$ and therefore
  \begin{equation*}
	e_G(X,Y) \ge \Big(\alpha\sum_{x\in X}\deg_\Gamma(x) \Big) 
	  - 2e_\Gamma(X) - e_\Gamma\big(X,V\setminus(X\cup Y)\big).
  \end{equation*}
  By properties~\ref{lem:Gnp:eX}--\ref{lem:Gnp:deg} of Lemma~\ref{lem:Gnp},
  if $|X|\ge n/\log n$ and $|X\cup Y|\le n-n/\log n$, then this
  implies
  \begin{equation}\label{eq:reduced}
  \begin{split}
    e_G(X,Y) 
    &\ge 
        \alpha\Big(1-\frac1{\log n}\Big)p|X|n
      - 2\Big(1+\frac1{\log n}\Big)p\binom{|X|}{2} \\
      &\hspace{2.5cm}
      - \Big(1+\frac1{\log n}\Big)p|X|\Big(n-|X|-|Y|\Big) \\
    &\ge  \big(
        \alpha(1-\eps')n
      - (1+\eps')(n-|Y|)
    \big) p|X| \,.
  \end{split}
  \end{equation}  
  Now fix $i\in[r]$ and let $\bar{V}_i:=V\setminus(V_0\cup V_i)$. Then
  \begin{equation*}
    e_G(V_i,\bar{V}_i)\le
      \big(\deg_R(i)+2r\sqrt{\eps'}\big) \left(1+\eps'\right)pL^2
      +\big(r-\deg_R(i)\big)dpL^2, 
  \end{equation*}
  since each cluster is contained in at most
  $r'\sqrt{\eps'}\le2r\sqrt{\eps'}$ $(\epsilon',p)$-irregular pairs and
  because~$R$ is a maximal $(\eps',d,p)$-reduced graph and $G\subset\Gamma$
  is $(1/\log n, 1+\eps')$-bounded with respect to $p$. On the other
  hand,~\eqref{eq:reduced} implies that
  \begin{equation*}\begin{split}
    e_G(V_i,\bar{V}_i)
    & \ge \Big(\alpha(1-\eps')n-(1+\eps')\big(|V_0|+|V_i|\big)\Big)p|V_i| \\
    & \ge \Big(\alpha(1-\eps')-(1+\eps')3\sqrt{\eps'}\Big)pLn,
  \end{split}\end{equation*}
  where we used $|V_0|\le(\eps'+\sqrt{\eps'})n$ and $|V_i|\le n/r'_0\le\eps'n$.
  We conclude that 
  \begin{equation*}
    \Big(\!\deg_R(i)(1+\eps'-d)+2r\sqrt{\eps'}(1+\eps')+rd\Big)pL^2
    \ge \Big(\alpha(1-\eps')-(1+\eps')3\sqrt{\eps'}\Big)prL^2,
  \end{equation*}
  since $n/L\ge r$.
  This gives
  \begin{multline*}
   \deg_R(i)(1+\eps'-d)
    \ge \Big(\alpha(1-\eps')-(1+\eps')3\sqrt{\eps'}-2\sqrt{\eps'}(1+\eps') 
    -d\Big)r \\ 
    \ge \Big(\alpha-\alpha\eps'-9\sqrt{\eps'}-d\Big)r
    \ge(\alpha-d-\eps/2)|R|\,.
  \end{multline*}
  Thus, $\deg_R(i)\geq(\alpha-d-\eps/2)(1+\eps'-d)^{-1}|R|
  \geq(\alpha-d-\eps/2)(1+\eps')^{-1}|R|
  \geq(\alpha-d-\eps/2)(1-\eps')|R|\geq(\alpha-d-\eps)|R|$.
  Hence the $(\eps,d,p)$-dense partition $V=V_0\dcup
  V_1\dcup\dots\dcup V_{r}$ has a reduced graph~$R$ with
  $\delta(R)\ge(\alpha-d-\eps)|R|$.
\end{proof}


\subsection{Proof of
Lemma~\ref{lem:stars-small}}\label{sec:aux:stars-small}

This proof makes use of a Chernoff bound for the
binomially distributed random variable $\stars[\Gamma](X,\hyper{F})$
appearing in this lemma (cf.\ Definition~\ref{def:stars} and the discussion
below this definition).

\begin{proof}[Proof of Lemma~\ref{lem:stars-small}]
  Given $\Delta$ and $\xi$ let $\nu$ and $c$ be constants satisfying
  \begin{equation}\label{eq:stars-small:cnu}
  \begin{alignedat}{2}
    -6\xi\log(2\xi) &\le -(6\xi-2\sqrt\nu)\log\xi, &\qquad
    2\nu &\le(\sqrt{\nu}-2\nu), \\
    \Delta+1-6\xi c^\Delta&\le -1, \qquad\text{and} &
    \Delta &\le\nu c^\Delta.    
  \end{alignedat}
  \end{equation}
  First we estimate the probability that there are $X$ and $\hyper{F}$ with
  $|\hyper{F}|\ge n/\log n$ fulfilling the requirements of the lemma but
  violating~\eqref{eq:stars-small}.  
  Chernoff's inequality $\Prob[Y\ge\Exp
  Y+t]\le\exp(-t)$ for a binomially distributed random variable $Y$ and $t\ge
  6\Exp Y$ 
  (see~\cite[Chapter~2]{purpleBook}):
  implies
  \begin{equation*}
    \Prob\Big[\stars[\Gamma](X,\hyper{F})
      \ge p^\Delta|X||\hyper{F}|+6\xi np^\Delta|\hyper{F}|\Big]
    \le\exp(-6\xi np^\Delta|\hyper{F}|)
    \le\exp(-6\xi c^\Delta|\hyper{F}|\log n)
  \end{equation*}
  for fixed $X$ and $\hyper{F}$ since $6\xi np^\Delta|\hyper{F}|\ge 6
  p^\Delta|X||\hyper{F}|$. As the number of choices for $\hyper{F}$ and $X$ can
  be bounded by $\sum_{f =n/\log n}^{\xi n}n^{\Delta f}$ and $2^n\le\exp(n)$,
  respectively, the probability we want to estimate is at most
  \begin{equation*}
    \sum_{f=\frac{n}{\log n}}^{\xi n} \!\!
      \exp\Big(\Delta f\log n + n - 6\xi c^\Delta f\log n\Big) 
    \le \sum_{f=\frac{n}{\log n}}^{\xi n} \!\!
      \exp\big(f\log n(\Delta+1-6\xi c^\Delta)\big),
  \end{equation*}
  which does not exceed $\xi n\exp(-n)$ by~\eqref{eq:stars-small:cnu} and thus
  tends to $0$ as $n$ tends to infinity.
  
  It remains to establish a similar bound on the probability that there are
  $X$ and $\hyper{F}$ with $|\hyper{F}|<n/\log n$ fulfilling the requirements
  of the lemma but violating~\eqref{eq:stars-small}. For this purpose we use that
  \begin{equation*}
    \Prob[Y\ge t]\le q^t\binom{m}{t}\le\exp\Big(-t\log\frac{t}{3qm}\Big)
  \end{equation*}
  for a random variable $Y$ with distribution $\Bin(m,q)$ and infer for fixed
  $X$ and $\hyper{F}$
  \begin{multline*}
    \Prob\Big[\stars[\Gamma](X,\hyper{F})
      \ge p^\Delta|X||\hyper{F}|+6\xi np^\Delta|\hyper{F}|\Big]
    \le \Prob\Big[\stars[\Gamma](X,\hyper{F})
      \ge 6\xi np^\Delta|\hyper{F}|\Big] \\
    \le\exp\left(-6\xi np^\Delta|\hyper{F}|\log\frac{2\xi n}{|X|}\right)
    \le\exp\left(-2\sqrt{\nu} np^\Delta|\hyper{F}|\log\frac{n}{|X|}\right).
  \end{multline*}
  because $-6\xi\log(2\xi)\le-(6\xi-2\sqrt{\nu})\log\xi\le
  (6\xi-2\sqrt{\nu})\log(n/|X|)$ by~\eqref{eq:stars-small:cnu}. The number of
  choices for $\hyper{F}$ and $X$ in total can be bounded by
  \begin{multline*}
    \sum_{f=1}^{\frac{n}{\log n}} \,\, \sum_{x=1}^{\nu np^\Delta f}
      n^{\Delta f}\binom{n}{x}
    \le \sum_{f=1}^{\frac{n}{\log n}} \,\, \sum_{x=1}^{\nu np^\Delta f}
      \exp\left(\Delta f\log n + \nu np^\Delta f\log\frac{en}{x}\right) \\ 
    \le \sum_{f=1}^{\frac{n}{\log n}} \,\, \sum_{x=1}^{\nu np^\Delta f}
      \exp\left(2\nu np^\Delta f\log\frac{en}{x}\right)
    \le \sum_{f=1}^{\frac{n}{\log n}} \,\, \sum_{x=1}^{\nu np^\Delta f}
      \exp\left(\sqrt{\nu}np^\Delta f\log\frac{n}{x}\right)
  \end{multline*}
  where the second inequality follows from $\Delta\log n\le\nu c^\Delta\log
  n\le \nu np^\Delta$ and the last from $2\nu\log
  e\le(\sqrt{\nu}-2\nu)\log(n/x)$ by~\eqref{eq:stars-small:cnu}. Therefore the
  probability under consideration is at most
  \begin{equation*}
    \sum_{f=1}^{\frac{n}{\log n}} \,\, \sum_{x=1}^{\nu p^\Delta nf}
      \exp\left(\sqrt{\nu}np^\Delta f\log\frac{n}{x} 
      - 2\sqrt{\nu}np^\Delta f\log\frac{n}{x}
      \right)\le n^2 \exp\left(-\sqrt{\nu}\log n \frac{n}{\log n}\right).
  \end{equation*}
\end{proof}


\subsection{Proof of
Lemma~\ref{lem:joint}}\label{sec:aux:joint}

We will use the following simple proposition about cuts in hypergraphs. This
proposition generalises the well known fact that any graph~$G$ admits a vertex
partition into sets of roughly equal size such that the resulting cut contains at least half
the edges of~$G$.

\begin{proposition}\label{prop:crosscut}
  Let $\hyper{G}=(V,\hyper{E})$ be an $\ell$-uniform hypergraph with $m$ edges
  and $n$ vertices such that $n\ge 3\ell$. Then there is a partition $V=V_1\dcup
  V_2$ with $|V_1|=\lfloor 2n/3 \rfloor$ and $|V_2|=\lceil n/3 \rceil$ such that at
  least $m\cdot\ell/2^{\ell+2}$ edges in $\hyper{E}$ are $1$-crossing, i.e.,
  they have exactly one vertex in $V_2$.
\end{proposition}
\begin{proof}
  Let $X$ be the number of $\frac13$-cuts of $V$, i.e., cuts $V=V_1\dcup V_2$
  with $|V_1|=\lfloor 2n/3 \rfloor$ and $|V_2|=\lceil n/3 \rceil$.  
  For a fixed edge $B$ there are precisely $2^\ell$ ways to distribute its
  vertices over $V_1\dcup V_2$ out of which exactly $\ell$ are such that $B$ is
  $1$-crossing. Further, for $r$ fixed vertices of $B$ exactly
  $\binom{n-\ell}{\lceil n/3 \rceil-r}$ of all $\frac13$-cuts have exactly these
  vertices in $V_2$. It is easy to check that 
  $$
    \binom{n-\ell}{\lceil n/3 \rceil-r}\le
    4\binom{n-\ell}{\lceil n/3 \rceil-1}
    \qquad\text{for all $0\le r\le\ell$\,.}
  $$
  It follows
  that $B$ is $1$-crossing for at least an $\frac14\ell/(2^\ell)$ fraction of
  all $\frac13$-cuts. Now assume that all $\frac13$-cuts have less than
  $m\cdot\ell/2^{\ell+2}$ edges that are $1$-crossing. Then double counting
  gives
  \begin{equation*}
   m\cdot\frac{\ell}{2^{\ell+2}}\cdot X>
   \sum_{B\in\hyper{E}} \#\big\{\,\text{$\tfrac13$-cuts s.t.\ $B$ is
   $1$-crossing}\,\big\} \ge m\cdot \frac14\frac{\ell}{2^\ell} \cdot X
  \end{equation*}
  which is a contradiction.
\end{proof}

In the proof of Lemma~\ref{lem:joint} we need to estimate the number of 
``bad'' $\ell$-sets in a vertex set $X$. For this purpose we will
use Proposition~\ref{prop:crosscut} to obtain a partition of~$X$ into sets
$X=X_1\dcup X_2$ such that a substantial proportion of all these bad
$\ell$-sets will be $1$-crossing and $X_1$ is not too small. In this way we
obtain many $(\ell-1)$-sets in $X_1$ most of which will, as we show, be
similarly bad as the $\ell$-sets we started with. This will allow us to
prove Lemma~\ref{lem:joint} by induction.

\begin{proof}[Proof of Lemma~\ref{lem:joint}]
  Let $\Delta$ and $d$ be given. Let $\Gamma$
  be an $n$-vertex graph, let $\ell$ be an integer, let $\eps'$, $\mu$,
  $\eps$, $\xi$ be positive real numbers, and let
  $p=p(n)$ be a function. We say that $\Gamma$ has property
  $P_\ell(\eps',\mu,\eps,\xi,p(n))$ if $\Gamma$ has the property stated in
  Lemma~\ref{lem:joint} with parameters 
  $\eps'$, $\mu$, $\eps$, $\xi$, $p(n)$ and with parameters
  and $\Delta$ and $d$.  
  Similarly, $\Gamma$ has property $D(\eps',\mu,\eps,\xi,p(n))$ if it satisfies
  the conclusion of Lemma~\ref{lem:reg-neighb} with these parameters and
  with $\Delta$ and $d_0:=d$. 
  For any fixed~$\ell>0$, we
  denote by ($\cP_\ell$) the following statement.
  \begin{itemize}[label={\rm($\cP_\ell$)}, leftmargin=*]
    \item {\it For all $\eps',\mu>0$ there is $\eps$ such that for all $\xi>0$
    there is $c>1$ such that a random graph $\Gamma=\Gnp$ with $p>c(\frac{\log
    n}{n})^{1/\Delta}$ has property $P_\ell(\eps',\mu,\eps,\xi,p(n))$ with
    probability $1-\smallo(1)$.}
  \end{itemize}
  We prove that ($\cP_\ell$) holds for every fixed $\ell>0$ by induction on
  $\ell$. The case $\ell=1$ is an easy consequence of
  Proposition~\ref{prop:typical} which states that in \emph{all}
  $(\eps,d,p)$-dense pairs most vertices have a large neighbourhood.
  
  For the inductive step assume that ($\cP_{\ell-1}$) holds. We will show that
  this implies ($\cP_\ell$). We start by specifying the constants appearing in
  statement ($\cP_\ell$).  
  Let $\eps'$ and $\mu$ be arbitrary positive
  constants. Set $\eps'_{\ell-1}:=\eps'$ and
  $\mu_{\ell-1}:=\frac1{100}\mu\tfrac{\ell}{2^{\ell+2}}$. Let $\eps_{\ell-1}$ be
  given by ($\cP_{\ell-1}$) for input parameters $\eps'_{\ell-1}$ and
  $\mu_{\ell-1}$. Set $\eps'_{\subref{lem:reg-neighb}}:=\eps_{\ell-1}$ and let
  $\eps_{\subref{lem:reg-neighb}}$ be as promised by Lemma~\ref{lem:reg-neighb}
  with parameters $\eps'_{\subref{lem:reg-neighb}}$ and
  $\mu_{\subref{lem:reg-neighb}}:=\frac{1}{2}$. Define
  $\eps:=\mu_{\ell-1}\eps_{\subref{lem:reg-neighb}}\eps'_{\ell-1}$.
  Next, let $\xi$ be an arbitrary parameter provided by
  the adversary in Lemma~\ref{lem:joint} and choose $\xi_{\ell-1}:=\xi(d-\eps)$
  and $\xi_{\subref{lem:reg-neighb}}:=\mu_{\ell-1}\xi$. Finally, let
  $c_{\ell-1}$ and $c_{\subref{lem:reg-neighb}}$ be given
  by ($\cP_{\ell-1}$) and by Lemma~\ref{lem:reg-neighb}, respectively,
  for the previously specified parameters together with $\xi_{\ell-1}$ and $\xi_{\subref{lem:reg-neighb}}$.
  Set $c:=\max\{c_{\ell-1},c_{\subref{lem:reg-neighb}}\}$.
  We will prove that with this choice of $\eps$ and $c$ the statement in
  ($\cP_\ell$) holds for the input parameters $\eps'$, $\mu$, and $\xi$.

  Let $\Gamma=\Gnp$ be a random graph.   
  By ($\cP_{\ell-1}$) and Lemma~\ref{lem:reg-neighb}, and by the choice of
  the parameters the graph $\Gamma$ has properties
  $$
    P_{\ell-1}(\eps'_{\ell-1},\mu_{\ell-1},\eps_{\ell-1},\xi_{\ell-1},p(n))
    \qqand
    D(\eps'_{\subref{lem:reg-neighb}},
    \mu_{\subref{lem:reg-neighb}}, \eps_{\subref{lem:reg-neighb}},
    \xi_{\subref{lem:reg-neighb}}, p(n)) 
  $$
  with probability $1-\smallo(1)$ if $n$ is large enough. We will show that a
  graph $\Gamma$ with these properties also satisfies $P_\ell(\eps',\mu,\eps,\xi,p(n))$. Let
  $G=(X\dcup Y,E)$ be an arbitrary subgraph of such a $\Gamma$ where
  $|X|=n_1$ and $|Y|=n_2$ with
  $n_1\ge\xi p^{\Delta-1}n$, $n_2\ge\xi p^{\Delta-\ell}n$, and $(X,Y)$
  is an $(\eps,d,p)$-dense pair.
  
  We would like to show that for
  $\hyper{B}_\ell:=\bad{\ell}{\eps'}{d}(X,Y)$ we have
  $|\hyper{B}_\ell|\le\mu n_1^\ell$. Assume for a contradiction that this is
  not the case. By Proposition~\ref{prop:crosscut} there is a cut $X=X_1\dcup X_2$ with
  $|X_1|=\lfloor 2n_1/3 \rfloor$ and $|X_2|=\lceil n_2/3 \rceil$ such that at
  least $|\hyper{B}_\ell|\cdot\ell/2^{\ell+2}$ of the $\ell$-sets in
  $\hyper{B}_\ell$ are $1$-crossing, i.e., have exactly one vertex in $X_2$.
  By Proposition~\ref{prop:typical} there are less than $\eps|X|$ vertices
  $x\in X_2$ such that $|N_Y(x)|<(d-\eps)p n_2$. We delete all $\ell$-sets from
  $\hyper{B}_\ell$ that contain such a vertex or are not $1$-crossing for
  $X=X_1\dcup X_2$ and call the resulting set $\hyper{B}'_\ell$. It follows
  that
  \begin{equation}\label{eq:joint:B'}
   |\hyper{B}'_\ell|
     \ge|\hyper{B}_\ell|\tfrac{\ell}{2^{\ell+2}}-\eps|X|n_1^{\ell-1}
     >\mu n_1^\ell\tfrac{\ell}{2^{\ell+2}}-\eps n_1^{\ell}
     \ge 20\mu_{\ell-1} n_1^{\ell}.
  \end{equation}
  Now, for each $v\in X_2$ we count the number of $\ell$-sets
  $B\in\hyper{B}'_\ell$ containing $v$. We delete all vertices $v$ from $X_2$
  for which this number is less than $|\hyper{B}'_\ell|/(10n_1)$ and call the
  resulting set $X'$. Observe that the definition of $\hyper{B}'_\ell$ implies
  that all vertices $x$ in $X'$ satisfy $|N_Y(x)|\ge(d-\eps)p n_2$.
  Because $\hyper{B}'_\ell$ contains only $1$-crossing $\ell$-sets we get
  \begin{equation*}
    |\hyper{B}'_\ell|
      \le |X_2\setminus X'|\frac{|\hyper{B}'_\ell|}{10n_1} + |X'|n_1^{\ell-1}
      \le \frac{|\hyper{B}'_\ell|}{10} + |X'|n_1^{\ell-1}
  \end{equation*}
  and thus
  \begin{equation*}
    |X'|\ge\frac{9}{10n_1^{\ell-1}}|\hyper{B}'_\ell|\geByRef{eq:joint:B'}
      10\mu_{\ell-1} n_1.
  \end{equation*}
  This together with Proposition~\ref{prop:subpairs} implies that the pairs
  $(X',Y)$ and $(Y,X_1)$ are $(\eps_{\subref{lem:reg-neighb}},d,p)$-dense.
  In addition we have $|X'|,|X_1|\ge\mu_{\ell-1} n_1\ge
  \mu_{\ell-1} \xi p^{\Delta-1}n= 
  \xi_{\subref{lem:reg-neighb}}p^{\Delta-1}n$
  and $|Y|\ge\xi p^{\Delta-\ell}n\ge\xi_{\subref{lem:reg-neighb}}
  p^{\Delta-2}n$. Because $\Gamma$ has property
  $D(\eps'_{\subref{lem:reg-neighb}}, \mu_{\subref{lem:reg-neighb}},
  \eps_{\subref{lem:reg-neighb}}, \xi_{\subref{lem:reg-neighb}}, p(n))$ we
  conclude for the tripartite graph $G[X'\dcup Y\dcup X_1]$
  that there are at least $|X'|-\mu_{\subref{lem:reg-neighb}}|X'|\ge
  1$ vertices $x$ in $X'$ such that $(N_Y(x),X_1)$ is
  $(\eps'_{\subref{lem:reg-neighb}},d,p)$-dense.
  Let $x^*\in X'$ be one of these vertices and set $Y':=N_Y(x^*)$.
  Thus $(Y',X_1)$ is $(\eps'_{\subref{lem:reg-neighb}},d,p)$-dense
  and since $X'$ only contains vertices with a large neighbourhood in $Y$ we
  have $|Y'|\ge(d-\eps)p n_2$. Furthermore, let $\hyper{B}'_\ell(x^*)$ be the family
  of $\ell$-sets in $\hyper{B}'_\ell$ that contain~$x^*$. Then
  $\hyper{B}'_\ell(x^*)$ contains $\ell$-sets with $\ell-1$ vertices in $X_1$
  and with one vertex, the vertex $x^*$, in $X_2$ because $\hyper{B}'_\ell$
  contains only $1$-crossing $\ell$-sets. By definition of $X'$ and because $x^*\in X'$ we
  have
  \begin{equation}\label{lem:joint:Bv}
    |\hyper{B}'_\ell(x^*)|\ge |\hyper{B}'_\ell|/(10n_1)
    \geByRef{eq:joint:B'} 2\mu_{\ell-1} n_1^{\ell-1}.
  \end{equation}

  For $B\in\hyper{B}'_\ell(x^*)$ let $\Pi_{\ell-1}(B)$ be the projection of $B$
  to $X_1$. This implies that $\Pi_{\ell-1}(B)$ is an $(\ell-1)$-set in
  $X_1$. In addition $N_{Y'}(\Pi_{\ell-1}(B))=N_Y(B)$ by definition of $Y'$ and
  hence $\Pi_{\ell-1}(B)$ has less than $(d-\eps')^\ell p^\ell n_2$
  common neighbours in $Y'$ because
  $B\in\hyper{B}'_\ell(x^*)\subset\hyper{B}_\ell$. Accordingly the family
  $\hyper{B}_{\ell-1}$ of all projections $\Pi_{\ell-1}(B)$ with
  $B\in\hyper{B}'_\ell(x^*)$ is a family of size $|\hyper{B}'_\ell(x^*)|$ and
  contains only $(\ell-1)$-sets $B'$ with
  \begin{equation*}
    |N_{Y'}(B')|
    \le (d-\eps')^\ell p^\ell n_2
    \le (d-\eps')^{\ell-1}p^{\ell-1}|Y'|
    = (d-\eps'_{\ell-1})^{\ell-1}p^{\ell-1}|Y'|.
  \end{equation*}
  This means
  $\hyper{B}_{\ell-1}\subset\bad{\ell-1}{\eps'_{\ell-1}}{d}(X_1,Y')$. 
  Recall that $(X_1,Y')$ is
  $(\eps'_{\subref{lem:reg-neighb}},d,p)$-dense by the choice of $x^*$.
  Because $|X|=n_1\ge\xi p^{\Delta-1}n$ and 
  $$
    |Y'|\ge(d-\eps)p n_2
    \ge(d-\eps)p\cdot \xi p^{\Delta-\ell}n =\xi_{\ell-1}p^{\Delta-(\ell-1)}n
  $$
  we can appeal to
  $P_{\ell-1}(\eps'_{\ell-1},\mu_{\ell-1},\eps_{\ell-1},\xi_{\ell-1},p(n))$  
  and conclude that
  \[ |\hyper{B}'_\ell(x^*)|=|\hyper{B}_{\ell-1}|
    \le|\bad{\ell-1}{\eps'_{\ell-1}}{d}(X,Y')|
    \le\mu_{\ell-1} n_1^{\ell-1}\,,
  \]
  contradicting~\eqref{lem:joint:Bv}.

  Because $G$ was arbitrary this shows that $\Gamma$ has property
  $P_\ell(\eps',\mu,\eps,\xi,p(n))$. Thus ($\cP_\ell$) holds, which finishes the
  proof of the inductive step.
\end{proof}

\subsection{Proof of Lemma~\ref{lem:Bad}}\label{sec:aux:Bad}

In this section we provide the proof of Lemma~\ref{lem:Bad} which 
examines the inheritance of $p$-density to
neighbourhoods of $\Delta$-sets.
For this purpose we will first establish a version of this lemma,
Lemma~\ref{lem:bad} below, which only considers $\Delta$-sets that are
crossing in a given vertex partition.

We need some definitions. Let $G=(V,E)$ be a graph, $X$ be a subset of its
vertices, and $X=X_1\dcup\dots\dcup X_T$
be a partition of $X$. Then, for integers $\ell,T>0$, we say that
an $\ell$-set $B\subset X$ is \emph{crossing} in
$X_1\dcup\dots\dcup X_T$ if there are indices $0<i_1<\dots<i_\ell<T$ such that
$B$ contains exactly one element in $X_{i_j}$ for each $j\in[\ell]$.
In this case we also write $B \in X_{i_1}\times\dots\times X_{i_\ell}$
(hence identifying crossing $\ell$-sets with $\ell$-tuples).

Now let  $p$, $\eps$, $d$ be positive
reals, and $Y$, $Z\subset V$ be vertex sets such that $X$, $Y$, and $Z$ are
mutually disjoint. Define 
\begin{equation*}
  \bad{\ell}{\eps}{d}(X_1,\dots,X_T;Y,Z)
\end{equation*}
to be the family of all those crossing $\ell$-sets $B$ in $X_1\dcup\dots\dcup
X_T$
that either satisfy $|\coN_Y(B)|<(d-\eps)^{\ell}p^{\ell}|Y|$
or have the property that $(\coN_Y(B),Z)$ is not $(\eps,d,p)$-dense in $G$.
Further, let
\begin{equation*}
  \Bad{\ell}{\eps}{d}(X_1,\dots,X_T;Y,Z)
\end{equation*}
be the family of crossing $\ell$-sets $B$ in $X_1\dcup\dots\dcup X_T$
that contain an $\ell'$-set $B'\subset B$ with
$\ell'>0$ such that $B'\in\bad{\ell'}{\eps}{d}(X_1,\dots,X_T;Y,Z)$. 

\begin{lemma}
\label{lem:bad}
  For all integers $\ell,\Delta>0$ and positive reals $d_0$, $\eps'$, and $\mu$
  there is $\eps$ such that for all $\xi>0$ there is $c>1$ such that if
  $p>c(\frac{\log n}{n})^{1/\Delta}$, then the following holds \aas\ for
  $\Gamma=\Gnp$. For $n_1,n_3\ge\xi p^{\Delta-1}n$ and $n_2\ge\xi
  p^{\Delta-\ell-1}n$ let $G=(X\dcup Y\dcup Z,E)$ be any tripartite subgraph
  of $\Gamma$ with $|X|=n_1$, $|Y|=n_2$, and $|Z|=n_3$. Assume further that
  $X=X_1\dcup\dots\dcup X_\ell$ with $|X_i|\ge\lfloor\frac{n_1}{\ell}\rfloor$
  and that $(X,Y)$ and $(Y,Z)$ are $(\eps,d,p)$-dense pairs with $d\ge d_0$.
  Then
  \begin{equation*}
    \big|\bad{\ell}{\eps'}{d}(X_1,\dots,X_\ell;Y,Z)\big|
    \le\mu n_1^\ell.
  \end{equation*}  
\end{lemma}
\begin{proof}
  Let $\Delta$ and $d_0$ be given. For a fixed $n$-vertex graph $\Gamma$,
  a fixed integer $\ell$ and fixed positive reals $\eps'$, $\mu$, $\eps$, $\xi$,
  and a function $p=p(n)$ we say that we say that a graph $\Gamma$ on $n$
  vertices has property $P_\ell(\eps',\mu,\eps,\xi,p(n))$ if $\Gamma$ has the
  property stated in the lemma for these parameters and for $\Delta$ and $d_0$,
  that is, whenever $G=(X\dcup Y\dcup Z,E)$ is a tripartite subgraph of $\Gamma$
  with the required properties, then~$G$ satisfies the conclusion of the lemma.
  For any fixed $\ell>0$,
  we denote by ($\cP_\ell$) the following
  statement. \begin{itemize}[label={\rm($\cP_\ell$)}, leftmargin=*]
    \item {\it For all $\eps',\mu>0$ there is $\eps$ such that for all $\xi>0$
    there is $c>1$ such that a random graph $\Gamma=\Gnp$ with $p>c(\frac{\log
    n}{n})^{1/\Delta}$ has property $P_\ell(\eps',\mu,\eps,\xi,p(n))$ with
    probability $1-\smallo(1)$.}
  \end{itemize}
  We prove that ($\cP_\ell$) holds for every fixed $\ell>0$ by induction on
  $\ell$.
  
  The case $\ell=1$ is an easy consequence of Lemma~\ref{lem:reg-neighb} and
  Proposition~\ref{prop:typical}. Indeed, let $\eps'$ and $\mu$ be
  arbitrary, let $\eps_{\subref{lem:reg-neighb}}$ be as given by
  Lemma~\ref{lem:reg-neighb} for $\Delta$, $d_0$, $\eps'$, and $\mu/2$ and
  fix $\eps:=\min\{\eps_{\subref{lem:reg-neighb}},\eps',\mu/2\}$. Let $\xi$ be
  arbitrary and pass it on to Lemma~\ref{lem:reg-neighb} for obtaining $c$.
  Now, let $\Gamma=\Gnp$ be a random graph. Then, by the choice of
  parameters, Lemma~\ref{lem:reg-neighb} asserts that the graph $\Gamma$ has
  the following property  with probability $1-o(1)$. Let $G=(X\dcup Y\dcup
  Z,E)$ be any subgraph with $X=X_1$ and $|X|=n_1$, $|Y|=n_2$, and
  $|Z|=n_3$, where $n_1,n_3\ge\xi p^{\Delta-1}n$ and $n_2\ge\xi p^{\Delta-2}n$,
  and $(X,Y)$ and $(Y,Z)$ are $(\eps,d,p)$-dense pairs. Then there are at most
  $\frac{\mu}{2}n_1$ vertices $x\in X$ such that $(N(x)\cap Y,Z)$ is not an
  $(\eps',d,p)$-dense pair in $G$. Because $\eps\le\mu/2$,
  Proposition~\ref{prop:typical} asserts that in every such $G$ there are at
  most $\frac{\mu}{2}n_1$ vertices $x\in X$ with $|N_Y(x)|<(d-\eps')p|Y|$. This
  implies that
  \begin{equation*}
    \big|\bad{1}{\eps'}{d}(X_1;Y,Z)\big|
    \le\mu n_1
  \end{equation*}
  holds with probability $1-\smallo(1)$ for all such subgraphs $G$ of the random
  graph $\Gamma$. Accordingly we get ($\cP_1$).

  For the inductive step assume that ($\cP_{\ell-1}$) and ($\cP_1$) hold. We
  will show that this implies ($\cP_\ell$). 
  Again, let $\eps'$ and $\mu$ be arbitrary positive constants. Let $\eps_1$ be
  as promised in the statement ($\cP_1$) for parameters
  $\eps'_1:=\eps'$ and $\mu_1:=\mu/2$. Set $\eps'_{\ell-1}:=
  \min\{\eps_1,\eps',\frac{\mu}{4}\}$, and let $\eps_{\ell-1}$  be given by
  ($\cP_{\ell-1}$) for parameters
  $\eps'_{\ell-1}$ and $\mu_{\ell-1}:=\frac{\mu}{4}$.
  We define $\eps:=\eps_{\ell-1}/(\ell+1)$. Next, let $\xi$ be an arbitrary
  parameter and choose
  \begin{equation}
  \label{eq:bad:xi}
    \xi_1:=\min\{\xi/(\ell+1), (d_0-\eps'_{\ell-1})^{\ell-1}\xi\}
    \quad\text{and}\quad
    \xi_{\ell-1}:=\xi/(\ell+1).
  \end{equation}
  Finally, let $c_1$ and $c_{\ell-1}$ be given by ($\cP_1$) and
  ($\cP_{\ell-1}$), respectively, for the previously specified parameters together with $\xi_1$
  and $\xi_{\ell-1}$. Set $c:=\max\{c_1,c_{\ell-1}\}$.
  We will prove that with this choice of $\eps$ and $c$ the statement in
  ($\cP_\ell$) holds for the input parameters $\eps'$, $\mu$, and $\xi$.
  For this purpose let $\Gamma=\Gnp$ be a random graph. By ($\cP_1$) and
  ($\cP_{\ell-1}$) and the choice of the parameters the graph $\Gamma$ has
  properties $P_1(\eps'_1,\mu_1,\eps_1,\xi_1,p(n))$ and
  $P_{\ell-1}(\eps'_{\ell-1},\mu_{\ell-1},\eps_{\ell-1},\xi_{\ell-1},p(n))$ with
  probability $1-\smallo(1)$. We will show that a graph $\Gamma$ with these
  properties also satisfies $P_\ell(\eps',\mu,\eps,\xi,p(n))$. Let $G=(X\dcup
  Y\dcup Z,E)$ be an arbitrary subgraph of such a $\Gamma$ where
  $X=X_1\dcup\dots\dcup X_\ell$, $|X|=n_1$, $|Y|=n_2$, $|Z|=n_3$, 
  with $n_1,n_3\ge\xi p^{\Delta-1}n$, $n_2\ge\xi p^{\Delta-\ell-1}n$,
  and $|X_i|\ge\lfloor\frac{n_1}{\ell}\rfloor$, and assume that $(X,Y)$ and
  $(Y,Z)$ are $(\eps,d,p)$-dense pairs for $d\ge d_0$.
  
  We would like to bound
  $\hyper{B}_\ell:=\bad{\ell}{\eps'}{d}(X_1,\dots,X_{\ell};Y,Z)$. For this
  purpose let $B'$ be a fixed $(\ell-1)$-set and define
  \begin{subequations}
  \begin{gather}
    \hyper{B}_{\ell-1}:=
      \bad{\ell-1}{\eps'_{\ell-1}}{d}(X_1,\dots,X_{\ell-1};Y,Z) \cup
      \bad{\ell-1}{\eps'_{\ell-1}}{d}(X_1,\dots,X_{\ell-1};Y,X_\ell) 
    \label{eq:bad:def:Bell-1}
    \\ 
    \hyper{B}_{1}(B'):=
      \bad{1}{\eps'}{d}(X_\ell;\coN_Y(B'),Z).
  \label{eq:bad:def:B1}
  \end{gather}
  \end{subequations}
  For an $\ell$-set $B\in X_1\times\dots\times X_{\ell}$ let further 
  $\Pi_{\ell-1}(B)$ denote the $(\ell-1)$-set that is the projection of $B$
  to $X_1\times\dots\times X_{\ell-1}$ and let $\Pi_{\ell}(B)$ be the 
  vertex that is the projection of $B$ to $X_\ell$.
  Now, consider an $\ell$-set $B$ that is contained in
  $B\in\hyper{B}_\ell$ but is
  such that $B':=\Pi_{\ell-1}(B)\not\in\hyper{B}_{\ell-1}$. Let
  $v=\Pi_\ell(B)\in X_\ell$ and $Y':=\coN_Y(B')$. We will show that then
  $v\in\hyper{B}_{1}(B')$. Indeed, since $B'\not\in\hyper{B}_{\ell-1}$ it
  follows from~\eqref{eq:bad:def:Bell-1} that
  \begin{equation*}
    B'\not\in \bad{\ell-1}{\eps'_{\ell-1}}{d}(X_1,\dots,X_{\ell-1};Y,Z)
  \end{equation*}
  and thus $|Y'|\ge(d-\eps'_{\ell-1})^{\ell-1}p^{\ell-1}n_2$. 
  As $\coN_Y(B)=\coN_{Y'}(v)$ we conclude that
  $$v\in\smash{\bad{1}{\eps'}{d}(X_\ell;Y',Z)}=\hyper{B}_{1}(B')$$
  by~\eqref{eq:bad:def:B1} because otherwise $(\coN_Y(B),Z)$ was
  $(\eps',d,p)$-dense and we had
  \begin{equation*}
    |\coN_Y(B)|
    \ge (d-\eps')p|Y'|
    \ge (d-\eps')p\cdot(d-\eps'_{\ell-1})^{\ell-1}p^{\ell-1}n_2
    \ge (d-\eps')^\ell p^\ell n_2,
  \end{equation*}
  which contradicts $B\in\hyper{B}_\ell$.
  Summarizing, we have
  \begin{equation}
  \label{eq:bad:Bell}
  \begin{split}
    \hyper{B}_\ell &=
      \{B\in\hyper{B}_\ell\colon\Pi_{\ell-1}(B)\in\hyper{B}_{\ell-1}\}
      \cup
      \{B\in\hyper{B}_\ell\colon\Pi_{\ell-1}(B)\not\in\hyper{B}_{\ell-1}\}
    \\
    &\subset
     \left(\hyper{B}_{\ell-1}\times X_\ell\right) \cup 
     \bigcup_{B'\not\in\hyper{B}_{\ell-1}}
     \{B'\}\times\hyper{B}_1(B').  
  \end{split}
  \end{equation}  
  For bounding $\hyper{B}_\ell$ we will thus
  estimate the sizes of $\hyper{B}_{\ell-1}$ and $\hyper{B}_1(B')$ for
  $B'\not\in\hyper{B}_{\ell-1}$. Let $X':=X_1\dcup\dots\dcup X_{\ell-1}$.
  Since $(X,Y)$ is $(\eps,d,p)$-dense we conclude from
  Proposition~\ref{prop:subpairs} that $(X',Y)$ and $(X_\ell,Y)$ are
  $(\eps_{\ell-1},d,p)$-dense pairs since $\eps(\ell+1)\le\eps_{\ell-1}$.
  Further, by the choice of $\xi_{\ell-1}$ we get $|X'|,|X_\ell|\ge
  n_1/(\ell+1)\ge\xi_{\ell-1} p^{\Delta-1}n$ since $n_1\ge\xi p^{\Delta-1}n$
  by assumption. Thus we can use the fact that $\Gamma$ has property
  $P_{\ell-1}(\eps'_{\ell-1},\mu_{\ell-1},\eps_{\ell-1},\xi_{\ell-1},p(n))$
  once on the tripartite subgraph induced on $X'\dcup Y\dcup Z$
  in $G$ and once on the tripartite subgraph induced on $X'\dcup Y\dcup X_\ell$
  in $G$ and infer that
  \begin{equation}
  \label{eq:bad:Bell-1}
    \left|\hyper{B}_{\ell-1}\right|\le 2\cdot\mu_{\ell-1}n_1^{\ell-2}
    =\frac{\mu}{2}n_1^{\ell-2}.
  \end{equation}
  For estimating $|\hyper{B}_1(B')|$ for $B'\not\in\hyper{B}_{\ell-1}$ let
  $Y':=\coN_Y(B')$. Observe that this implies that $(Y',Z)$ and $(X_\ell,Y')$
  are $(\eps_1,d,p)$-dense pairs because
  $\eps'_{\ell-1}\le\eps_1$, and that
  \begin{equation*}
  	|Y'| \ge (d-\eps'_{\ell-1})^{\ell-1}p^{\ell-1}n_2
  	  \ge (d-\eps'_{\ell-1})^{\ell-1} p^{\ell-1}\cdot \xi p^{\Delta-\ell-1}n
  	  \geByRef{eq:bad:xi} \xi_1 p^{\Delta-1}n.
  \end{equation*}
  By~\eqref{eq:bad:xi} $|X_\ell|,|Z|\ge\xi p^{\Delta-1} n/(\ell+1)\ge
  \xi_1 p^{\Delta-1} n$.
  As $\Gamma$ satisfies $P_1(\eps'_1,\mu_1,\eps_1,\xi_1,p(n))$ we conclude that
  \begin{equation}
  \label{eq:bad:B1}
    \left|\hyper{B}_1(B')\right|
    \eqByRef{eq:bad:def:B1}
    |\bad{1}{\eps'}{d}(X_\ell;Y',Z)|\le\mu_1 n_1\le \frac{\mu}{2}n_1.
  \end{equation}
  In view of~\eqref{eq:bad:Bell}, combining~\eqref{eq:bad:Bell-1}
  and~\eqref{eq:bad:B1} gives
  \begin{equation*}
    \left|\bad{\ell}{\eps'}{d}(X_1,\dots,X_{\ell};Y,Z)\right|=
    \left|\hyper{B}_{\ell}\right|\le
    \frac{\mu}{2}n_1^{\ell-1} \cdot n_1 + n_1^{\ell-1}\cdot\frac{\mu}{2}
      n_1=\mu n_1^\ell.
  \end{equation*}
  Because $G$ was arbitrary this shows that $\Gamma$ has property
  $P_\ell(\eps',\mu,\eps,\xi,p(n))$. Thus ($\cP_\ell$) holds which finishes the
  proof of the inductive step.
\end{proof}

In the proof of Lemma~\ref{lem:Bad} we now first partition the vertex set~$X$,
in which we count bad $\ell$-sets, arbitrarily into~$T$ vertex sets of equal
size. Lemma~\ref{lem:bad} then implies that for all $\ell'\in[\ell]$ there are not
many bad $\ell'$-sets that are crossing in this partition. It follows that only
few $\ell$-sets in~$X$ contain a bad $\ell'$-set for some $\ell'\in[\ell]$ 
(recall that in Definition~\ref{def:bad} for
$\Bad{\ell}{\eps}{d}(X,Y,Z)$  such
$\ell'$-sets are considered). Moreover, if~$T$ is sufficiently large then the
number of non-crossing $\ell$-sets is negligible. Hence we obtain that there
are few bad sets in total.

\begin{proof}[Proof of Lemma~\ref{lem:Bad}]
  Given $\Delta,\ell,d_0,\eps'$ and $\mu$ let
  $T$ be such that $\mu T\ge 2$, fix
  $\mu_{\subref{lem:bad}}:=\frac12\mu/(\ell T^\ell)$. For $j\in[\ell]$ let
  $\eps_{j}$ be given by Lemma~\ref{lem:bad} with $\ell$ replaced by $j$ and
  for $\Delta$, $d_0$, $\eps'$, and $\mu_{\subref{lem:bad}}$ and set
  $\eps_{\subref{lem:bad}}:=\min_{j\in[\ell]}\eps_{j}$. Define
  $\eps:=\eps_{\subref{lem:bad}}/(T+1)$. Now, in Lemma~\ref{lem:Bad}
  let $\xi$ be given by the adversary for this $\eps$.
  Set $\xi_{\subref{lem:bad}}:=\xi/(T+1)$, and let $c$ be given by
  Lemma~\ref{lem:bad} for this $\xi_{\subref{lem:bad}}$.

  
  Let $\Gamma=\Gnp$ with $p\ge c(\frac{\log n}{n})^{1/\Delta}$. Then \aas\
  the graph $\Gamma$ satisfies the statement in Lemma~\ref{lem:bad}
  for parameters $j\in[\ell]$, $\Delta$, $d_0$, $\eps'$,
  $\mu_{\subref{lem:bad}}$, and $\xi_{\subref{lem:bad}}$.
  Assume that $\Gamma$ has this property for all $j\in[\ell]$.
  We will show that it then also satisfies the statement in
  Lemma~\ref{lem:Bad}.
  
  Indeed, let $G$ and $X$, $Y$, $Z$ be arbitrary with the
  properties as required in Lemma~\ref{lem:Bad}. Let $X=X_1\dcup\dots\dcup
  X_T$ be an arbitrary partition of $X$ with $|X_i|\ge\lfloor\frac{n_1}{T}\rfloor$.
  We will first show that there are not many bad crossing $\ell$-sets with
  respect to this partition, i.e., we will bound the size of
  $\Bad{\ell}{\eps'}{d}(X_1,\dots,X_T;Y,Z)$. 
  By definition
  \begin{equation*}
    \big|\Bad{\ell}{\eps'}{d}(X_1,\dots,X_T;Y,Z)\big|\le
    \sum_{j\in[\ell]} \big|\bad{j}{\eps'}{d}(X_1,\dots,X_T;Y,Z)\big|
       \cdot n_1^{\ell-j}.
  \end{equation*}

  Now, fix $j\in[\ell]$ and an index set
  $\{i_1,\dots,i_{j}\}\in\binom{[T]}{j}$ and consider the induced
  tripartite subgraph $G'=(X' \dcup Y \dcup Z,E')$ of $G$ with
  $X'=X_{i_1}\dcup\dots\dcup X_{i_j}$. Observe that
  $|Y|\ge\xi_{\subref{lem:bad}}p^{\Delta-j-1}n$,
  $|Z|\ge\xi_{\subref{lem:bad}}p^{\Delta-1}n$, and $n'_1:=|X'|\ge j\lfloor
  n_1/T\rfloor\ge\xi_{\subref{lem:bad}}p^{\Delta-1}n$. 
  By definition $\eps (T+1)/j\le\eps_{\subref{lem:bad}} \le\eps_{j}$ 
  and so by Proposition~\ref{prop:subpairs} the pair $(X',Y)$ is
  $(\eps_{j},d,p)$-dense. 
  Thus, because $\Gamma$ satisfies the statement in Lemma~\ref{lem:bad}
  for parameters $j$, $\Delta$, $d$, $\eps'$,
  $\mu_{\subref{lem:bad}}$, and $\xi_{\subref{lem:bad}}$
  we have that $G'$ satisfies
  \begin{equation*}
    \big|\bad[G']{j}{\eps'}{d}(X_{i_1},\dots,X_{i_{j}};Y,Z)\big|  
    \le \mu_{\subref{lem:bad}} (n'_1)^{j}.
  \end{equation*}
  As there are $\binom{T}{j}$ choices for the index set
  $\{i_1,\dots,i_{j}\}$ this implies
  \begin{equation*}
    \big|\bad[G]{j}{\eps'}{d}(X_1,\dots,X_T;Y,Z)\big|  
    \le \binom{T}{j} \mu_{\subref{lem:bad}} (n'_1)^{j}
    \le T^{j} \mu_{\subref{lem:bad}} n_1^{j},
  \end{equation*}
  and thus
  \begin{equation*}
    \big|\Bad{\ell}{\eps'}{d}(X_1,\dots,X_T;Y,Z)\big|
    \le \sum_{j\in[\ell]}T^{j} \mu_{\subref{lem:bad}} n_1^{j}
      \cdot n_1^{\ell-j}
    \le\frac12\mu n_1^\ell.
  \end{equation*}
  The number of $\ell$-sets in $X$ that are not crossing with respect to the
  partition $X=X_1\dcup\dots\dcup X_T$ is at most
  $T\binom{n_1/T}{2}\binom{n_1}{\ell-2}\le \frac1T
  n_1^\ell\le\frac12\mu n_1^\ell$ and so we get
  $|\Bad{\ell}{\eps'}{d}(X,Y,Z)|\le\mu n_1^\ell$.
\end{proof}

}

\end{document}